\documentclass[11pt]{article}
\usepackage[a4paper]{geometry}
\usepackage{aeguill}                 
\usepackage{graphicx}
\usepackage{amsmath,amsfonts,amssymb,amsthm}
\usepackage{pstricks} 
\usepackage{epstopdf}
\usepackage{srcltx}
\usepackage[utf8]{inputenc} 
\usepackage{hyperref} 
\usepackage[english]{babel} 
\usepackage{comment}

\title{Measure equivalence classification of transvection-free right-angled Artin groups}
\author{Camille Horbez and Jingyin Huang}

\usepackage[all]{xy}
\usepackage{bbm}
\begin{document}
\maketitle 
\newtheorem{de}{Definition} [section]
\newtheorem{theo}[de]{Theorem} 
\newtheorem{prop}[de]{Proposition}
\newtheorem{lemma}[de]{Lemma}
\newtheorem{cor}[de]{Corollary}
\newtheorem{propd}[de]{Proposition-Definition}
\newtheorem{conj}[de]{Conjecture}
\newtheorem{claim}{Claim}
\newtheorem*{claim2}{Claim}

\newtheorem{theointro}{Theorem}
\newtheorem*{defintro}{Definition}
\newtheorem{corintro}[theointro]{Corollary}
\newtheorem{questionintro}[theointro]{Question}
\newtheorem{propintro}[theointro]{Proposition}

\theoremstyle{remark}
\newtheorem{rk}[de]{Remark}
\newtheorem{ex}[de]{Example}
\newtheorem{question}[de]{Question}

\normalsize

\newcommand{\Aut}{\mathrm{Aut}}
\newcommand{\Out}{\mathrm{Out}}
\newcommand{\Inn}{\mathrm{Inn}}
\newcommand{\stab}{\operatorname{Stab}}
\newcommand{\pstab}{\operatorname{Pstab}}
\newcommand{\dunion}{\sqcup}
\newcommand{\eps}{\varepsilon}
\renewcommand{\epsilon}{\varepsilon}
\newcommand{\calf}{\mathcal{F}}
\newcommand{\cali}{\mathcal{I}}
\newcommand{\caly}{\mathcal{Y}}
\newcommand{\calx}{\mathcal{X}}
\newcommand{\calz}{\mathcal{Z}}
\newcommand{\calo}{\mathcal{O}}
\newcommand{\calb}{\mathcal{B}}
\newcommand{\calq}{\mathcal{Q}}
\newcommand{\calu}{\mathcal{U}}
\newcommand{\call}{\mathcal{L}}
\newcommand{\bbR}{\mathbb{R}}
\newcommand{\bbZ}{\mathbb{Z}}
\newcommand{\bbD}{\mathbb{D}}
\newcommand{\NT}{\mathrm{NT}}
\newcommand{\cat}{\mathrm{CAT}(-1)}
\newcommand{\CAT}{\mathrm{CAT}(0)}
\newcommand{\actson}{\curvearrowright}
\newcommand{\caln}{\mathcal{N}}
\newcommand{\calg}{\mathcal{G}}
\newcommand{\Prob}{\mathrm{Prob}}
\newcommand{\calt}{\mathcal{T}}
\newcommand{\calc}{\mathcal{C}}
\newcommand{\adm}{\mathrm{adm}}
\newcommand{\cala}{\mathcal{A}}
\newcommand{\cals}{\mathcal{S}}
\newcommand{\calh}{\mathcal{H}}
\newcommand{\Stab}{\mathrm{Stab}}
\newcommand{\bdd}{\mathrm{bdd}}
\newcommand{\calp}{\mathcal{P}}
\newcommand{\Fix}{\mathrm{Fix}}
\newcommand{\todo}{\textbf{TO DO:~}}
\newcommand{\cald}{\mathcal{D}}
\newcommand{\rad}{r}
\newcommand{\good}{\mathrm{good}}
\newcommand{\Faces}{\mathrm{Faces}}
\newcommand{\del}{\mathbb{D}_\Gamma}
\newcommand{\Mod}{\mathrm{Mod}}
\newcommand{\Comm}{\mathrm{Comm}}
\newcommand{\si}{\sigma}
\newcommand{\st}{\mathrm{st}}
\newcommand{\lk}{\mathrm{lk}}
\newcommand{\supp}{\mathrm{supp}}
\newcommand{\HHS}{\mathrm{HHS}}
\newcommand{\reg}{\mathrm{reg}}
\newcommand{\calk}{\mathcal{K}}

\makeatletter
\edef\@tempa#1#2{\def#1{\mathaccent\string"\noexpand\accentclass@#2 }}
\@tempa\rond{017}
\makeatother

\begin{abstract}
We prove that if two transvection-free right-angled Artin groups are measure equivalent, then they have isomorphic extension graphs. As a consequence, two right-angled Artin groups with finite outer automorphism groups are measure equivalent if and only if they are isomorphic. This matches the quasi-isometry classification.

However, in contrast with the quasi-isometry question, we observe that no right-angled Artin group is superrigid for measure equivalence in the strongest possible sense, for two reasons. First, a right-angled Artin group $G$ is always measure equivalent to any graph product of infinite countable amenable groups over the same defining graph. Second, when $G$ is nonabelian, the automorphism group of the universal cover of the Salvetti complex of $G$ always contains infinitely generated (non-uniform) lattices.
\end{abstract}

\section*{Introduction}

Measured group theory studies groups through their ergodic actions on standard probability spaces.  
A central quest is to classify groups up to \emph{measure equivalence}, a notion introduced by Gromov \cite{Gro} as a measure-theoretic analogue of that of quasi-isometry between finitely generated groups.

The definition is as follows: two countable groups $G_1$ and $G_2$ are \emph{measure equivalent} if there exists a standard measure space $\Sigma$ equipped with a measure-preserving action of $G_1\times G_2$ by Borel automorphisms, such that for every $i\in\{1,2\}$, the $G_i$-action on $\Sigma$ is essentially free and has a finite measure fundamental domain. A typical example is that two lattices in the same locally compact second countable group are always measure equivalent.

A first striking result in the theory, due to Ornstein and Weiss \cite{OW} (building on previous work of Dye \cite{Dye1,Dye2}) is that any two countably infinite amenable groups are measure equivalent. At the other extreme, strong rigidity results have been established from the viewpoint of measure equivalence, notably for lattices in higher rank Lie groups (Furman \cite{Fur}, building on earlier work of Zimmer \cite{Zim1,Zim2}) and for mapping class groups of finite-type surfaces (Kida \cite{Kid}). In \cite{MS}, Monod and Shalom used bounded cohomology techniques to obtain rigidity results for products of certain negatively curved groups. Other developments include superrigidity results for certain amalgamated free products \cite{Kid2}, for other groups related to mapping class groups \cite{CK}, or classification results within certain classes of groups, like Baumslag--Solitar groups \cite{kida2014invariants,HR}. See also \cite{Sha-survey,Gab-survey,Fur-survey} for general surveys in measured group theory.

In previous work \cite{HH}, we obtained measure equivalence superrigidity results for certain classes of two-dimensional Artin groups of hyperbolic type. In the present paper, we focus on right-angled Artin groups, which have a very simple definition and played a prominent role in geometric group theory recently.
We obtain a complete measure equivalence classification of right-angled Artin groups with finite outer automorphism group.  One motivation for studying these groups from the viewpoint of measured group theory is that they usually tend to be significantly less rigid than certain other classes of Artin groups, as will be evidenced in Corollary~\ref{corintro:graph-products} and Theorem~\ref{theointro-non-superrigid} below. With this in mind, one hopes to explore the subtle line between rigidity and flexibility phenomena in this setting. 

\paragraph*{Main classification theorem.} We recall that given a finite simple graph $\Gamma$ (i.e.\ with no edge loops and no multiple edges between vertices), the right-angled Artin group $G_\Gamma$ is the group defined by the following presentation: the generators are the vertices of $\Gamma$, and two generators commute if and only if the corresponding vertices of $\Gamma$ are joined by an edge. Our main theorem is the following.

\begin{theointro}\label{theointro:1}
Let $G_1$ and $G_2$ be two right-angled Artin groups such that $\Out(G_1)$ and $\Out(G_2)$ are finite. Then $G_1$ and $G_2$ are measure equivalent if and only if they are isomorphic. 
\end{theointro}

Finiteness of $\Out(G_\Gamma)$ can easily be checked on the defining graph $\Gamma$; by work of Laurence \cite{laurence1995generating} and Servatius \cite{Ser}, this amounts to the two conditions below. For a vertex $v\in\Gamma$, the star of $v$, denoted by $\st(v)$, is the subgraph made of all vertices which are either equal or adjacent to $v$, together with edges of $\Gamma$ between these vertices.
\begin{enumerate}
\item $\Gamma$ does not contain two distinct vertices $v$ and $w$ such that the link of $v$ is contained in the star of $w$ -- this prevents the existence of transvections $v\mapsto vw$ in $\Aut(G_\Gamma)$ (we then say that $G_\Gamma$ is \emph{transvection-free});
\item $\Gamma$ contains no separating star, i.e. there does not exist any vertex $v\in\Gamma$ such that $\Gamma\setminus \st(v)$ is disconnected. This prevents the existence of partial conjugations.
\end{enumerate}

Every right-angled Artin group has an associated \emph{extension graph}, introduced by Kim and Koberda in \cite{kim2013embedability} as the graph whose vertices are the rank one parabolic subgroups of $G_\Gamma$ (i.e.\ conjugates of the cyclic subgroups generated by the standard generators of $G_\Gamma$), two of these being joined by an edge exactly when they commute. This graph can be viewed as an analogue of the curve graph in the context of right-angled Artin groups \cite{kim2014geometry}. Theorem~\ref{theointro:1} is in fact obtained as a consequence of the following theorem, using the additional fact that two right-angled Artin groups with finite outer automorphism groups have isomorphic extension graphs if and only if they are isomorphic \cite{Hua}. 

\begin{theointro}\label{theointro:main}
Let $G_1$ and $G_2$ be two transvection-free right-angled Artin groups. If $G_1$ and $G_2$ are measure equivalent, then they have isomorphic extension graphs.
\end{theointro}

In fact, in Theorem~\ref{theointro:1}, if we only assume that $\Out(G_1)$ is finite and that $G_2$ is transvection-free, then using \cite[Theorem~1.2]{Hua}, we deduce that $G_1$ and $G_2$ are measure equivalent if and only if they have isomorphic extension graphs, and in this case this happens if and only if $G_2$ is a finite-index subgroup of $G_1$. 

\medskip

 To our knowledge, right-angled Artin groups have not been systematically studied from the point of view of measured group theory before; however, a few things were already known concerning their measure equivalence classification. As already mentioned, all countably infinite free abelian groups (and more generally amenable groups) are measure equivalent \cite{Dye2,OW} (and not measure equivalent to nonamenable groups, see e.g.\ \cite[Corollary~3.2]{Fur-survey}). The behavior of measure equivalence under free products was thoroughly studied by Alvarez and Gaboriau in \cite{AG}. The invariance of the canonical decomposition of a right-angled Artin group as a direct product (corresponding to the maximal decomposition of $\Gamma$ as a join) can be proved by combining work of Monod and Shalom \cite[Theorems~1.16 and~1.17]{MS} together with a theorem of Chatterji, Fern\'os and Iozzi \cite[Corollary~1.8]{CFI} stating that  non-cyclic right-angled Artin groups that do not split as direct products belong to the class $\mathcal{C}_{\mathrm{reg}}$. Also, Gaboriau proved that $\ell^2$-Betti numbers can be used to obtain measure equivalence invariants \cite{Gab-betti}; they have been computed by Davis and Leary for many Artin groups including all the right-angled ones \cite{DL}. Finally, we comment that Theorem~\ref{theointro:1} was proved in our previous work on Artin groups under the extra assumption that the underlying graphs of $G_1$ and $G_2$ have girth at least $5$ \cite[Corollary~10.3]{HH}. However, our proof in \cite{HH} relies on fact that, under this extra assumption, $G_1$ and $G_2$ are special instances of $2$-dimensional Artin groups of hyperbolic type. Our methods in \cite{HH} are specific to this class of Artin groups, and therefore cannot be used to prove Theorem~\ref{theointro:1} in general.  

\paragraph*{Orbit equivalence and $W^*$-equivalence.} Our results can be reinterpreted in the language of orbit equivalence: indeed, by a result of Furman \cite{Fur2} (with an additional argument by Gaboriau \cite[Theorem~2.3]{Gab2}), two groups are measure equivalent if and only if they admit stably orbit equivalent essentially free measure-preserving ergodic actions on standard probability spaces. As will be explained with more details later in this introduction, the way we prove Theorem~\ref{theointro:main} is actually through this viewpoint -- see Theorem~\ref{theo:main-2} for a precise statement phrased in terms of measured groupoids, also including the case of non-free actions.

As right-angled Artin groups are cubical, our results can in fact also be phrased using the weaker notion of (stable) $W^*$-equivalence (we refer to \cite[Section~6.2]{PV} for a detailed discussion of the notion of stable $W^*$-equivalence). 

\begin{corintro}\label{corintro:von-neumann}
Let $G_1$ and $G_2$ be two right-angled Artin groups with finite outer automorphism groups. Assume that $G_1$ and $G_2$ have essentially free ergodic probability measure-preserving actions on standard probability spaces which are stably $W^*$-equivalent (i.e.\ their associated von Neumann algebras are isomorphic, or more generally one is isomorphic to an amplification of the other). 

Then $G_1$ and $G_2$ are isomorphic.
\end{corintro}

The version for $W^*$-equivalence (i.e.\ assuming that the two von Neumann algebras are isomorphic, without passing to amplifications) can be seen directly as a consequence of \cite[Corollary~4.2]{HHL}; we will explain in Section~\ref{sec:von-neumann} how the full statement follows from work of Popa and Vaes \cite{PV1}.
\\
\\
\textit{Remark.} In Corollary~\ref{corintro:von-neumann}, one cannot expect to reach the stronger conclusion that the given actions of $G_1$ and $G_2$ are conjugate. See indeed Remark~\ref{rk:oe-non-conj} for a construction of free, ergodic, probability measure-preserving actions that are orbit equivalent but not conjugate.

\paragraph*{Comparison with quasi-isometry, and failure of superrigidity.} Incidently, the classification results given by Theorems~\ref{theointro:1} and~\ref{theointro:main} both match the quasi-isometry classification obtained by the second named author in \cite{Hua}, generalizing earlier work of Bestvina, Kleiner and Sageev \cite{BKS}. A lot of work has revolved around the quasi-isometry classification problem for right-angled Artin groups, starting from work of Behrstock and Neumann \cite{behrstock2008quasi} for right-angled Artin groups whose defining graph is a tree, with a higher dimensional generalization in \cite{behrstock2010quasi}, and followed more recently by \cite{Hua4,Mar}, for instance.

However, apart from the analogy with mapping class groups and Artin groups suggesting that having a finite outer automorphism group is often the source of further rigidity phenomena, there was \emph{a priori} no reason to expect that the measure equivalence and quasi-isometry classification should match. By contrast, they already fail to match on the most basic class of right-angled Artin groups (with infinite outer automorphism group), namely the free abelian groups.

In fact, this fundamental difference on the most basic examples is the source of an extremely different situation between the quasi-isometry and measure equivalence classifications of right-angled Artin groups. This is evidenced by our next proposition (and its corollary), that we prove by adapting an argument of Gaboriau \cite{Gab-cost,Gab}, who dealt with the case of free products. 

We now recall the notion of graph products introduced by Green in \cite{green1990graph}. Given a finite simple graph $\Gamma$ and an assignment of a group $G_v$ to every vertex $v$ of $\Gamma$, the \emph{graph product} over $\Gamma$ with vertex groups $\{G_v\}_{v\in V\Gamma}$ is the group obtained from the free product of the groups $G_v$ by adding as only extra relations that every element of $G_v$ commutes with every element of $G_w$ whenever $v$ and $w$ are adjacent in $\Gamma$.

\begin{propintro}\label{propintro:graph-products}
Let $G$ and $H$ be two groups that split as graph products over the same finite simple graph $\Gamma$, with countable vertex groups $\{G_v\}_{v\in V\Gamma}$ and $\{H_v\}_{v\in V\Gamma}$, respectively. Assume that for every vertex $v\in V\Gamma$, the groups $G_v$ and $H_v$ admit orbit equivalent free measure-preserving actions on standard probability spaces.

Then $G$ and $H$ admit orbit equivalent free measure-preserving actions on standard probability spaces; in particular they are measure equivalent.
\end{propintro} 

We mention that orbit equivalence cannot be replaced by measure equivalence in the above statement, even for free products (see \cite[Section~2.2]{Gab}). Since right-angled Artin groups are graph products where the vertex groups are isomorphic to $\mathbb{Z}$, combining the above proposition with the aforementioned theorem of Ornstein and Weiss \cite{OW} yields the following.

\begin{corintro}\label{corintro:graph-products}
For every right-angled Artin group $G$, every group obtained as a graph product over the defining graph of $G$ with countably infinite amenable vertex groups, is measure equivalent to $G$. 
\end{corintro}

By choosing the vertex groups to be isomorphic to $\mathbb{Z}^n$, this gives infinitely many right-angled Artin groups that are all measure equivalent to $G$, but pairwise non-quasi-isometric. As pointed out by a referee, Corollary~\ref{corintro:graph-products} also yields a continuum of countable groups that are measure equivalent to a given right-angled Artin group (see Theorem~\ref{theo:non-rigidity}).

Let us also mention that conversely, there are examples of right-angled Artin groups which are quasi-isometric but not measure equivalent. For instance, the groups $G_n=(F_3\times F_3)\ast F_n$ are all quasi-isometric \cite[Theorem~1.5]{Why} but pairwise not measure equivalent by comparison of their $\ell^2$-Betti numbers \cite[Corollaire~0.3]{Gab-betti}.

Corollary~\ref{corintro:graph-products} also shows that right-angled Artin groups are never superrigid for measure equivalence in the strongest possible sense. This is in strong contrast to the situation for some non-right-angled Artin groups with finite outer automorphism groups \cite{HH}, or mapping class groups \cite{Kid}; this also strongly contrasts with the quasi-isometry superrigidity results obtained by the second named author in \cite{Hua2}.

In fact, there is another source of failure of measure equivalence superrigidity of right-angled Artin groups, given by the following theorem.

\begin{theointro}\label{theointro-non-superrigid}
For every nonabelian right-angled Artin group $G$, the automorphism group of the universal cover of the Salvetti complex of $G$ contains an infinitely generated (non-uniform) lattice $H$ (in particular $H$ is measure equivalent to $G$, but not commensurable to $G$). 
\end{theointro}

These lattices are constructed as follows. First we take an isometrically embedded tree in the universal cover and a non-uniform lattice $H$ acting on this tree (see e.g.\ \cite{bass2001tree}). Then we extend this action to a larger group acting on the ambient space. This method only produces non-uniform lattices which are not finitely generated, and we do not know whether there are finitely generated examples.

\paragraph*{A word on the proof of the classification theorem.}
In the remainder of this introduction, we would like to explain how we prove our main classification theorem, in the form of Theorem~\ref{theointro:main} (from which Theorem~\ref{theointro:1} follows). The general structure of our proof is inspired by Kida's strategy in the mapping class group setting \cite{Kid}. The orbit equivalence interpretation reduces our proof to a problem about measured groupoids associated to measure-preserving actions of right-angled Artin groups on standard probability spaces. Given two such actions of $G_1$ and $G_2$ on the same standard probability space $Y$ with the same orbit structure -- in other words, a measured groupoid  over $Y$ coming with two cocycles, one towards $G_1$ and the other towards $G_2$, the goal is to build a canonical map that associates to every point of $Y$, an isomorphism between $\Gamma^e_1$ and $\Gamma_2^e$, where $\Gamma^e_i$ denotes the extension graph of $G_i$. To this end, given a transvection-free right-angled Artin group $G$, we characterize subgroupoids of $G\ltimes Y$ (the groupoid coming from the $G$-action on $Y$ -- or in fact, more generally, a restriction of this groupoid) that naturally correspond to stabilizers of vertices of $\Gamma^e$, or pairs of stabilizers of adjacent vertices of $\Gamma^e$, in a purely groupoid-theoretic way.

In order to keep this introduction not too technical, we will not provide any groupoid-theoretic statement, however we will describe some of the ideas behind their group-theoretic analogues. Namely, we will explain how to obtain an algebraic characterization of stabilizers of vertices of $\Gamma^e$ in $G_\Gamma$ (see Proposition~\ref{prop:raag} for the version for groupoids), and an algebraic characterization of adjacency of vertices.

So let $v$ be a vertex in $\Gamma^e$ -- corresponding to a cyclic parabolic subgroup $Z_v$ of $G_\Gamma$. The stabilizer of $v$ for the $G_\Gamma$-action on $\Gamma^e$ is the centralizer of $Z_v$. Classical work of Servatius \cite{Ser} ensures that this centralizer splits as a direct product $Z_v\times Z_v^{\perp}$, where $Z_v^{\perp}$ is a parabolic subgroup of $G_\Gamma$. More precisely, if $Z_v$ is conjugate to the standard cyclic subgroup associated to a vertex $\bar v$ of $\Gamma$, then $Z_v^{\perp}$ is conjugate to the standard parabolic subgroup associated to the link of $\bar v$ in $\Gamma$. The transvection-freeness assumption ensures that $Z_v^{\perp}$ is nonabelian. In fact the following holds. 
\begin{itemize}
\item[($*$)] The centralizer $C_{G_\Gamma}(Z_v)$ is a nonamenable subgroup of $G_\Gamma$ that contains an infinite normal amenable subgroup. Among the subgroups of $G_\Gamma$, it is maximal with respect to this property.
\end{itemize}
This is reminiscent of the characterization of centralizers of Dehn twists -- i.e.\ stabilizers of isotopy classes of essential simple closed curves -- for finite-type mapping class groups.  Property~$(\ast)$ is phrased on purpose in terms of \emph{amenability} and \emph{normality}, as these notions have groupoid-theoretic analogues -- in our proof, amenability is used to obtain invariant probability measures on the (compact) Roller boundary of the universal cover of the Salvetti complex, as will be explained below under the heading `Geometric tools'. 

However (contrary to the case of the curve graph of a surface), Property~$(*)$ is not enough to characterize vertex stabilizers of $\Gamma^e$: indeed $G_\Gamma$ could also contain a maximal parabolic subgroup splitting as a direct product of two nonabelian subgroups $P_1\times P_2$ with trivial center, and in this case a subgroup of the form $\langle g\rangle \times P_2$ with $g$ generic in $P_1$, would also satisfy Property~$(*)$. To distinguish these two types of subgroups, we make the following observation: when $H=C_{G_\Gamma}(Z_v)$, the centralizer $H'$ of any infinite amenable subgroup $Z'$ not commensurable to $Z_v$, will look very different from $H$, in the sense that one can always find a nonamenable subgroup of $H$ that intersects $H'$ trivially. On the other hand, when $H$ is of the form $\langle g\rangle \times P_2$ as above, by choosing $g'$ to be any other generic element of $P_1$, the centralizer $H'$ of $g'$ will be much closer to $H$, and in particular every nonamenable subgroup of $H'$ will have to intersect $P_2$, whence $H$, nontrivially.    See assertion~2.(b) from Proposition~\ref{prop:raag} for the groupoid-theoretic version.  

We also need to characterize adjacency in $\Gamma^e$. For that, the key observation is that two vertices $v,w\in \Gamma^e$ are adjacent if and only if there are only finitely many vertices of $\Gamma^e$ that are fixed by $C_{G_\Gamma}(Z_v)\cap C_{G_\Gamma}(Z_w)$.

\paragraph*{Geometric tools.} Our proof of Theorem~\ref{theointro:main} is an implementation of a groupoid version of the above strategy. This groupoid version involves a certain geometric setting (with features of negative curvature), in order to apply an argument introduced by Adams in \cite{Ada}. In the work of Kida \cite{Kid}, this came in the form of a natural partition of the (compact) Thurston boundary of the Teichmüller space into arational and nonarational laminations. In our previous work on Artin groups of hyperbolic type \cite{HH}, this came in the form of a partition of the horofunction compactification of a certain $\mathrm{CAT}(-1)$ simplicial complex on which the group is acting. In the present paper, we exploit the cubical geometry of right-angled Artin groups by considering the Roller boundary $\partial_R\widetilde{S}_\Gamma$ of the universal cover of the Salvetti complex, and its partition, introduced by Fern\'os in \cite{Fer}, into \emph{regular} points and \emph{non-regular} points.

The three important features we exploit are the following: every non-regular point has a natural associated parabolic subgroup (see Theorem~\ref{theointro:roller} below); there is a \emph{barycenter map} that canonically associates a vertex of $\widetilde{S}_\Gamma$ to every triple of pairwise distinct regular points \cite{FLM}; the action of $G_\Gamma$ on the set of regular points (in fact on the whole $\partial_R\widetilde{S}_\Gamma$) is Borel amenable \cite{NS,Duc}.

Let us explain how this geometric setting is used, by explaining why a subgroup $H\subseteq G_\Gamma$ satisfying Property~$(*)$ above (in particular, containing an infinite normal amenable subgroup $A$) preserves a parabolic subgroup. Amenability of $A$ is used to get an invariant probability measure $\nu$ on the \emph{compact} metrizable space $\partial_R\widetilde{S}_\Gamma$. Theorem~\ref{theointro:roller} below, saying that every non-regular point has a natural associated parabolic subgroup, is exploited to build a  canonical $A$-invariant (whence $H$-invariant) parabolic subgroup if $\nu$ gives positive measure to the set of non-regular points. If $\nu$ is supported on the set of regular points, and if its support has cardinality at least $3$, then the aforementioned \emph{barycenter map} is used to contradict the fact that $A$ is infinite. Finally, if $\nu$ is supported on at most two points, then the amenability of point stabilizers is exploited to get a contradiction to the fact that $H$ is nonamenable -- in the groupoid version, this is replaced by the Borel amenability of the action of $G_\Gamma$ on $\partial_R\widetilde{S}_\Gamma$.

Let us finally state the result we use concerning non-regular points (see Section~\ref{sec:background-roller} for relevant definitions). In this statement $\partial_R\widetilde{S}_\Gamma$ is equipped with its Borel $\sigma$-algebra.

\begin{theointro}\label{theointro:roller}
Let $\Gamma$ be a finite simple graph. Then we can assign to each non-regular point $\xi\in\partial_R\widetilde{S}_\Gamma$ a unique standard subcomplex $Y\subseteq\widetilde{S}_\Gamma$ such that for every combinatorial geodesic ray $r$ representing $\xi$, the subset $Y$ is the smallest standard subcomplex of $\widetilde{S}_\Gamma$ that contains a subray of $r$. This assignment is measurable and $G_\Gamma$-equivariant, and $Y$ is contained in a standard subcomplex of $\widetilde{S}_\Gamma$ which splits non-trivially as a product.  
\end{theointro}

\paragraph*{Open questions.} Our work raises several questions regarding the measure equivalence classification and rigidity of right-angled Artin groups.
\begin{enumerate}
\item Given a right-angled Artin group $G_\Gamma$ with finite outer automorphism group, what can be said of countable groups $H$ that are measure equivalent to $G_\Gamma$? For instance, do they necessarily act on a $\mathrm{CAT}(0)$ cube complex with amenable vertex stabilizers?
\item A sharper notion of \emph{$L^1$-measure equivalence} was introduced by Bader, Furman and Sauer in \cite{BFS}, by imposing an integrability condition on the measure equivalence cocycle. Are there right-angled Artin groups that are superrigid for $L^1$-measure equivalence? It turns out that free abelian groups of different ranks are not $L^1$-measure equivalent to one another (see \cite{Aus}), thus suggesting that the obstruction coming from Corollary~\ref{corintro:graph-products} vanishes under this sharper notion. The question may be subdivided into two parts. First, if a finitely generated group $H$ is $L^1$-measure equivalent to $G_\Gamma$, is it commensurable to a lattice in $\Aut(\widetilde{S}_\Gamma)$? Second, are all finitely generated such lattices cocompact? 
\item Classify right-angled Artin groups up to measure equivalence, beyond the class of those with finite outer automorphism group. For instance, Behrstock and Neumann proved in \cite{behrstock2008quasi} that all right-angled Artin groups whose defining graph is a tree of diameter at least $3$ are quasi-isometric. Are they all measure equivalent?    
\end{enumerate}

\paragraph*{Organization of the paper.}

The paper has four parts. The first part is mostly a background section on right-angled Artin groups. In the second part, we prove Theorem~\ref{theointro:roller} concerning the geometry of the Roller boundary of the universal cover of the Salvetti complex. The third part implements the groupoid version of the strategy explained in this introduction, and gives a proof of our main classification theorems (Theorems~\ref{theointro:1} and~\ref{theointro:main}). In Section~4, we present two sources of failure of measure equivalence superrigidity of right-angled Artin groups: graph products of countably infinite amenable groups, and non-uniform lattices in the automorphism group of the universal cover of the Salvetti complex. 

\paragraph*{Acknowledgments.} 
We thank the referees for their careful reading of our manuscript and their helpful suggestions. The first named author acknowledges support from the Agence Nationale de la Recherche under Grant ANR-16-CE40-0006 DAGGER.

\setcounter{tocdepth}{1}
\tableofcontents

\section{Right-angled Artin groups: background and complements}

In the present section, we review standard facts about right-angled Artin groups and establish a few statements that we will need in the sequel. More background on right-angled Artin groups can be found in \cite{charney2007introduction}. Section~\ref{sec:review} reviews basic facts about right-angled Artin groups, their parabolic subgroups, their automorphisms, and Salvetti complexes. Section~\ref{sec:rigidity} reviews earlier work of the second named author \cite{Hua} establishing a rigidity statement for extension graphs of right-angled Artin groups. Sections~\ref{sec:transvection-free} and~\ref{sec:full-support} establish two extra statements that are used in Section~\ref{sec:me}.

We will stick to the setting of right-angled Artin groups in the present paper, but would like to mention that many of the results presented here have also been generalized to broader contexts. As a specific example, several results regarding parabolic subgroups have been generalized to graph products in \cite[Section~3]{antolin2015tits}.

\subsection{Review of basic facts about right-angled Artin groups}\label{sec:review}

\paragraph*{Definition.} Given a finite simple graph $\Gamma$ with vertex set $V\Gamma$, the \emph{right-angled Artin group} with defining graph $\Gamma$, denoted by $G_\Gamma$, is given by the following presentation:
\begin{center}
	$\langle V\Gamma$\ |\ $[v,w]=1$ if $v$ and $w$ are joined by an edge$\rangle$.
\end{center}
In this way $V\Gamma$ is identified to a \textit{standard generating set} for $G_\Gamma$. Given two finite simple graphs $\Gamma$ and $\Lambda$, a theorem of Droms \cite{Dro} asserts that $G_\Gamma$ and $G_\Lambda$ are isomorphic if and only if $\Gamma$ and $\Lambda$ are isomorphic. 

\paragraph*{Parabolic subgroups.} Let $\Lambda\subseteq\Gamma$ be a \emph{full} subgraph, i.e.\ two vertices of $\Lambda$ are adjacent in $\Lambda$ if and only if they are adjacent in $\Gamma$. Then there is an injective homomorphism $G_{\Lambda}\hookrightarrow G_{\Gamma}$, whose image is called a \emph{standard subgroup} of $G_{\Gamma}$ with \emph{type} $\Lambda$. A conjugate of such a subgroup is called a \emph{parabolic subgroup of $G_\Gamma$ with type $\Lambda$}.  The type of a parabolic subgroup is well-defined: if $\Lambda_1,\Lambda_2\subseteq\Gamma$ are two full subgraphs, then $G_{\Lambda_1}$ and $G_{\Lambda_2}$ are not conjugate unless $\Lambda_1=\Lambda_2$, as follows from \cite[Proposition~2.2]{charney2007automorphisms}. Note that the definitions of parabolic subgroups and standard subgroups depend on the choice of a standard generating set for $G_\Gamma$.

The \emph{star} of a vertex $v$ in $\Gamma$, denoted by $\st(v)$, is the full subgraph spanned by $v$ and all the vertices that are adjacent to $v$. Its \emph{link} $\lk(v)$ is defined to be the full subgraph spanned by all the vertices that are adjacent to $v$. 

Given two graphs $\Gamma_1$ and $\Gamma_2$, we denote by $\Gamma_1\circ \Gamma_2$ the join of $\Gamma_1$ and $\Gamma_2$. Every finite simple graph $\Gamma$ has a canonical join decomposition $\Gamma=\Gamma_0\circ \Gamma_1\circ\cdots\circ\Gamma_k$, where $\Gamma_0$ is the maximal clique factor of $\Gamma$, and for every $i\in\{1,\dots,k\}$, the graph $\Gamma_i$ is \emph{irreducible}, i.e.\ it does not admit any nontrivial join decomposition. We call this the \emph{de Rham decomposition} of $\Gamma$. There is also an induced \emph{de Rham decomposition} of $G_\Gamma$, namely $G_\Gamma=\mathbb Z^n\times G_{\Gamma_1}\times\cdots\times G_{\Gamma_k}$.

For a full subgraph $\Lambda\subseteq \Gamma$, we define $\Lambda^{\perp}$ to be the full subgraph of $\Gamma$ spanned by the collection of all vertices of $\Gamma\setminus \Lambda$ which are adjacent to every vertex of $\Lambda$. For example, with this terminology, we have $\lk(v)=\{v\}^\perp$. The following was proved by Charney, Crisp and Vogtmann in \cite[Proposition~2.2]{charney2007automorphisms}, building on work of Godelle \cite{God}.

\begin{prop}\label{prop:normalizer}
	Let $\Gamma$ be a finite simple graph, and let $\Lambda\subseteq\Gamma$ be a full subgraph. Then the normalizer of $G_{\Lambda}$ in $G_{\Gamma}$ is $G_{\Lambda\circ \Lambda^{\perp}}$. 
\end{prop} 

Let now $P=gG_{\Lambda}g^{-1}$ be a parabolic subgroup. We define $P^\perp=gG_{\Lambda^\perp}g^{-1}$. This is well-defined: if we can write the parabolic subgroup $P$ in two different ways $gG_{\Lambda}g^{-1}$ and $hG_{\Lambda'}h^{-1}$, then $\Lambda=\Lambda'$ and Proposition~\ref{prop:normalizer} ensures that $gG_{\Lambda^\perp}g^{-1}=hG_{\Lambda^\perp}h^{-1}$. Moreover, the normalizer of $P$ in $G_{\Gamma}$ is $P\times P^\perp$.  

In the following lemma, we record basic facts about parabolic subgroups of right-angled Artin groups, which follow from the works of Duncan, Kazachkov and Remeslennikov in \cite[Section~2.2]{duncan2007parabolic}. Note that (5) and (6) generalize Proposition~\ref{prop:normalizer}.

\begin{lemma}[Duncan--Kazachkov--Remeslennikov]\label{lemma:parabolics}
Let $\Gamma$ be a finite simple graph.
\begin{enumerate}
	\item The intersection of two parabolic subgroups of $G_\Gamma$ is a parabolic subgroup. Moreover, for full subgraphs $\Gamma_1$ and $\Gamma_2$ of $\Gamma$, one has $G_{\Gamma_1}\cap G_{\Gamma_2}=G_{\Gamma_1\cap\Gamma_2}$. 
	\item There is a finite integer $n>0$ such that every chain $P_1\subsetneq P_2\subsetneq\cdots\subsetneq P_k$ of parabolic subgroups of $G_\Gamma$ has length $k\le n$.
	\item Every element of $G_\Gamma$ is contained in a unique smallest parabolic subgroup. 
	\item If $P_1$ and $P_2$ are two parabolic subgroups of $G_\Gamma$ with $P_1\subseteq P_2$, then the type of $P_1$ is contained in the type of $P_2$.
	\item For every parabolic subgroup $P$ of $G_\Gamma$ and every $g\in G$, if $P\subseteq g(P\times P^\perp)g^{-1}$, then $g\in P\times P^\perp$.
	\item Given any two induced subgraphs $\Gamma_1,\Gamma_2$ of $\Gamma$, and any $g\in G_\Gamma$, if $g G_{\Gamma_1}g^{-1}\subseteq G_{\Gamma_2}$, then there exists $h\in G_{\Gamma_2}$ such that $g G_{\Gamma_1}g^{-1}=h G_{\Gamma_1}h^{-1}$ (in particular $G_{\Gamma_1}\subseteq G_{\Gamma_2}$).
\end{enumerate}
\end{lemma}

\begin{proof}
The first assertion is \cite[Proposition~2.6 and Lemma~2.7]{duncan2007parabolic}. Assertion~(2) follows from \cite[Proposition~2.6]{duncan2007parabolic}, and assertion~(3) follows from \cite[Proposition~2.8]{duncan2007parabolic}. Assertions (4), (5) and (6) follow from \cite[Corollary~2.5]{duncan2007parabolic} or \cite[Proposition~2.2]{charney2007automorphisms}.
\end{proof}

\begin{rk}
	\label{remark:parabolic in parabolic}
Every parabolic subgroup $H$ of $G_\Gamma$ is itself a right-angled Artin group, with a choice of generating set induced from the generating set of $G_\Gamma$. By Proposition~\ref{prop:normalizer}, this choice is well-defined up to conjugation by elements inside $H$. Thus it makes sense to talk about parabolic subgroups of $H$ with respect to this choice of generating set of $H$. Every parabolic subgroup of $H$ is then naturally a parabolic subgroup of $G_\Gamma$. Conversely, by Lemma~\ref{lemma:parabolics} (4) and (6), every parabolic subgroup $H'$ of $G_\Gamma$ with $H'\subseteq H$ is in fact a parabolic subgroup of $H$. Thus from now on, we will just refer to $H'$ as a parabolic subgroup, without specifying its ambient group.
\end{rk}

The smallest parabolic subgroup that contains an element $g\in G_\Gamma$ is called the \emph{support} of $g$. The \emph{type} of $g$ is then defined as the type of its support.

\begin{lemma}
	\label{lemma:setwise vs pointwise}
Let $\Gamma$ be a finite simple graph, and let $\mathbb P$ be the collection of parabolic subgroups of $G_\Gamma$, equipped with the conjugation action of $G_\Gamma$. Then for every finite set $\calf \subset \mathbb P$, the pointwise stabilizer of $\calf$ in $G_\Gamma$ coincides with the setwise stabilizer of $\calf$.
\end{lemma}

\begin{proof}
Let $P$ be a parabolic subgroup in the set $\calf$. Take $g\in G_\Gamma$ with $g\calf=\calf$. Then there exists $k\neq 0$ such that $g^k$ normalizes $P$, so $g^k\in P\times P^\perp$.  Hence the subgroup $\langle g\rangle$ has a finite-index subgroup contained in the parabolic subgroup $P\times P^\perp$. By \cite[Lemma~6.4]{minasyan2012hereditary}, we have $\langle g\rangle\subseteq P\times P^\perp$. Hence $g$ fixes $P$ and the lemma follows.
\end{proof}

\paragraph*{Automorphisms.} Let $\Gamma$ be a finite simple graph. Laurence \cite{laurence1995generating} and Servatius \cite{Ser} showed that the outer automorphism group $\Out(G_\Gamma)$ is generated by the outer classes of four types of automorphisms, namely \emph{inversions}, \emph{graph automorphisms}, \emph{transvections} and \emph{partial conjugations}. We say that $\Gamma$ (or $G_\Gamma$) is \emph{transvection-free} if $\Out(G_\Gamma)$ does not contain any transvection; equivalently, there do not exist distinct vertices $v,w\in\Gamma$ such that $\lk(w)\subseteq \st(v)$. When $\Gamma$ is transvection-free, the set of parabolic subgroups of $G_\Gamma$ does not depend on the choice of standard generating set of $G_\Gamma$. It also follows from the work of Laurence and Servatius that $\Out(G_\Gamma)$ is finite if and only if $\Gamma$ is transvection-free and does not contain any separating star.

\paragraph*{The Salvetti complex.} Every right-angled Artin group $G_\Gamma$ is the fundamental group of a locally CAT(0) cube complex $S_\Gamma$, called the \emph{Salvetti complex}, obtained from a bouquet of $|V\Gamma|$ circles by gluing a $k$-cube per $k$-clique of $\Gamma$ (recall that a \emph{$k$-clique} is a complete subgraph on $k$ vertices). We refer to \cite[Section~3.6]{charney2007introduction} for more details. The 2-skeleton of $S_\Gamma$ is the presentation complex of $G_\Gamma$. For each full subgraph $\Lambda\subseteq \Gamma$, there is an isometric embedding $S_{\Lambda}\hookrightarrow S_\Gamma$. Let $\widetilde{S}_\Gamma$ be the universal cover of $S_\Gamma$, which is a CAT(0) cube complex. A \emph{standard subcomplex} of $\widetilde{S}_\Gamma$ of type $\Lambda$ is a connected component of the inverse image of $S_{\Lambda}\subseteq S_\Gamma$ with respect to the covering map $\widetilde{S}_\Gamma\to S_\Gamma$. In particular, a standard subcomplex whose type is the empty subgraph is the same as a vertex of $\widetilde{S}_\Gamma$. We collect several standard facts on $\widetilde{S}_\Gamma$.
\begin{enumerate}
	\item Two standard subcomplexes of the same type are either disjoint or equal. If a standard subcomplex of type $\Lambda_1$ and a standard subcomplex of type $\Lambda_2$ have nonempty intersection, then their intersection is a standard subcomplex of type $\Lambda_1\cap\Lambda_2$.
	\item The stabilizer of a standard subcomplex is a parabolic subgroup. Each parabolic subgroup can be realized as the stabilizer of some (non-unique) standard subcomplex. 
\end{enumerate}

\subsection{Rigidity of the extension graph}\label{sec:rigidity}

Let $\Gamma$ be a connected finite simple graph. The \emph{extension graph} of $\Gamma$, denoted $\Gamma^e$, was defined by Kim and Koberda \cite{kim2013embedability} to be the graph whose vertices are the parabolic subgroups of $G_\Gamma$ isomorphic to $\mathbb Z$, where two vertices are adjacent if the corresponding parabolic subgroups commute. As follows from Proposition~\ref{prop:normalizer} and Lemma~\ref{lemma:parabolics}(6), two distinct cyclic parabolic subgroups $P_1=g_1\langle v_1\rangle g_1^{-1}$ and $P_2=g_2\langle v_2\rangle g_2^{-1}$ commute if and only if $v_1$ and $v_2$ are adjacent in $\Gamma$ and there exists $g\in G_\Gamma$ such that for every $i\in\{1,2\}$, one has  $P_i=g\langle v_i\rangle g^{-1}$. The conjugation action of $G_\Gamma$ on itself induces an action $G_\Gamma\actson \Gamma^e$. The above observation about adjacency can be rephrased by saying that $\Gamma$ is a fundamental domain for the $G_\Gamma$-action on $\Gamma^e$. 

Now we recall the following rigidity result on extension graphs, which is a combination of  \cite[Corollary 4.16 and Lemma 4.17]{Hua}.

\begin{theo}[\cite{Hua}]\label{theo:rigidity}
Let $\Gamma_1$ and $\Gamma_2$ be two finite simple graphs, and assume that for every $i\in\{1,2\}$, the group $\Out(G_{\Gamma_i})$ is finite.

Then $\Gamma_1^e$ and $\Gamma_2^e$ are isomorphic if and only if $\Gamma_1$ and $\Gamma_2$ are isomorphic.
\end{theo}

In particular, Theorem~\ref{theointro:1} from the introduction follows from Theorem~\ref{theointro:main}.

\subsection{A lemma about transvection-free right-angled Artin groups}\label{sec:transvection-free}

In the sequel of the paper, one way in which the transvection-freeness assumption will be used is through the following easy lemma.

\begin{lemma}\label{lemma:transvection-free}
Let $G=G_\Gamma$ be a transvection-free right-angled Artin group. Let $Z$ be a cyclic parabolic subgroup of $G$, and let $P$ be a nontrivial parabolic subgroup of $G$ such that $Z\times Z^{\perp}\subseteq P\times P^{\perp}$.
\begin{enumerate}
\item If $P$ is cyclic, then $Z=P$. 
\item If $P$ is noncyclic, then $P\cap Z^{\perp}$ is nonabelian.
\end{enumerate}
\end{lemma}

\begin{proof}
Up to conjugation, we can assume $P$ is standard, hence $P\times P^\perp$ is also standard. As $Z\times Z^\perp$ can be viewed as a parabolic subgroup of $P\times P^\perp$ (see Remark~\ref{remark:parabolic in parabolic}), up to conjugating $Z\times Z^\perp$ by an element in $P\times P^\perp$, we can assume without loss of generality that both $Z$ and $Z\times Z^\perp$ are standard as such conjugation brings $P$ to itself.
	
 Let $\Gamma_P\subseteq\Gamma$ be the type of $P$. Then the type of $P^\perp$ is $(\Gamma_P)^\perp$. Let $v_Z$ be the vertex of $\Gamma$ which is the type of $Z$. As $Z\times Z^{\perp}\subseteq P\times P^{\perp}$, we have $\st(v_Z)\subseteq \Gamma_P\circ (\Gamma_P)^\perp$. 

Let $\lk(v_Z)=\Gamma_1\circ\Gamma_2\circ\cdots\circ\Gamma_k$ be the de Rham decomposition of $\lk(v_Z)$, which induces $Z^{\perp}=Q_1\times\dots \times Q_k$. The transvection-free condition implies that no $\Gamma_i$ is a clique (in particular no $\Gamma_i$ is reduced to one vertex), as otherwise the link of $v_Z$ would be contained in the star of every vertex of $\Gamma_i$. Thus, each $Q_i$ is nonabelian. As each $\Gamma_i$ is irreducible, either $\Gamma_i\subseteq\Gamma_P$ or $\Gamma_i\subseteq (\Gamma_{P})^\perp$. 

If $P$ is cyclic, then $\Gamma_P$ is reduced to a vertex $v_P$, no $\Gamma_i$ can be contained in $\Gamma_P$, and we deduce that $\lk(v_Z)\subseteq (\Gamma_P)^\perp$. So $\lk(v_Z)$ is contained in the star of $v_P$, and the transvection-free condition thus implies that $\{v_Z\}=\Gamma_P$. It follows that $Z=P$ as they are both standard, showing that the first assertion of the lemma holds.

If $P$ is noncyclic, then $\lk(v_Z)\nsubseteq (\Gamma_P)^\perp$, otherwise we have $\lk(v_Z)\subseteq \st(w)$ for every vertex $w\in\Gamma_P$. Thus at least one $\Gamma_i$ is contained in $\Gamma_P$. Therefore $\Gamma_P\cap\lk(v_Z)$ contains an irreducible graph with at least two vertices, and the second assertion of the lemma follows.
\end{proof}

\subsection{Full support subgroups}\label{sec:full-support}

\begin{prop}\label{prop:kim-koberda}
Let $\Gamma$ be a finite simple graph whose de Rham decomposition has no clique factor. Then $G_\Gamma$ contains a nonabelian free subgroup $F$ such that no nontrivial element of $F$ is contained in a proper parabolic subgroup of $G_\Gamma$.
\end{prop}
\begin{proof}
We claim that if $\Gamma=\Gamma_1\circ\Gamma_2$, then each parabolic subgroup $P$ of $G_\Gamma$ splits as a product $P_1\times P_2$ where $P_i=P\cap G_{\Gamma_i}$ is a parabolic subgroup of $G_{\Gamma_i}$. This is clearly true if $P$ is standard, otherwise there is $g=g_1g_2\in G_{\Gamma}$ ($g_i\in G_{\Gamma_i}$) such that $gPg^{-1}$ is standard. On the other hand, $gG_{\Gamma_i}g^{-1}=g_i G_{\Gamma_i} g^{-1}_i=G_{\Gamma_i}$. Thus the claim follows. This claim implies that it suffices to prove the lemma when $\Gamma$ is not a join. 

In what follows, we are equipping $\widetilde{S}_\Gamma$ with the standard $\mathrm{CAT}(0)$ metric. 
We will use the following simple observation: if a nontrivial element $g$ belongs to a parabolic subgroup of type $\Gamma'$, then it has an axis (with respect to the action $G_\Gamma\actson \widetilde{S}_\Gamma$) contained in a standard subcomplex of type $\Gamma'$, hence every axis of $g$ is contained in a finite neighborhood of this standard subcomplex (see \cite[Chapter II.6]{bridson2013metric} for basic properties of axes).

Let $X$ be the wedge of two circles $C_1$ and $C_2$, with the wedge point denoted by $x_0$. We will choose two words $W_1$ and $W_2$ in the standard generating set $V\Gamma$ such that 
\begin{enumerate}
	\item[(1)] each $W_i$ uses all the generators of $G_\Gamma$;
	\item[(2)] the map $\varphi:X\to S_\Gamma$ defined by mapping $x_0$ to the base vertex of $S_\Gamma$ and mapping each $C_i$ to the edge path in the $1$-skeleton of $S_\Gamma$ corresponding to the word $W_i$ is a local isometry ($C_i$ is metricized so that its length is equal to the word length of $W_i$).
\end{enumerate}
Then $\varphi_\ast(\pi_1 X)$ is the free subgroup satisfying our requirement, as (2) implies that any lift $\tilde \varphi:\widetilde X\to \widetilde{S}_\Gamma$ to the universal covers maps an axis of a nontrivial element $g\in \pi_1 X$ to an axis $\ell$ of $\varphi_\ast(g)$, and (1) implies that $\ell$ is not contained in any finite neighborhood of any proper standard subcomplex.

Let $\Gamma^c$ be the complement graph of $\Gamma$ (i.e.\ these two graphs have the same vertex set, two vertices in $\Gamma^c$ are adjacent if they are not adjacent in $\Gamma$). As $\Gamma$ is not a join, $\Gamma^c$ is connected. Let $u_0,v_0$ be two adjacent vertices in $\Gamma^c$. For each vertex $v\in \Gamma^c$, let $p_v$ an edge path traveling from $v_0$ to $v$ then back to $v_0$ in $\Gamma^c$. Listing consecutive vertices in $p_v$ leads to a word $W_v$ starting with $v_0$ and ending with $v_0$. Let $W$ be the product (in any order) of all $W_v$ with $v$ ranging over all vertices of $\Gamma^c$. Then the words $W_1=v_0Wu_0$ and $W_2=v^{-1}_0u_0Wu^{-1}_0$ satisfy the above two conditions.
\end{proof}

\begin{rk}
Proposition~\ref{prop:kim-koberda} can also be proved by showing that, when $\Gamma$ does not decompose nontrivially as a join, the group $G_\Gamma$ acts nonelementarily on the graph of parabolic subgroups of $G_\Gamma$ -- having one vertex per proper parabolic subgroup of $G_\Gamma$, where two such subgroups are joined by an edge whenever they have nontrivial intersection -- which turns out to be hyperbolic. We decided to provide an elementary proof that does not rely on a hyperbolicity statement. 
\end{rk}

\section{The Roller boundary of the Salvetti complex} 

The goal of the present section is to prove Theorem~\ref{theointro:roller} from the introduction (Theorem~\ref{theo:roller} below). We start with a short review on Roller boundaries of CAT(0) cube complexes.

\subsection{Background on the Roller boundary}\label{sec:background-roller} 

\paragraph*{General background.} The Roller boundary, implicit in the work of Roller \cite{Rol} and explicitly introduced in \cite{BCGNW}, gives a way of compactifying any CAT(0) cube complex. We refer to \cite{sageev2012cat} for background on CAT(0) cube complexes, including a discussion on hyperplanes and halfspaces. We now review the definition and a few facts about the Roller boundary.

Let $X$ be a CAT(0) cube complex. Given a subset $Z\subseteq X$, we let $\calh(Z)$ be the set of all hyperplanes of $X$ that intersect $Z$ nontrivially. Let $\mathfrak{H}$ be the set of all halfspaces of $X$ (the boundary hyperplane of a halfspace $h$ will be denoted $\partial h$). There is an embedding of the vertex set $V(X)$ into $\{0,1\}^{\mathfrak{H}}$ (equipped with the product topology), sending a vertex $v$ to the map $\mathfrak{H}\to\{0,1\}$ that sends a halfspace $h$ to $1$ if and only if $v\in h$. The closure of the image of this embedding yields a compactification of $V(X)$. The \emph{Roller boundary} of $X$, which we denote by $\partial_RX$, is the complement of $V(X)$ in this compactification: it is compact whenever $X$ is locally compact. A point $\xi\in\partial_R X$ is thus a map $\mathfrak{H}\to\{0,1\}$, and we let $U_\xi\subseteq\mathfrak{H}$ be the set of all halfspaces sent to $1$ under this map. Let $Y\subseteq X$ be a convex subcomplex. Then there is a continuous embedding $\partial_RY\to \partial_R X$ such that a point $\xi\in\partial_R X$ lies in the image of this embedding if and only if it corresponds to a map $\mathfrak{H}\to\{0,1\}$ that sends every halfspace containing $Y$ to $1$. Thus we will identify $\partial_R Y$ with a closed subset of $\partial_R X$.

When $X$ has countably many hyperplanes (e.g.\ when $X=\widetilde{S}_\Gamma$ is the universal cover of the Salvetti complex associated to a finite simple graph $\Gamma$), the space $\{0,1\}^{\mathfrak{H}}$ is metrizable, so $\partial_RX$ is metrizable. 

\paragraph*{Combinatorial geodesic rays and the Roller boundary.} We denote by $X^{(1)}$ the $1$-skeleton of $X$, equipped with the path metric. Geodesic rays in $X^{(1)}$ are called \emph{combinatorial geodesic rays}. Let $x\in X$ be a vertex, and let $r$ be a combinatorial geodesic ray in $X^{(1)}$ originating at $x$. For every hyperplane $\mathfrak{h}$, the ray $r$ selects exactly one of the two halfspaces complementary to $\mathfrak{h}$, by considering the halfspace which virtually contains $r$ (i.e.\ contains $r$ up to a finite segment). Thus $r$ defines a point in $\partial_R X$. Two combinatorial geodesic rays are \emph{equivalent} if they cross the same set of hyperplanes. There is a 1-1 correspondence between equivalence classes of combinatorial geodesic rays based at a given point $x\in X$, and points in $\partial_R X$ (see e.g.\ \cite[Section~A.2]{genevois2020contracting}).

Now suppose that two combinatorial geodesic rays $r$ and $r'$ (potentially with different origins $x$ and $x'$) represent the same point of $\partial_R X$. By definition of $\partial_R X$, a halfspace of $X$ virtually contains $r$ if and only if it virtually contains $r'$. Also, every hyperplane $\mathfrak{h}$ contained in the symmetric difference $\calh(r)\Delta\calh(r')$ separates $x$ from $x'$. In particular $\calh(r)\Delta\calh(r')$ is finite. 

\paragraph*{Regular points in the Roller boundary.} Two hyperplanes $\mathfrak{h}_1$ and $\mathfrak{h}_2$ of $X$ are \emph{strongly separated} \cite[Definition~2.1]{behrstock2012divergence} if there does not exist any hyperplane $\mathfrak{h}$ such that for every $i\in\{1,2\}$, one has $\mathfrak{h}\cap\mathfrak{h}_i\neq\emptyset$. The notion of a regular point of the Roller boundary, introduced by Fern\'os in \cite[Definition~7.3]{Fer}, is defined as follows: a point $\xi\in\partial_R X$ is \emph{regular} if given any two halfspaces $h_1,h_2\in U_\xi$, there exists a halfspace $k\in U_\xi$ such that $k\subseteq h_1\cap h_2$ and for every $i\in\{1,2\}$, the hyperplanes $\partial k$ and $\partial h_i$ are strongly separated. We denote by $\partial_{\reg}X$ the subspace of $\partial_R X$ made of regular points. 

\subsection{Non-regular points in the Roller boundary of the Salvetti complex}

Let $\Gamma$ be a finite simple graph. We denote by $\mathbb{S}_J$ the set of all standard subcomplexes of $\widetilde{S}_\Gamma$ that are contained in some join standard subcomplex of $\widetilde{S}_\Gamma$, i.e.\ a standard subcomplex whose type admits a nontrivial join decomposition. The goal of the present section is to prove the following theorem.

\begin{theo}\label{theo:roller}
Let $\Gamma$ be a finite simple graph. Then $\partial_\reg\widetilde{S}_\Gamma$ is a Borel subset of $\partial_R\widetilde{S}_\Gamma$, and there exists a unique Borel $G_\Gamma$-equivariant map $$\Phi:\partial_R\widetilde{S}_\Gamma\setminus\partial_\reg\widetilde{S}_\Gamma\to\mathbb{S}_J$$ such that for every nonregular point $\xi\in\partial_R\widetilde{S}_\Gamma$ and every combinatorial geodesic ray $r$ representing $\xi$, the subset $\Phi(\xi)$ is the smallest standard subcomplex of $\widetilde{S}_\Gamma$ that virtually contains $r$.  
\end{theo}

Theorem~\ref{theo:roller} is a consequence of Lemmas~\ref{lemma:confined},~\ref{lem:borel} and~\ref{lem:join subgroup} below.

\begin{lemma}\label{lemma:confined}
Let $\Gamma$ be a finite simple graph. For every $\xi\in\partial_R\widetilde{S}_\Gamma$, there exists a unique standard subcomplex $Y\subseteq \widetilde{S}_\Gamma$ such that for every combinatorial geodesic ray $r$ representing $\xi$, the subcomplex $Y$ is the smallest standard subcomplex that virtually contains $r$.  
\end{lemma}

\begin{proof}
Recall that the intersection of two standard subcomplexes is again a standard subcomplex and that there is a uniform bound on the length of a strictly descending chain of standard subcomplexes. Therefore, for every combinatorial geodesic ray $r$, there is a unique minimal standard subcomplex $Y_r\subseteq\widetilde{S}_\Gamma$ which virtually contains $r$. 

We are thus left with showing that if $r$ and $r'$ are combinatorial geodesic rays representing the same point of $\partial_R \widetilde{S}_\Gamma$ (possibly with different base points), then $Y_r=Y_{r'}$. In this case, as observed in Section~\ref{sec:background-roller},
\begin{enumerate}
	\item[(1)] a halfspace of $\widetilde{S}_\Gamma$ virtually contains $r$ if and only if it virtually contains $r'$;
	\item[(2)] the symmetric difference of $\mathcal{H}(r)$ and $\mathcal{H}(r')$ is finite.
\end{enumerate}
Let $\Gamma_r$ be the type of $Y_r$. As the 1-skeleton of $\widetilde{S}_\Gamma$ is the Cayley graph of $G_\Gamma$, we label edges of $\widetilde{S}_\Gamma$ by vertices of $\Gamma$. Then there exists a subray $r_1$ of $r$ such that the collection $L(r_2)$ of labels of edges of any further subray $r_2$ of $r_1$ satisfies $L(r_2)=V\Gamma_r$. We define $r'_1$ and $\Gamma_{r'}$ similarly. By (2), for any subray $r'_3$ of $r'_1$, there is a subray $r_3$ of $r_1$ with $L(r_3)\subseteq L(r'_3)$. Thus $\Gamma_r\subseteq \Gamma_{r'}$. Similarly $\Gamma_{r'}\subseteq \Gamma_r$. Therefore $\Gamma_r=\Gamma_{r'}$, and in fact $Y_r=Y_{r'}$: indeed (1) implies that $Y_r\cap Y_{r'}\neq\emptyset$, as otherwise $Y_r$ and $Y_{r'}$ would be separated by a hyperplane (see \cite[Corollary 13.10]{haglund2008special}). This concludes our proof.
\end{proof}

Lemma~\ref{lemma:confined} yields a well-defined map $\Phi:\partial_R \widetilde{S}_\Gamma\to \mathbb{S}$, where $\mathbb{S}$ is the set of all standard subcomplexes of $\widetilde{S}_\Gamma$ (including $\widetilde{S}_\Gamma$ itself). Moreover, the map $\Phi$ is equivariant with respect to the natural actions of $G_\Gamma$. 

\begin{rk}
The map $\Phi$ can be reinterpreted as follows: Lemma~\ref{lemma:confined} shows that given any two standard subcomplexes $Y_1,Y_2\subseteq\widetilde{S}_\Gamma$, one has $\partial_R Y_1\cap\partial_R Y_2=\partial_R(Y_1\cap Y_2)$ (viewed as subsets of $\partial_R\widetilde{S}_\Gamma$), and the map $\Phi$ then sends a point $\xi\in\partial_R\widetilde{S}_\Gamma$ to the smallest standard subcomplex $Y$ such that $\xi\in\partial_RY$.
\end{rk}

We equip the countable set $\mathbb{S}$ with the discrete topology.

\begin{lemma}
	\label{lem:borel}
The map $\Phi$ is Borel.
\end{lemma}

\begin{proof}
Let $Y\in \mathbb {S}$ be a standard subcomplex, and let $x\in \widetilde{S}_\Gamma$. Then a point $\xi\in\partial_R\widetilde{S}_\Gamma$ belongs to $\Phi^{-1}(Y)$ if and only if it is represented by a combinatorial geodesic ray based at $x$ which is virtually contained in $Y$, but not virtually contained in any proper standard subcomplex of $Y$. Thus $\Phi^{-1}(Y)=\partial_R Y\setminus(\cup_{Z\in \mathbb {S}, Z\subsetneq Y}\partial_R Z)$. The lemma follows as the subspaces $\partial_R Z$ and $\partial_R Y$ are closed subsets of $\partial_R \widetilde{S}_\Gamma$.
\end{proof}

Now we recall some standard facts to prepare for the next lemma. As $\widetilde{S}^{(1)}_\Gamma$ is the Cayley graph of $G_\Gamma$ with respect to its standard generating set, we label each edge of $\widetilde{S}_\Gamma$ by a vertex of $\Gamma$. Recall that an edge is \emph{dual} to a hyperplane if this hyperplane intersects the edge in its midpoint. Note that if two edges are dual to the same hyperplane, then they have the same label. Hence each hyperplane has a well-defined label. Note that two hyperplanes having the same label are either equal or disjoint. Let $h$ be a hyperplane in $\widetilde{S}_\Gamma$. Then the smallest subcomplex of $\widetilde{S}_\Gamma$ containing $h$ splits as a product $h\times [0,1]$. Note that both $h\times\{0\}$ and $h\times\{1\}$ are standard subcomplexes of type $\lk(v)$, where $v\in \Gamma$ is the label of $h$. Thus it is natural to consider \emph{standard subcomplexes} of $h$ and their types, which correspond to standard subcomplexes of $h\times\{0\}$ (or $h\times\{1\}$).

Let $X_1$ and $X_2$ be standard subcomplexes of $\widetilde{S}_\Gamma$. It is a standard fact that one can find a standard subcomplex $B\subseteq X_1$ such that $\calh(B)=\calh(X_1)\cap \calh(X_2)$. Actually, we can take $B$ to be the subset of $X_1$ made of all points whose distance to $X_2$ is equal to $d(X_1,X_2)$. Then $B$ is a standard subcomplex of $X_1$ (see \cite[Lemma~3.1]{Hua}) and $\calh(B)=\calh(X_1)\cap \calh(X_2)$ (see e.g. \cite[Lemma~2.14]{Hua3}). By the discussion in the previous paragraph, the fact still holds if we require $X_1$ and $X_2$ to be hyperplanes.

\begin{lemma}\label{lem:join subgroup}
Let $\Gamma$ be a finite simple graph. Let $r$ be a combinatorial geodesic ray in $\widetilde{S}_\Gamma$. Then $r$ represents a point in $\partial_\reg\widetilde{S}_\Gamma$ if and only if $r$ is not virtually contained in a join standard subcomplex of $\widetilde{S}_\Gamma$.
\end{lemma}

\begin{proof}
First we show that if $r$ is virtually contained in a join standard subcomplex $X$ representing $\xi\in \partial_R \widetilde{S}_\Gamma$, then $\xi$ is not regular. To see this,  let $X_1$ and $X_2$ be two unbounded standard subcomplexes of $\widetilde{S}_\Gamma$ such that $X$ splits as $X=X_1\times X_2$. We assume without loss of generality that $r\subset X$, and the image of $r$ under the natural projection $X\to X_1$ is unbounded. Take halfspaces $h_1,h_2\in U_\xi$ such that $\partial h_i\in \mathcal H(X_1)\cap\mathcal H(r)$. By definition of  a regular point, it suffices to show that if $k$ is a halfspace with $k\in U_\xi$ and $k\subset h_1\cap h_2$, then $\partial k$ and $\partial h_i$ are not strongly separated for any $i\in\{1,2\}$. Note that $\partial k\cap X\neq \emptyset$, as otherwise we will have $r\subset k$, and therefore $r\subset h_1$, contradicting that $\partial h_1\in \mathcal H(r)$. As $\mathcal H(X)=\mathcal H(X_1)\sqcup\mathcal H(X_2)$, we have $\partial k\in \mathcal H(X_1)$ or $\partial k\in \mathcal H(X_2)$. The latter is not possible as it will imply $\partial k\cap\partial h_1\neq\emptyset$, contradicting $k\subset h_1$. Now $\partial k\in \mathcal H(X_1)$ implies that $\partial k$ and $\partial h_1$ cannot be strongly separated, because any hyperplane in $\mathcal H(X_2)$ will intersect both $\partial k$ and $\partial h_1$.

Now we suppose that $r$ does not represent a regular point. Let $\{\mathfrak h_i\}_{i\in\mathbb{N}}$ be an infinite collection of hyperplanes crossed by $r$ such that for every $i\in\mathbb{N}$, the hyperplane $\mathfrak h_{i}$ separates $\mathfrak h_{i-1}$ from $\mathfrak h_{i+1}$ (such a collection can be found by looking at hyperplanes with the same label, though this is a general fact for geodesic rays in finite dimensional CAT(0) cube complexes). By  \cite[Proposition~7.5]{Fer}, up to passing to a subcollection, we can assume that for all $i,j\in\mathbb{N}$, the hyperplanes $\mathfrak h_i$ and $\mathfrak h_j$ are not strongly separated.

For every $i\in\mathbb{N}$, let $B_i\subseteq\mathfrak h_1$ be a standard subcomplex of $\mathfrak h_1$ with $\calh(B_i)=\calh(\mathfrak h_1)\cap \calh(\mathfrak h_i)$. For every $i_1<i_2<i_3$, every hyperplane that intersects both $\mathfrak h_{i_1}$ and $\mathfrak h_{i_3}$ also intersects $\mathfrak h_{i_2}$. This implies that for every $j\le i$, one has $\calh(B_i)\subseteq \calh(B_j)$. Thus the type $\Gamma_i$ of $B_i$ is contained in $\Gamma_j$, with $\Gamma_i=\Gamma_j$ if and only if $\calh(B_i)=\calh(B_j)$. Thus there exists $i_0\in\mathbb{N}$ such that for every $i\ge i_0$, one has $\Gamma_i=\Gamma_{i_0}$. We define $A_1=B_{i_0}$. Then $\calh(A_1)=\cap_{j=1}^{\infty}\calh(\mathfrak h_j)$. Moreover, as $\mathfrak h_1$ and $\mathfrak h_{i_0}$ are not strongly separated, we have $\Gamma_{i_0}\neq\emptyset$. Similarly, for each $i\ge 1$ we define a standard subcomplex $A_i\subseteq \mathfrak h_i$ such that $\calh(A_i)=\cap_{k=i}^{\infty}\calh(\mathfrak h_k)$. Thus $\calh(A_j)\subseteq\calh(A_i)$ whenever $j\le i$. By a similar argument as before (the type of every $A_i$ is a subgraph of $\Gamma$), up to passing to an infinite subset of $\{\mathfrak h_i\}_{i\ge 1}$, we can (and will) assume that $\calh(A_i)=\calh(A_j)$ for every $i\neq j$. Let $\Lambda$ be the type of $A_1$, and let $Y$ be the standard subcomplex of $\widetilde{S}_\Gamma$ with type $\Lambda\circ\Lambda^\perp$ that contains $A_1$. Then for every $i\ge 1$, we have $A_i\subseteq Y$. We already know $\Lambda\neq\emptyset$ from the previous discussion, and $\Lambda^\perp\neq\emptyset$ as the label of any edge dual to the hyperplane containing $A_1$ belongs to $\Lambda^\perp$. We pass to a further infinite subset of $\{\mathfrak h_i\}_{i=1}^\infty$ so that $\calh(A_1)=\calh(\mathfrak h_1)\cap \calh(\mathfrak h_2)$.

Let $r'$ be a combinatorial subray of $r$ with $r'\cap \mathfrak h_2=\emptyset$. We claim that $\calh(r')\subseteq\calh(Y)$. This claim will imply that $r'$ is contained in a standard subcomplex of the same type as $Y$, thus concluding the proof of the lemma.

We are thus left proving the above claim. Suppose towards a contradiction that there exists $\mathfrak h'\in \calh(r')\setminus\calh(Y)$. If $\mathfrak h'$ intersects infinitely many members of $\{\mathfrak h_i\}_{i=1}^\infty$, then there exists $j_0\in\mathbb{N}$ such that $\mathfrak h'\in \cap_{j=j_0}^\infty \calh(\mathfrak h_j)= \calh(A_{j_0})=\calh(A_1)\subseteq\calh(Y)$, a contradiction. Thus $\mathfrak h'$ intersects at most finitely many members from $\{\mathfrak h_i\}_{i=1}^\infty$. Let $e'$ (resp.\ $e_i$) be the edge of $r$ dual to $\mathfrak h'$ (resp.\ $\mathfrak h_i$). Then there exists $\ell_0\in\mathbb{N}$ such that $e_{\ell_0}$ occurs after $e'$ in $r'$ and $\mathfrak h_{\ell_0}\cap \mathfrak h'=\emptyset$. We will now argue separately depending on whether $\mathfrak{h}'$ intersects $\mathfrak{h}_1$ or not, and reach a contradiction in both cases.

We first assume that $\mathfrak h'\cap \mathfrak h_1\neq\emptyset$. As $r'$ intersects $\mathfrak h'$ but not $\mathfrak{h}_2$, it follows that $\mathfrak{h}'$ also intersects $\mathfrak{h}_2$. Therefore $\mathfrak{h}'$ has nontrivial intersection with both $\mathfrak h_1$ and $\mathfrak h_2$, which implies that $\mathfrak h'\in\calh(\mathfrak h_1)\cap \calh(\mathfrak h_2)=\calh(A_1)\subseteq\calh(Y)$. This is a contradiction. 

We now assume that  $\mathfrak h'\cap \mathfrak h_1=\emptyset$. As $e_1$ occurs before $e'$ in $r$, we know that $\mathfrak h'$ separates $\mathfrak h_1$ from $\mathfrak h_{\ell_0}$. As $\mathfrak h_1$ and $\mathfrak h_{\ell_0}$ belong to $\calh(Y)$, we have $\mathfrak h'\in \calh(Y)$ by convexity of $Y$, a contradiction. This concludes our proof.
\end{proof}

We denote by $\mathbb P_J$ the set of all parabolic subgroups of $G_\Gamma$ which are contained in a parabolic subgroup whose type admits a nontrivial join decomposition. Combining the map $\Phi$ given by Theorem~\ref{theo:roller} with the $G_\Gamma$-equivariant map $\mathbb{S}_J\to\mathbb{P}_J$ sending a standard subcomplex to its stabilizer, we reach the following corollary, which is the form in which Theorem~\ref{theo:roller} will be used in the sequel.

\begin{cor}\label{cor:roller}
Let $\Gamma$ be a finite simple graph. There exists a Borel $G_\Gamma$-equivariant map $\partial_R\widetilde{S}_\Gamma\setminus\partial_\reg\widetilde{S}_\Gamma\to\mathbb{P}_J$.
\qed
\end{cor}

\begin{rk}
It is natural to ask how much of the discussion in this section applies to more general $\mathrm{CAT}(0)$ cube complexes. As a starting point, we consider the following geometric question.
Let $X$ be a $\mathrm{CAT}(0)$ cube complex, possibly with a geometric group action. Is it possible to split $\partial_R X$ into a regular part $K_1$, corresponding to ``hyperbolic directions'' in $X$, and a non-regular part $K_2=\partial_{R} X\setminus K_1$, corresponding to ``product regions'' of $X$ in the sense that each point in $K_2$ is contained in the Roller boundary of a subcomplex which splits as a product of two unbounded cube complexes? 

Note that in general we cannot take $K_1$ to be $\partial_{\reg} X$. Indeed, consider $X=[0,1]\times \mathbb R$ with the usual cubical structure. Then every point of the Roller boundary is non-regular, but no subcomplex in $X$ splits as a product of two unbounded cube complexes. To remedy this, we define a point $\xi\in\partial_R X$ to be \emph{weakly regular} if there exist $N\ge 0$ and an infinite descending chain of halfspaces $\{h_i\}_{i=1}^\infty\subset U_\xi$ such that for any $i\neq j$, there are at most $N$ hyperplanes of $X$ intersecting both $\partial h_i$ and $\partial h_j$. Note that the definition of a regular point corresponds to the case where $N=0$ (see \cite[Proposition 7.5]{Fer}).
In the above example $X=[0,1]\times\mathbb{R}$, every point of the Roller boundary is now weakly regular. And we can take $K_1$ to be the collection of weakly regular points. 

A second issue is that a non weakly regular point $\xi$ might be contained in the Roller boundary of a cubical staircase inside $X$ (see e.g.\ \cite[Figure~1]{hagen2020hierarchical}), but if this
staircase is not contained in a quarter plane, then $\xi$ is not contained in the Roller boundary of a product subcomplex. Work of Hagen and Susse \cite{hagen2020hierarchical} gives conditions on when one can complete the staircases in $X$ to quarter planes, i.e. \cite[Theorem~A]{hagen2020hierarchical} implies the existence of a \emph{factor system} as defined in \cite[Section~1.2]{hagen2020hierarchical}, and factor systems lead to completions of staircases as discussed in \cite[Section~7]{hagen2020hierarchical}. In particular \cite[Theorem~A]{hagen2020hierarchical} applies to universal covers of Salvetti complexes and many other cube complexes.

Given the above connection, one can prove Lemma~\ref{lem:join subgroup} using results from \cite{hagen2020hierarchical}. However, we chose to write out a self-contained proof for the convenience of the reader, as in the setting of Lemma~\ref{lem:join subgroup} the proof is a short argument.
\end{rk}

\section{Main measure equivalence classification theorem}\label{sec:me}

In this section, we prove our main theorems concerning the measure equivalence classification of right-angled Artin groups with finite outer automorphism groups (Theorems~\ref{theointro:1} and~\ref{theointro:main}).

\subsection{Measured groupoids}

\paragraph*{General definitions.} General background about measured groupoids can be found in \cite[Section~2.1]{AD} or \cite{Kid-survey}, for instance. By definition, a Borel groupoid $\mathcal{G}$ is a standard Borel space which comes equipped with a base space $Y$ (which is also a standard Borel space, and every element $g\in\calg$ can be thought of as an arrow with a source $s(g)$ and a range $r(g)$ in $Y$, the maps $s$ and $r$ being Borel). A Borel groupoid $\mathcal{G}$ also comes equipped by definition with a composition law, an inverse map (all Borel) and a neutral element $e_y$ for every $y\in Y$, satisfying the axioms of groupoids (see e.g.\ \cite[Definition~2.10]{Kid-survey}). All Borel groupoids considered in the present paper are \emph{discrete}, i.e.\ there are at most countably many elements with a given source or range. 

As in \cite[Definition~2.13]{Kid-survey}, we define a \emph{measured groupoid} as a Borel groupoid $\mathcal{G}$ for which the base space $Y$ has a quasi-invariant $\sigma$-finite measure -- which will always be a probability measure in the present paper. There is a natural notion of measured subgroupoid of $\mathcal{G}$, as well as a notion of restriction $\mathcal{G}_{|U}$ of a measured groupoid $\mathcal{G}$ to a Borel subset $U$ of the base space $Y$, by only considering elements of $\calg$ whose source and range both belong to $U$. 

\paragraph*{Stably trivial groupoids, groupoids of infinite type.} A measured groupoid $\calg$ over a base space $Y$ is \emph{trivial} if $\calg=\{e_y|y\in Y\}$. It is \emph{stably trivial} if there exist a conull Borel subset $Y^*\subseteq Y$ and a partition $Y^*=\dunion_{i\in I} Y_i$ into at most countably many Borel subsets such that for every $i\in I$, the groupoid $\calg_{|Y_i}$ is trivial. It is \emph{of infinite type} if for every Borel subset $U\subseteq Y$ of positive measure, and a.e.\ $y\in U$, there are infinitely many elements of $\calg_{|U}$ with source $y$. Notice that any restriction of a subgroupoid of infinite type is again of infinite type.

\paragraph*{Cocycles.} Given a measured groupoid $\calg$ over a base space $Y$ and a countable group $G$, a \emph{strict cocycle} $\rho:\calg\to G$ is a Borel map such that for all $g_1,g_2\in\calg$ with $s(g_1)=r(g_2)$ (where $s$ and $r$ denote the source and range maps, respectively), one has $\rho(g_1g_2)=\rho(g_1)\rho(g_2)$. Here the word \emph{strict} refers to the fact that this relation is assumed to hold everywhere, not just on a conull Borel subset. The \emph{kernel} of a strict cocycle $\rho$ is the subgroupoid of $\calg$ made of all elements $g$ such that $\rho(g)=e$.

Let $\calg$ be a measured groupoid over a base space $Y$, and let $G$ be a countable group. We say that a strict cocycle $\calg\to G$ is \emph{action-type} if it has trivial kernel and for every infinite subgroup $H\subseteq G$, the subgroupoid $\rho^{-1}(H)$ is of infinite type. An important example, which motivates the terminology, is that when $\calg$ is the measured groupoid naturally associated to a measure-preserving action of a countable group $G$ on a finite measure space $Y$, the cocycle $\rho$ given by the action is action-type (see \cite[Proposition~2.26]{Kid-survey} for details, which relies on work of Adams \cite{Ada}).  We emphasize that it is crucial in this example that the $G$-action on $Y$ preserves a finite measure, see Remark~\ref{rk:action-type} below.

\paragraph*{Invariant maps and stabilizers.} When $G$ is acting on a standard Borel space $\Delta$ by Borel automorphisms, and $\rho:\calg\to G$ is a strict cocycle, we say that a Borel map $\phi:Y\to\Delta$ is \emph{$(\calg,\rho)$-invariant} if there exists a conull Borel subset $Y^*\subseteq Y$ such that for all $g\in\calg_{|Y^*}$, one has $\phi(r(g))=\rho(g)\phi(s(g))$.
Also, for every $\delta\in\Delta$, we say that $\delta$ is \emph{$(\calg,\rho)$-invariant} if the constant map with value $\delta$ is $(\calg,\rho)$-invariant. In other words, there exists a conull Borel subset $Y^*\subseteq Y$ such that $\rho(\calg_{|Y^*})\subseteq \Stab_G(\delta)$. The \emph{$(\calg,\rho)$-stabilizer} of $\delta$ is defined as the subgroupoid of $\calg$ made of all $g\in\calg$ such that $\rho(g)\in\Stab_G(\delta)$.

\paragraph*{Stable containment and stable equivalence.} Let $\calg$ be a measured groupoid over a base space $Y$, and let $\calh,\calh'$ be two measured subgroupoids of $\calg$. We say that $\calh$ is \emph{stably contained} in $\calh'$ if there exists a conull Borel subset $Y^*\subseteq Y$ and a partition  $Y^*=\dunion_{i\in I}Y_i$ into at most countably many Borel subsets, such that for every $i\in I$, one has $\calh_{|Y_i}\subseteq\calh'_{|Y_i}$. We say that $\calh$ and $\calh'$ are \emph{stably equivalent} if each of them is stably contained in the other -- in other words, there exists a conull Borel subset $Y^*\subseteq Y$ and a partition  $Y^*=\dunion_{i\in I}Y_i$ into at most countably many Borel subsets such that for every $i\in I$, one has $\calh_{|Y_i}=\calh'_{|Y_i}$. 

\paragraph*{Normalization.}

A notion of normality of a subgroupoid was proposed by Kida in \cite[Chapter~4, Section~6.1]{Kid-memoir}, building on work of Feldman, Sutherland and Zimmer \cite{FSZ}, defined as follows. Given a measured groupoid $\calg$ over a base space $Y$, a measured subgroupoid $\calh\subseteq\calg$, and a Borel subset $B\subseteq\calg$, one says that $\calh$ is \emph{$B$-invariant} if there exists a conull Borel subset $Y^*\subseteq Y$ such that for every $g_1,g_2\in B\cap\calg_{|Y^*}$ and every $h\in\calg_{|Y^*}$ satisfying $s(h)=s(g_1)$ and $r(h)=s(g_2)$, one has $h\in\calh$ if and only if $g_2hg_1^{-1}\in\calh$. Given two measured subgroupoids $\calh,\calh'\subseteq\calg$, one says that $\calh$ is \emph{normalized} by $\calh'$ if there exists a covering of $\calh'$ by countably many Borel subsets $B_n\subseteq\calg$ so that for every $n\in\mathbb{N}$, the groupoid $\calh$ is $B_n$-invariant. One says that $\calh$ is \emph{stably normalized} by $\calh'$ if there exists a partition $Y=\dunion_{i\in I}Y_i$ into at most countably many Borel subsets such that for every $i\in I$, the groupoid $\calh_{|Y_i}$ is normalized by $\calh'_{|Y_i}$.

 An important example is that if $\calg$ is  equipped with a strict cocycle $\rho$ towards a countable group $G$, and if $H_1,H_2\subseteq G$ are two subgroups such that $H_1$ is normalized by $H_2$, then $\rho^{-1}(H_1)$ is normalized by $\rho^{-1}(H_2)$.

Using a theorem of Lusin and Novikov (see \cite[Theorem~18.10]{Kec}), the subsets $B_n$ that arise in the definition of $\calh$ being normalized by $\calh'$ can always be chosen so that the restrictions of the source and range maps to each $B_n$ are Borel isomorphisms to Borel subsets of $Y$.

\paragraph*{Amenability.}

We will adopt the same definition of amenability of a measured groupoid as in \cite{Kid-survey}. The crucial way in which amenability is used is the following: if $\calg$ is a measured groupoid over a base space $Y$ with a cocycle $\rho$ towards a countable group $G$, and if $G$ acts by homeomorphisms on a compact metrizable space $K$, then there exists a $(\calg,\rho)$-invariant Borel map $Y\to\Prob(K)$, where $\Prob(K)$ denotes the space of Borel probability measures on $K$ equipped with the weak-$*$ topology (coming from the duality with the space of real-valued continuous functions on $K$), see \cite[Proposition~4.14]{Kid-survey}. We say that a groupoid $\calg$ over a base space $Y$ is \emph{everywhere nonamenable} if for every Borel subset $U\subseteq Y$ of positive measure, the groupoid $\calg_{|U}$ is nonamenable. We will need a few invariance properties: every subgroupoid of an amenable groupoid is amenable \cite[Theorem~4.16]{Kid-survey} (and every restriction of an amenable groupoid is amenable); also, if $Y^*\subseteq Y$ is a conull Borel subset of the base space and $Y^*=\dunion_{i\in I} Y_i$ is a partition into at most countably many Borel subsets, and if for every $i\in I$, the groupoid $\calg_{|Y_i}$ is amenable, then $\calg$ is amenable. 

\begin{rk}
	\label{rk:Borel} 
Given a separable metrizable topological space $X$, the set $\Prob(X)$ of Borel probability measures on $X$ is endowed with the topology generated by the maps $\mu\mapsto\int fd\mu$, where $f$ varies over the set of all real-valued continuous bounded functions on $X$. When $X$ is compact, the Riesz--Markov--Kakutani theorem identifies $\Prob(X)$ with a subspace of the unit ball of the dual of $C(X;\mathbb{R})$ (the space of real-valued continuous functions on $X$), and under this identification the above topology coincides with the weak-$*$ topology. When $X$ is countable, this is the same as the topology of pointwise convergence. In our applications, the space $X$ will always be either a subspace of a compact metrizable space or a countable set with discrete topology. By \cite[Theorem~17.24]{Kec}, the Borel $\sigma$-algebra on $\Prob(X)$ is the $\sigma$-algebra generated by the maps $\mu\mapsto\mu(A)$, where $A$ varies over the Borel subsets of $X$. As a consequence, every Borel map $f:X\to Y$ induces a Borel map $\Prob(X)\to\Prob(Y)$.
\end{rk}

\subsection{Measured groupoids with cocycles towards free groups}

Throughout the paper, we will extensively use the following observation of Kida.

\begin{lemma}[{Kida \cite[Lemma~3.20]{Kid}}]\label{lemma:everywhere-nonamenable}
Let $G$ be a countable group, and let $\calg$ be a measured groupoid over a standard finite measure space $Y$.

If there exists a strict action-type cocycle $\calg\to G$, and if $G$ contains a nonabelian free subgroup, then $\calg$ is everywhere nonamenable.
\end{lemma}

\begin{rk}\label{rk:action-type}
 In \cite[Lemma~3.20]{Kid}, Kida assumes that the groupoid $\calg$ preserves a probability measure. In his proof, this assumption is used to ensure that for every infinite cyclic subgroup $A\subseteq G$ and every Borel subset $U\subseteq Y$ of positive measure, the groupoid $\rho^{-1}(A)_{|U}$ is of infinite type. For us, this is automatically ensured by our assumption that the cocycle $\calg\to G$ is action-type. 

We warn the reader about the following example. Let $\partial_\infty F_N\rtimes F_N$ be the measured groupoid coming from the action of a non-amenable free group on its boundary (which quasi-preserves a probability measure, but does not preserve any). Then $\partial_\infty F_N\rtimes F_N$ is amenable: this follows from the universal amenability of the action of a free group on its boundary. But the natural cocycle $\partial_\infty F_N\rtimes F_N\to F_N$ is not of action type.
\end{rk}

 We mention that on the other hand, if there exists a strict cocycle with trivial kernel from $\calg$ to an amenable group, then $\calg$ is amenable (this follows from \cite[Proposition~4.33]{Kid-memoir}, for instance). We will need the following lemma, whose proof relies on techniques introduced by Adams in \cite{Ada}.

\begin{lemma}\label{lemma:free}
Let $A_1,A_2,A_3$ be finitely generated free groups, and let $\calg$ be a measured groupoid over a base space $Y$, equipped with a strict cocycle $\rho:\calg\to A_1\ast A_2\ast A_3$ with trivial kernel. Let $\cala$ be an amenable measured subgroupoid of $\calg$. 

Then there exist $i\in\{1,2,3\}$ and a Borel subset $U\subseteq Y$ of positive measure such that $(\cala\cap \rho^{-1}(A_i))_{|U}$ is trivial.
\end{lemma}

In the following proof, we will use Remark~\ref{rk:Borel} implicitly when we check Borel measurability of maps.

\begin{proof}
For every $i\in\{1,2,3\}$, we let $\cala_i=\cala\cap\rho^{-1}(A_i)$. Given $i\neq j$, we let $\cala_{ij}=\cala\cap\rho^{-1}(A_i\ast A_j)$. Assume towards a contradiction that for every Borel subset $U\subseteq Y$ of positive measure, none of the subgroupoids $(\cala_1)_{|U},(\cala_2)_{|U},(\cala_3)_{|U}$ is trivial (in particular, neither are the restrictions of $\cala$ and of the $\cala_{ij}$ to $U$). Let $F=A_1\ast A_2\ast A_3$. For every $i\in\{1,2,3\}$, let $B_i$ be a free basis of $A_i$, and let $T$ be the Cayley graph of $F$ with respect to the free basis $B_1\cup B_2\cup B_3$, on which $F$ naturally acts by isometries. As $\partial_\infty T$ is compact and metrizable and $\cala$ is amenable, there exists an $(\cala,\rho)$-invariant Borel map $\theta:Y\to \Prob(\partial_\infty T)$. 

We claim that for every such map $\theta$, and almost every $y\in Y$, the support of $\theta(y)$ has cardinality at most $2$. Indeed, otherwise, there exists a Borel subset $U\subseteq Y$ of positive measure such that for every $y\in U$, the support of the probability measure $\theta(y)$ has cardinality at least $3$. Thus, the probability measure $\theta(y)\otimes\theta(y)\otimes\theta(y)$ on $(\partial_\infty T)^3$ gives positive measure to the subset $(\partial_\infty T)^{(3)}$ made of pairwise distinct triples. By restricting and renormalizing, we thus derive an $(\cala_{|U},\rho)$-invariant Borel map $U\to\Prob((\partial_\infty T)^{(3)})$. Denoting by $V(T)$ the vertex set of $T$, there is a natural $F$-equivariant barycenter map $(\partial_\infty T)^{(3)}\to V(T)$. This map is Borel (in fact continuous) with respect to the natural topology on $(\partial_\infty T)^{(3)}$. This yields an $(\cala_{|U},\rho)$-invariant Borel map $U\to\Prob(V(T))$. Let $\calf(V(T))$ denote the countable set of all nonempty finite subsets of $V(T)$. As $V(T)$ is countable, there is an $F$-equivariant map $\Prob(V(T))\to\calf(V(T))$, sending a probability measure $\nu$ to the finite set made of all elements of maximal $\nu$-measure. One readily checks that this map is Borel. We deduce that there is an $(\cala_{|U},\rho)$-invariant Borel map $U\to\calf(V(T))$. Restricting to a positive measure Borel subset $U'\subseteq U$ where this map takes a constant value, we deduce (up to replacing $U'$ by a conull Borel subset) that $\rho(\cala_{|U'})=\{1\}$. As $\rho$ has trivial kernel, it follows that $\cala_{|U'}$ is trivial, a contradiction. 

Let $\calp_{\le 2}(\partial_\infty T)$ (resp.\ $\calp_{=2}(\partial_\infty T)$) denote the set of all nonempty subsets of $\partial_\infty T$ of cardinality at most $2$ (resp.\ equal to $2$). The above argument shows that there exists an $(\cala,\rho)$-invariant Borel map $\theta:Y\to\calp_{\le 2}(\partial_\infty T)$ (obtained after mapping a probability measure to its support). In particular, this map is $(\cala_i,\rho)$-invariant for every $i\in\{1,2,3\}$. Notice also that given two such maps, the above argument also ensures that their union should still take its values (essentially) in $\calp_{\le 2}(\partial_\infty T)$.

For every $i\in\{1,2,3\}$, we let $T_i$ be the Cayley tree of $A_i$ with respect to the free basis $B_i$. For $i\neq j$, we let $T_{ij}$ be the Cayley tree of $A_i\ast A_j$ with respect to the free basis $B_i\cup B_j$. There are natural embeddings $T_i\hookrightarrow T_{ij}\hookrightarrow T$ and $\partial_\infty T_i\hookrightarrow\partial_\infty T_{ij}\hookrightarrow \partial_\infty T$. 

Given distinct $i,j\in\{1,2,3\}$, the following three assertions follow from the same argument as above, using the fact that all subgroupoids $\cala_i$ and $\cala_{ij}$ are amenable and have nontrivial restrictions to any Borel subset of positive measure:  
\begin{enumerate}
\item there exists an $(\cala_i,\rho)$-invariant Borel map $\alpha_i:Y\to\calp_{\le 2}(\partial_\infty T_i)$;
\item there exists an $(\cala_{ij},\rho)$-invariant Borel map $\beta_{ij}:Y\to\calp_{\le 2}(\partial_\infty T_{ij})$ (in particular $\beta_{ij}$ is both $(\cala_i,\rho)$-invariant and $(\cala_j,\rho)$-invariant);
\item given any Borel subset $U\subseteq Y$ of positive measure, the union of any two $((\cala_i)_{|U},\rho)$-invariant Borel maps $U\to\calp_{\le 2}(\partial_\infty T)$ again takes its values (essentially) in $\calp_{\le 2}(\partial_\infty T)$.
\end{enumerate}
The third point ensures that $\alpha_i\cup\beta_{ij}$ takes (essentially) its values in $\calp_{\le 2}(\partial_\infty T_{ij})$. As $\partial_\infty T_i\cap\partial_\infty T_j=\emptyset$ whenever $i\neq j$ (viewed as subsets of $\partial_\infty T_{ij}$), we can find a Borel partition $Y=Y_i\cup Y_j$ such that 
\begin{enumerate}
\item for a.e.\ $y\in Y_i$, the set $\alpha_i(y)\cup\beta_{ij}(y)$ has cardinality exactly $2$; 
\item for a.e.\ $y\in Y_j$, the set $\alpha_j(y)\cup\beta_{ij}(y)$ has cardinality exactly $2$.
\end{enumerate}
In particular, we can find a Borel subset $U\subseteq Y$ of positive measure, and $i_{12}\in\{1,2\}$, such that there exists an $((\cala_{i_{12}})_{|U},\rho)$-invariant Borel map $U\to\calp_{=2}(\partial_\infty T_{12})$. Using the third assertion above, this implies that any $((\cala_{i_{12}})_{|U},\rho)$-invariant Borel map $U\to\calp_{\le 2}(\partial_\infty T)$ must take its values in $\calp_{\le 2}(\partial_\infty T_{12})$. Likewise, up to restricting $U$ to a further positive measure Borel subset, we find $i_{13}\in\{1,3\}$ and $i_{23}\in\{2,3\}$ such that every $((\cala_{i_{13}})_{|U},\rho)$-invariant Borel map $U\to\calp_{\le 2}(\partial_\infty T)$ must take its values in $\calp_{\le 2}(\partial_\infty T_{13})$, and every $((\cala_{i_{23}})_{|U},\rho)$-invariant Borel map $U\to\calp_{\le 2}(\partial_\infty T)$ must take its values in $\calp_{\le 2}(\partial_\infty T_{23})$. As $\partial_\infty T_{12}\cap\partial_\infty T_{13}\cap\partial_\infty T_{23}=\emptyset$, we have reached a contradiction to the existence of an $(\cala_{|U},\rho)$-invariant Borel map $U\to\calp_{\le 2}(\partial_\infty T)$. 
\end{proof}

\begin{cor}\label{cor:free-2}
Let $A_1,A_2,A_3$ be finitely generated free groups, and let $\calg$ be a measured groupoid over a standard finite measure space $Y$, equipped with a strict cocycle $\rho:\calg\to A_1\ast A_2\ast A_3$ with trivial kernel. Let $\cala$ be an amenable measured subgroupoid of $\calg$.

Then there exists a Borel partition $Y=Y_1\dunion Y_2\dunion Y_3$ such that for every $i\in\{1,2,3\}$, the groupoid $(\cala\cap\rho^{-1}(A_i))_{|Y_i}$ is stably trivial.
\end{cor}

\begin{proof}
For every $i\in\{1,2,3\}$, let $\cala_i=\cala\cap\rho^{-1}(A_i)$, and let $Y_i$ be a Borel subset of maximal measure such that $(\cala_i)_{|Y_i}$ is stably trivial (this exists because if $(Y_{i,n})$ is a measure-maximizing sequence of such sets, then $\cala_i$ is still stably trivial in restriction to their countable union). It is enough to show that $Y=Y_1\cup Y_2\cup Y_3$ (up to null sets). Otherwise, there would exist a Borel subset $U\subseteq Y$ of positive measure such that for every Borel subset $V\subseteq U$ of positive measure, none of the three subgroupoids $(\cala_1)_{|V},(\cala_2)_{|V},(\cala_3)_{|V}$ is trivial. This contradicts Lemma~\ref{lemma:free} applied to the ambient groupoid $\calg_{|U}$.
\end{proof}

\subsection{Support of a groupoid with a cocycle to a right-angled Artin group}

In the present section, we consider the following general situation: $G$ is a countable group, $\calg$ is a measured groupoid over a finite measure space $Y$, coming with a cocycle $\rho:\calg\to G$, and $\mathbb{P}$ is an \emph{admissible} collection of subgroups of $G$, i.e.\ $\mathbb{P}$ is countable, conjugation-invariant, stable under intersections, $G\in \mathbb {P}$, and there does not exist any infinite descending chain $P_1\supsetneq P_2\supsetneq\cdots$ in $\mathbb{P}$. The group $G$ acts on $\mathbb{P}$ by conjugation. We aim at defining a notion of support for the pair $(\calg,\rho)$, which can be viewed as a canonical assignment of a subgroup in $\mathbb{P}$ to (almost) every point $y\in Y$. In the sequel of the paper, this will always be applied with $G=G_\Gamma$ a right-angled Artin group, and $\mathbb{P}$ the collection of parabolic subgroups with respect to a fixed standard generating set of $G_\Gamma$, see Remark~\ref{rk:support-raag} below. 

\begin{de}
Let $G$ be a countable group, let $\mathcal{G}$ be a measured groupoid over a standard finite measure space $Y$, and let $\rho:\calg\to G$ be a strict cocycle. Let $\mathbb{P}$ be an admissible collection of subgroups of $G$, and let $P\in\mathbb{P}$. We say that $(\calg,\rho)$ is \emph{tightly $P$-supported with respect to $\mathbb{P}$} if the following two conditions hold:
\begin{enumerate}
\item there exists a conull Borel subset $Y^*\subseteq Y$ such that $\rho(\calg_{|Y^*})\subseteq P$;
\item there does not exist a Borel subset $U\subseteq Y$ of positive measure and a subgroup $Q\in\mathbb{P}$ with $Q\subsetneq P$ such that $\rho(\calg_{|U})\subseteq Q$.
\end{enumerate}
\end{de}

A subgroup $P\in\mathbb{P}$ such that $(\calg,\rho)$ is tightly $P$-supported  with respect to $\mathbb{P}$, if it exists, is unique: this follows from the fact that $\mathbb{P}$ is assumed to be stable under intersections. The existence of $P$ is ensured up to a countable partition of the base space $Y$ by the following lemma.

\begin{lemma}\label{lemma:support}
Let $G$ be a countable group, let $\calg$ be a measured groupoid over a standard finite measure space $Y$, and let $\rho:\calg\to G$ be a strict cocycle. Let $\mathbb{P}$ be an admissible collection of subgroups of $G$. 

Then there exists a partition $Y=\dunion_{i\in I}Y_i$ into at most countably many Borel subsets such that for every $i\in I$, there exists $P_i\in\mathbb{P}$ such that $(\calg_{|Y_i},\rho)$ is tightly $P_i$-supported with respect to $\mathbb{P}$.
\end{lemma}

\begin{proof}
For every $P\in\mathbb{P}$, we claim that there exists a Borel subset $U_P\subseteq Y$ such that 
\begin{enumerate}
\item there exists a conull Borel subset $U_P^*\subseteq U_P$ and a partition $U_P^*=\dunion_{j\in J} U_j$ into at most countably many Borel subsets, such that for every $j\in J$, one has $\rho(\calg_{|U_j})\subseteq P$;
\item the subset $U_P$ is maximal up to measure $0$ with respect to the above property (i.e.\ if $V_P$ is another Borel subset of $Y$ that satisfies this property, then $V_P\setminus U_P$ is a null set).
\end{enumerate}
Indeed, this is shown by taking $U_P$ of maximal (finite) measure satisfying Property~1 (this is possible because if $(U_{P,k})_{k\in\mathbb{N}}$ is a maximizing sequence for the measure, then the countable union of all sets $U_{P,k}$ again satisfies the property).

Now, for every $P\in\mathbb{P}$, we let $Y_P=U_P\setminus\bigcup_{Q\subsetneq P} U_Q$, where the union is taken over all subgroups $Q\in\mathbb{P}$ that are properly contained in $P$: this is a Borel subset of $Y$ because $\mathbb{P}$ is countable. By construction, there exists a partition $Y_P=\dunion_{j\in J}Y_{P,j}$ into at most countably many Borel subsets such that $(\calg_{|Y_{P,j}},\rho)$ is tightly $P$-supported for every $j\in J$. We finally observe that the sets $Y_P$ form a countable partition of $Y$ (up to null sets). This is shown by observing that the set $\{P\in\mathbb{P}|y\in U_P\}$ depends measurably on $y$, is almost everywhere nonempty (as $Y=U_G$ up to null sets), and using the fact that $\mathbb{P}$ does not contain any infinite descending chain, we see that this set almost surely has a minimum $P_y$. Then for almost every $y\in Y$, we have $y\in Y_{P_y}$. Finally, the pairwise intersections of the sets $Y_P$ are null sets because given $P_1,P_2\in\mathbb{P}$, one has $U_{P_1}\cap U_{P_2}=U_{P_1\cap P_2}$ (up to null sets).
\end{proof}

We finally establish a canonicity statement for the support of a groupoid, coming in the form of its invariance by normalizers. 

\begin{lemma}\label{lemma:support-invariant-normal}
Let $G$ be a countable group, let $\calg$ be a measured groupoid over a standard finite measure space $Y$, and let $\rho:\calg\to G$ be a strict cocycle. Let $\mathbb{P}$ be an admissible collection of subgroups of $G$. Let $\calh$ and $\calh'$ be two measured subgroupoids of $\calg$. Assume that there exists $P\in\mathbb{P}$ such that $(\calh,\rho)$ is tightly $P$-supported with respect to $\mathbb{P}$. Assume also that $\calh$ is normalized by $\calh'$.

Then $P$ is $(\calh',\rho)$-invariant: in other words, denoting by $N_G(P)$ the normalizer of $P$ in $G$, there exists a conull Borel subset $Y^*\subseteq Y$ such that $\rho(\calh'_{|Y^*})\subseteq N_G(P)$. 
\end{lemma}

\begin{proof}
As $\calh$ is normalized by $\calh'$, there exists a covering of $\calh'$ by countably many Borel subsets $B_k$ that all leave $\calh$ invariant, with $s(B_k)$ Borel. Up to passing to a further covering, we can assume that for every $k\in\mathbb{N}$, the $\rho$-image of $B_k$ is constant, with value an element $g_k\in G$. Up to replacing $Y$ by a conull Borel subset, we will also assume that $s(B_k)$ has positive measure for every $k\in\mathbb{N}$.

 As a consequence of normality, there exists a conull Borel subset $Y^*\subseteq Y$ such that for every $k\in\mathbb{N}$, one has $\rho(\calh_{|s(B_k)\cap Y^*})\subseteq P\cap g_k^{-1}Pg_k$. As $(\calh,\rho)$ is tightly $P$-supported and $P\cap g_k^{-1}Pg_k$ belongs to $\mathbb{P}$, it follows that $P\subseteq g_k^{-1}Pg_k$. Likewise, up to replacing $Y^*$ by a further conull Borel subset, for every $k\in\mathbb{N}$, one has $\rho(\calh_{|r(B_k)\cap Y^*})\subseteq P\cap g_kPg_k^{-1}$. It follows that $P\subseteq g_kPg_k^{-1}$. Combining the above two inclusions, we derive that $P=g_kPg_k^{-1}$, i.e.\ $g_k\in N_G(P)$. As this is true for every $k$ and the subsets $B_k$ cover $\calh'$, we deduce that $\rho(\calh'_{|Y^{\ast}}) \subseteq N_G(P)$, which concludes our proof. 
\end{proof}

\begin{rk}\label{rk:support-raag}
From now on, we will let $G=G_\Gamma$ be a right-angled Artin group with a standard generating set $S=V\Gamma$, and let $\mathbb{P}$ be the set of all parabolic subgroups of $G$ (which is admissible by Lemma~\ref{lemma:parabolics}). In what follows, parabolic subgroups always refer to parabolic subgroups with respect to $S$. In fact, when $\Gamma$ is transvection-free (as we will ultimately assume in this section), the collection of parabolic subgroups does not depend on the choice of $S$. When $\calg$ is a measured groupoid coming with a cocycle $\rho:\calg\to G$, and $P\in\mathbb{P}$, we will simply say that $(\calg,\rho)$ is \emph{tightly $P$-supported} if it is tightly $P$-supported with respect to the collection $\mathbb{P}$ of all parabolic subgroups.
\end{rk}

\subsection{Taking advantage of amenable normalized subgroupoids}

In this section, we set up arguments towards the groupoid-theoretic version of Property~$(*)$ from the introduction. Following an argument of Adams \cite{Ada}, we will exploit the Borel amenability of the action of a right-angled Artin group on the Roller boundary of the universal cover of its Salvetti complex (\cite[Theorem~5.11]{Duc}, building on \cite{BCGNW}) in a crucial way. We start by reviewing facts about Borel amenability of group actions.

\subsubsection{Review on Borel amenability of group actions}\label{subsubsec:review}

Let $G$ be a countable group, and let $X$ be a  standard Borel space equipped with a $G$-action by Borel automorphisms. The $G$-action on $X$ is \emph{Borel amenable} if there exists a sequence of Borel maps $\nu_n:X\to\Prob(G)$ (where $\Prob(G)$ is equipped with the topology of pointwise convergence) such that for every $x\in X$ and every $g\in G$, one has $||\nu_n(gx)-g\nu_n(x)||_1\to 0$ as $n$ goes to $+\infty$. We will need the following facts about Borel amenability of group actions.
\begin{enumerate}
\item A countable group $G$ is amenable if and only if the trivial $G$-action on a point is Borel amenable.
\item If the $G$-action on $X$ is Borel amenable, then so is its restriction to every $G$-invariant Borel subset of $X$.
\item Let $G$ be a countable group, and let $X$ be a Borel space equipped with a $G$-action by Borel automorphisms. Let $\calp_{\le 2}(X)$ be the set of all nonempty subsets of $X$ of cardinality at most $2$, equipped with its natural structure of standard Borel space coming from that of $X$. Then the $G$-action on $\calp_{\le 2}(X)$ is Borel amenable. This is an easy consequence of the definition, see e.g.\ \cite[Lemma~6.14]{HH}.
\item Let $G_1,\dots,G_k$ be countable groups, and let $X_1,\dots,X_k$ be standard Borel spaces such that for every $i\in\{1,\dots,k\}$, the space $X_i$ comes equipped with a Borel amenable $G_i$-action by Borel automorphisms. Then the product action of $G=G_1\times\dots\times G_k$ on $X=X_1\times\dots\times X_k$ is Borel amenable. Indeed, an elementary computation shows that if $\nu_n^i:X_i\to\Prob(G_i)$ is a sequence of probability measures that witnesses the Borel amenability of the $G_i$-action on $X_i$, then the probability measures $\nu_n$ defined by letting $\nu_n(x_1,\dots,x_k)(g_1,\dots,g_k)=\nu_n^1(x_1)(g_1)\times\dots\times\nu_n^k(x_k)(g_k)$ witness the Borel amenability of the $G$-action on $X$. 
\end{enumerate}
We also refer the reader to \cite[Section~2]{GHL} for a discussion on the relation to other notions of amenability of group actions -- in particular, Borel amenability of a group action implies its \emph{universal amenability} (i.e.\ Zimmer amenability with respect to every quasi-invariant probability measure), another notion that is commonly used in the literature (in particular this is the notion used by Kida in \cite{Kid-memoir}).

\subsubsection{Using amenable normalized subgroupoids}

\begin{lemma}\label{lemma:amenable-normal}
Let $G$ be a right-angled Artin group. Let $\calg$ be a measured groupoid over a standard finite measure space $Y$, and let $\cala$ and $\calh$ be measured subgroupoids of $\calg$. Let $\rho:\calg\to G$ be a strict cocycle, such that $\rho_{|\calh}$ has trivial kernel. Let $P\in\mathbb{P}$ be a parabolic subgroup, and assume that $(\cala,\rho)$ is tightly $P$-supported. Assume that $\cala$ is amenable and normalized by $\calh$. 

Then $\calh\cap\rho^{-1}(P)$ is amenable. If additionally $P^{\perp}$ is amenable, then $\calh$ is amenable. 
\end{lemma}

\begin{rk}
If $P=\{1\}$, the conclusion is clear. Indeed, as $\rho_{|\calh}$ has trivial kernel, $\calh\cap\rho^{-1}(P)$ is then trivial (whence amenable). In addition, amenability of $P^{\perp}$ is then equivalent to the amenability of $G$, and if $G$ is amenable, then $\calh$ is also amenable since $\rho_{|\calh}$ has trivial kernel. 

We may therefore as well assume that $P\neq\{1\}$. Note that in this case $\cala$ is not trivial, even after restricting it to a positive measure Borel subset: indeed, if $\cala_{|U}$ were trivial for some Borel subset $U\subseteq Y$ of positive measure, then $(\cala,\rho)$ would not be tightly $P$-supported by definition.
\end{rk}

\begin{proof}
The proof relies on an argument that dates back to work of Adams \cite{Ada}. Up to replacing $Y$ by a conull Borel subset, we can assume that $\rho(\cala)\subseteq P$. As $\cala$ is normalized by $\calh$, using Lemma~\ref{lemma:support-invariant-normal} together with the fact that the normalizer of $P$ is $P\times P^{\perp}$, we can also assume that $\rho(\calh)\subseteq P\times P^\perp$. 

Let $P=P_1\times\dots\times P_k\times\mathbb{Z}^m$ be the de Rham decomposition of $P$. For every $i\in\{1,\dots,k\}$, we let $\rho_i:\cala\to P_i$ be the cocycle obtained by postcomposing $\rho$ with the $i^{\text{th}}$ projection. For every $i\in\{1,\dots,k\}$, let $\partial_R \widetilde{S}_i$ be the Roller boundary of the universal cover $\widetilde{S}_i$ of the Salvetti complex of $P_i$, and let $\partial_\reg \widetilde{S}_i$ be the subspace of $\partial_R \widetilde{S}_i$ made of regular points. 

Let $i\in\{1,\dots,k\}$. As $\cala$ is amenable and $\partial_R \widetilde{S}_i$ is compact and metrizable, by \cite[Proposition~4.14]{Kid-survey}, there exists an $(\cala,\rho_i)$-invariant Borel map $\theta:Y\to\Prob(\partial_R \widetilde{S}_i)$ (recall that $\Prob(\partial_R\widetilde{S}_i)$ is equipped with the weak-$*$ topology, coming from the duality with the space of continuous functions on $\partial_R\widetilde{S}_i$). 

We claim that for a.e.\ $y\in Y$, the probability measure $\theta(y)$ gives full measure to the Borel subset $\partial_\reg \widetilde{S}_i$. Indeed, otherwise, there exists a Borel subset $V\subseteq Y$ of positive measure for which we can find an $(\cala_{|V},\rho_i)$-invariant Borel map $V\to\Prob(\partial_R\widetilde{S}_i\setminus\partial_\reg\widetilde{S}_i)$. Let $\mathbb{P}_J(P_i)$ be the set of all nontrivial parabolic subgroups of $P_i$ that are contained in some parabolic subgroup of $P_i$ whose type decomposes nontrivially as a join. 
By Corollary~\ref{cor:roller}, there is a $P_i$-equivariant Borel map $\partial_R\widetilde{S}_i\setminus\partial_\reg\widetilde{S}_i\to \mathbb{P}_J(P_i)$. Therefore, we get an $(\cala_{|V},\rho_i)$-invariant Borel map $V\to\Prob(\mathbb{P}_J(P_i))$. Let $\calf(\mathbb{P}_J(P_i))$ be the set of all nonempty finite subsets of $\mathbb{P}_J(P_i)$. As $\mathbb{P}_J(P_i)$ is countable, there is also a $P_i$-equivariant Borel map $\Prob(\mathbb{P}_J(P_i))\to\calf(\mathbb{P}_J(P_i))$, sending a probability measure $\eta$ to the finite set of all elements of $\mathbb{P}_J(P_i)$ of maximal $\eta$-measure. In summary, we obtain an $(\cala_{|V},\rho_i)$-invariant Borel map $V\to\calf(\mathbb{P}_J(P_i))$. By Lemma~\ref{lemma:setwise vs pointwise}, the setwise $P_i$-stabilizer of every finite subset of $\mathbb{P}_J(P_i)$ coincides with its pointwise stabilizer. By restricting to a positive measure Borel subset $W\subseteq V$ where the above map is constant, we thus get an $(\cala_{|W},\rho_i)$-invariant parabolic subgroup $Q'_i\in\mathbb{P}_J(P_i)$, i.e.\ $\rho_i(\cala_{|W})\subseteq Q'_i\times (Q'_i)^{\perp_i}$ (where $\perp_i$ denotes the orthogonality relation taken with respect to the ambient group $P_i$).

The fact that $Q'_i\in\mathbb{P}_J(P_i)$ means that there exists a proper parabolic subgroup $Q_i\subset P_i$ such that $Q'_i\subseteq Q_i\times Q_i^{\perp_i}$, with both $Q_i$ and $Q_i^{\perp_i}$ nontrivial.
We now claim that $Q'_i\times (Q'_i)^{\perp_i}\subsetneq P_i$. Indeed, if $(Q'_i)^{\perp_i}$ is trivial, then $Q'_i$ is contained in the join subgroup $Q_i\times Q_i^{\perp_i}$ which is properly contained in $P_i$ as the type of $P_i$ does not split as a nontrivial join. If $(Q'_i)^{\perp_i}$ is nontrivial, then $Q'_i\times (Q'_i)^{\perp_i}$ is proper for the same reason. We then have $\rho(\cala_{|W})\subseteq (Q'_i\times (Q'_i)^{\perp_i})\times\prod_{j\neq i} P_j\times \mathbb{Z}^m,$ which is a proper parabolic subgroup of $P$. This contradicts the fact that $(\cala,\rho)$ is tightly $P$-supported.  

We next claim that for a.e.\ $y\in Y$, the support of the probability measure $\theta(y)$ has cardinality at most $2$. Indeed, otherwise, there exists a Borel subset $U\subseteq Y$ of positive measure such that for all $y\in U$, the probability measure $\theta(y)\otimes\theta(y)\otimes\theta(y)$ gives positive measure to the space $(\partial_\reg \widetilde{S}_i)^{(3)}$ made of pairwise distinct triples. Thus, we get an $(\cala_{|U},\rho_i)$-invariant Borel map $U\to \Prob((\partial_\reg \widetilde{S}_i)^{(3)})$. By work of Fern\'os, Lécureux and Mathéus \cite[Lemmas~5.14 and~6.21]{FLM}, there exists a $P_i$-equivariant Borel map $(\partial_\reg \widetilde{S}_i)^{(3)}\to V(\widetilde{S}_i)$ -- where $V(\widetilde{S}_i)$ denotes the vertex set of $\widetilde{S}_i$. 
This yields an $(\cala_{|U},\rho_i)$-invariant Borel map $U\to \Prob(V(\widetilde{S}_i))$. Using the fact that $V(\widetilde{S}_i)$ is countable, we deduce as above an $(\cala_{|U},\rho_i)$-invariant Borel map $\varphi:U\to\calf(V(\widetilde{S}_i))$. Then there exists a Borel subset $W\subseteq U$ of positive measure (in restriction to which $\varphi$ is a constant map), such that $\rho_i(\cala_{|W})=\{1\}$, which is a contradiction to the fact that $(\cala,\rho)$ is tightly $P$-supported, as above.

In fact, the above argument shows that for every $(\cala,\rho_i)$-invariant Borel map $\eta:Y\to\Prob(\partial_\reg \widetilde{S}_i)$ and a.e.\ $y\in Y$, the support of the probability measure $\eta(y)$ has cardinality at most $2$ (this observation will be important when applying Adams' argument right below; it corresponds to Adams' \emph{bipolarity} assumption \cite[Definition~3.1]{Ada}). 

Using an argument due to Adams \cite[Lemmas~3.2 and~3.3]{Ada}, we get a Borel map $Y\to \calp_{\le 2}(\partial_\reg \widetilde{S}_i)$ which is both $(\cala,\rho_i)$-invariant and $(\calh,\rho_i)$-invariant. By combining these maps, we derive an $(\calh,\rho)$-invariant Borel map $$Y\to \calp_{\le 2}(\partial_\reg \widetilde{S}_1)\times\dots\times \calp_{\le 2}(\partial_\reg \widetilde{S}_k).$$ 

By \cite[Theorem~5.11]{Duc}, for every $i\in\{1,\dots,k\}$, the action of $P_i$ on $\partial_R \widetilde{S}_i$ is Borel amenable. Therefore, so are the $P_i$-actions on $\partial_\reg \widetilde{S}_i$ and on $\calp_{\le 2}(\partial_\reg\widetilde{S}_i)$ -- we could alternatively have applied a theorem of Nevo and Sageev \cite[Theorem~7.2]{NS} here. We have a product action of $P=P_1\times\dots\times P_k\times \mathbb{Z}^m$ on $\calp_{\le 2}(\partial_\reg \widetilde{S}_1)\times\dots\times \calp_{\le 2}(\partial_\reg \widetilde{S}_k)\times\{*\}$, where the $\mathbb{Z}^m$ factor is acting trivially on the point $\ast$, and also trivially on the other factors. Combining the fact that the trivial action of $\mathbb{Z}^m$ on a point is Borel amenable, with Fact~4 from Section~\ref{subsubsec:review}, we deduce that the action of $P$ on $\calp_{\le 2}(\partial_\reg \widetilde{S}_1)\times\dots\times \calp_{\le 2}(\partial_\reg \widetilde{S}_k)$ is Borel amenable. As $\rho:\calh \cap\rho^{-1}(P)\to P_1\times\dots\times P_k\times\mathbb{Z}^m$ has trivial kernel, it follows from \cite[Proposition~4.33]{Kid-memoir} (see also \cite[Proposition~3.38]{GH}) that $\calh \cap\rho^{-1}(P)$ is amenable. This shows that the first conclusion of the lemma holds. We deduce the second conclusion in a similar way from the amenability of the action of $P\times P^\perp$ on $\calp_{\le 2}(\partial_\reg \widetilde{S}_1)\times\dots\times \calp_{\le 2}(\partial_\reg \widetilde{S}_k)\times\{*\}$ whenever $P^\perp$ is amenable. 
\end{proof}

We now add a maximality condition. 

\begin{lemma}\label{lemma:amenable-normal-maximal}
Let $G$ be a transvection-free right-angled Artin group. Let $\calg$ be a measured groupoid over a standard finite measure space $Y$, and let $\rho:\calg\to G$ be a strict action-type cocycle. Let $\cala$ and $\calh$ be measured subgroupoids of $\calg$. Let $P\in\mathbb{P}$ be a parabolic subgroup. Assume that 
\begin{enumerate}
\item $\cala$ is amenable, and $(\cala,\rho)$ is tightly $P$-supported;
\item $\calh$ is everywhere nonamenable;
\item $\cala$ is normalized by $\calh$;
\item if $\calh'$ is a measured subgroupoid of $\calg$ which is everywhere nonamenable and normalizes an amenable subgroupoid of $\calg$ of infinite type, and if $\calh\subseteq\calh'$, then $\calh$ is stably equivalent to $\calh'$. 
\end{enumerate}
Then $P$ either has trivial center or is isomorphic to $\mathbb{Z}$. 
\end{lemma}

\begin{rk}
In the sequel, this will be applied in the weaker form that $P$ is either isomorphic to $\mathbb{Z}$ or nonamenable.
\end{rk}

\begin{proof}
Assume that $P$ has nontrivial center. By the centralizer theorem in \cite{Ser}, we can then find a cyclic parabolic subgroup $Z\subseteq P$ contained in the center of $P$, and we aim to prove that $P=Z$.  

As  $P$ is $(\calh,\rho)$-invariant (Lemma~\ref{lemma:support-invariant-normal}), up to replacing $Y$ by a conull Borel subset, we can assume that $\rho(\calh)$ is contained in the normalizer of $P$, which is equal to $P\times P^{\perp}$ (Proposition~\ref{prop:normalizer}). Moreover $P\times P^{\perp}$ is nonamenable: otherwise, as $\rho$ has trivial kernel, the groupoid $\rho^{-1}(P\times P^{\perp})$ would be amenable, and $\calh$, as a subgroupoid of $\rho^{-1}(P\times P^{\perp})$, would also be amenable. 

Let $P'$ be the centralizer of $Z$ (equal to $Z\times Z^{\perp}$). Then $P'$ is a parabolic subgroup of $G$ that contains $P\times P^\perp$; in particular it is nonamenable. Let $\calh'=\rho^{-1}(P')$. Then $\calh\subseteq\calh'$ and these subgroupoids are everywhere nonamenable. In addition $\calh'$ normalizes $\rho^{-1}(Z)$, an amenable subgroupoid of $\calg$ of infinite type. The maximality assumption thus implies that $\calh$ is stably equivalent to $\calh'$. Using that $\rho$ is action-type, we deduce in particular that each element in $P'$ has a non-zero power which is contained in $P\times P^\perp$, hence $P'\subseteq P\times P^{\perp}$ by Lemma~\ref{lemma:setwise vs pointwise}. Thus $P'=P\times P^\perp$.

We now prove that $P$ is abelian. Otherwise, it contains a nonabelian free group, so Lemma~\ref{lemma:everywhere-nonamenable} implies that $\rho^{-1}(P)$ is everywhere nonamenable. On the other hand, since $\rho$ is action-type, it has trivial kernel, and Lemma~\ref{lemma:amenable-normal} thus implies that $\calh\cap \rho^{-1}(P)$ is amenable. But $\calh\cap\rho^{-1}(P)$ is stably equivalent to $\calh'\cap\rho^{-1}(P)=\rho^{-1}(P)$, so we get a contradiction. 

As $P$ is abelian and $P\times P^{\perp}$ is equal to the centralizer of $Z$, the transvection-free assumption implies that $P=Z$: otherwise, there would exist a transvection that multiplies a generator of $Z$ by a nontrivial element of $P$. This concludes our proof. 
\end{proof}

\subsection{A uniqueness statement}

\begin{lemma}\label{lemma:uniqueness}
Let $G$ be a transvection-free right-angled Artin group. Let $\calg$ be a measured groupoid over a standard finite measure space $Y$, and let $\rho:\calg\to G$ be a strict action-type cocycle. Let $Z,Z'\in\mathbb{P}$ be two cyclic parabolic subgroups of $G$. 

If there exists a Borel subset $U\subseteq Y$ of positive measure such that the $(\calg_{|U},\rho)$-stabilizer of $Z$ is contained in the $(\calg_{|U},\rho)$-stabilizer of $Z'$, then $Z=Z'$.
\end{lemma}

\begin{proof}
We will prove the contrapositive statement. Assume that $Z\neq Z'$, and let $U\subseteq Y$ be a Borel subset of positive measure. As $G$ is transvection-free, Lemma~\ref{lemma:transvection-free} ensures that there exists an infinite order element $g$ contained in $Z\times Z^\perp$ but not in $Z'\times (Z')^{\perp}$. As $\rho$ is action-type, the groupoid $(\rho^{-1}(\langle g\rangle))_{|U}$ is of infinite type, and it is contained in the $(\calg_{|U},\rho)$-stabilizer of $Z$ but not in the $(\calg_{|U},\rho)$-stabilizer of $Z'$.
\end{proof}

\subsection{Characterization of vertex stabilizers of the extension graph}\label{sec:vertex}

Given a measured groupoid $\calg$ over a base space $Y$, we say that a measured subgroupoid $\calh$ is \emph{stably maximal} (among subgroupoids of $\calg$) with respect to a property $(\mathsf{P})$ if for every measured subgroupoid $\calh'$ of $\calg$ satisfying $(\mathsf{P})$, if $\calh$ is stably contained in $\calh'$, then $\calh$ is stably equivalent to $\calh'$. 

The following key proposition gives a groupoid-theoretic characterization of vertex stabilizers of the extension graph; Property~2.(a) should be viewed as the groupoid-theoretic analogue of Property~$(*)$ from the introduction. Some intuition on Property~2.(b) coming from the group-theoretic setting is also provided in the introduction, under the heading \emph{A word on the proof of the classification theorem}.

\begin{prop}
	\label{prop:raag}
Let $G$ be a non-cyclic transvection-free right-angled Artin group. Let $\calg$ be a measured groupoid over a standard finite measure space $Y$, and let $\rho:\calg\to G$ be a strict action-type cocycle. Let $\calh\subseteq\calg$ be a measured subgroupoid. Then the following conditions are equivalent.
\begin{enumerate}
\item There exist a conull Borel subset $Y^*\subseteq Y$ and a partition $Y^*=\dunion_{i\in I}Y_i$ into at most countably many Borel subsets such that for every $i\in I$, there exists a cyclic parabolic subgroup $Z_i$ of $G$ such that $\calh_{|Y_i}$ is equal to the $(\calg_{|Y_i},\rho)$-stabilizer of $Z_i$, i.e.\  $\calh_{|Y_i}=(\rho^{-1}(Z_i\times Z_i^{\perp}))_{|Y_i}$.
\item The following two conditions hold:
\begin{enumerate}
	\item the subgroupoid $\calh$ is everywhere nonamenable and stably normalizes an amenable subgroupoid of $\calg$ of infinite type, and $\calh$ is stably maximal (among subgroupoids of $\calg$) with respect to these two properties;
	\item for every Borel subset $U\subseteq Y$ of positive measure, whenever $\calh'\subseteq\calg_{|U}$ is an everywhere nonamenable measured subgroupoid which normalizes an amenable subgroupoid $\cala'$ of $\calg_{|U}$ of infinite type with $\cala'\cap\calh_{|U}$ stably trivial, there exists an everywhere nonamenable measured subgroupoid $\calh''\subseteq\calh_{|U}$ such that $\calh'\cap\calh''$ is stably trivial.
\end{enumerate} 
\end{enumerate}
\end{prop}

 In the sequel, a subgroupoid $\calh$ of $\calg$ satisfying one of the equivalent conclusions of Proposition~\ref{prop:raag} will be called a \emph{$(\calg,\rho)$-vertex subgroupoid}: this terminology suggests that it plays the role of a vertex stabilizer of the extension graph $\Gamma^e$ -- where $\Gamma$ is the underlying graph of $G$. 
 
 Notice that Proposition~\ref{prop:raag} gives a purely groupoid-theoretic characterization of $(\calg,\rho)$-vertex subgroupoids, with no reference to the cocycle $\rho$; therefore, being a $(\calg,\rho)$-vertex subgroupoid is a notion that does not depend on the choice of a strict action-type cocycle $\rho:\calg\to G$. Notice also that a $(\calg,\rho)$-vertex subgroupoid $\calh$ naturally comes with a Borel map $\theta:Y\to V\Gamma^e$, by letting $\theta(y)=Z_i$ whenever $y\in Y_i$. Up to measure $0$, this map does not depend on the choice of a Borel partition as in the first conclusion of Proposition~\ref{prop:raag}, in view of the uniqueness statement given in Lemma~\ref{lemma:uniqueness}. We call it the \emph{parabolic map} of $(\calh,\rho)$ -- we insist here that this map \emph{does} depend on $\rho$.

\begin{proof}
We first prove that $1\Rightarrow 2.$ Assume that $\calh$ satisfies assertion~1. Without loss of generality, we can assume that all subsets $Y_i$ in the given partition of $Y^*$ have positive measure.

We first prove that $\calh$ satisfies assertion~2.(a). As $G$ is transvection-free and not isomorphic to $\mathbb{Z}$, for every $i\in I$, the group $Z_i\times Z_i^{\perp}$ contains a nonabelian free subgroup. As $\rho$ is action-type, Lemma~\ref{lemma:everywhere-nonamenable} shows that $\calh_{|Y_i}$ is everywhere nonamenable. Thus $\calh$ is everywhere nonamenable. Let now $\cala$ be a measured subgroupoid of $\calg$ such that for every $i\in I$, one has  $\cala_{|Y_i}=(\rho^{-1}(Z_i))_{|Y_i}$. Then $\cala$ is stably normalized by $\calh$, and as $Z_i$ is cyclic and $\rho$ has trivial kernel the groupoid $\cala$ is amenable -- and it is of infinite type because $\rho$ is action-type. 

Let us now prove the maximality condition from assertion~2.(a). For this, let $\hat{\calh}$ be a subgroupoid of $\calg$ which is everywhere nonamenable, stably normalizes an amenable subgroupoid $\hat{\cala}$ of $\calg$ of infinite type, and such that $\calh$ is stably contained in $\hat{\calh}$. By Lemma~\ref{lemma:support}, for every $i\in I$, we can find a partition $Y_i=\dunion_{j\in J_i} V_{i,j}$ into at most countably many Borel subsets of positive measure such that for every $j\in J_i$, there exists a parabolic subgroup $P_{i,j}$ such that $(\hat\cala_{|V_{i,j}},\rho)$ is tightly $P_{i,j}$-supported, and $P_{i,j}$ is nontrivial since $\hat\cala$ is of infinite type and $\rho$ has trivial kernel. As $\hat\cala$ is stably normalized by $\hat\calh$, up to refining the above partition, we can further assume that for every $j\in J_i$, the groupoid $\hat\cala_{|V_{i,j}}$ is normalized by $\hat\calh_{|V_{i,j}}$. Lemma~\ref{lemma:support-invariant-normal} then implies that up to replacing $Y_i$ by a conull Borel subset, for every $j\in J_i$, one has $\rho(\hat\calh_{|V_{i,j}})\subseteq P_{i,j}\times P_{i,j}^{\perp}$. 
As $\hat\calh$ is everywhere nonamenable, Lemma~\ref{lemma:amenable-normal} implies that all groups $P_{i,j}^{\perp}$ are nonamenable, and $(\hat\calh\cap\rho^{-1}(P_{i,j}))_{|V_{i,j}}$ is amenable.

We now fix $i\in I$ and $j\in J_i$. As $\calh$ is stably contained in $\hat\calh$ and $\rho$ is action-type, we know that for every $g\in Z_i\times Z_i^\perp$, there exists an integer $k\neq 0$ such that $g^k\in P_{i,j}\times P_{i,j}^{\perp}$. Lemma~\ref{lemma:setwise vs pointwise} then shows that $Z_i\times Z_i^\perp\subseteq P_{i,j}\times P_{i,j}^{\perp}$. Since $G$ is transvection-free and $P_{i,j}$ is nontrivial, Lemma~\ref{lemma:transvection-free} ensures that either $P_{i,j}=Z_i$, or else $P_{i,j}\cap Z_i^{\perp}$ is nonamenable. But if the latter holds, then $(\calh\cap\rho^{-1}(P_{i,j}\cap Z_i^\perp))_{|V_{i,j}}$ (which is everywhere nonamenable) is not stably contained in $(\hat\calh\cap\rho^{-1}(P_{i,j}\cap Z_i^{\perp}))_{|V_{i,j}}$ (which is amenable), a contradiction. This shows that $P_{i,j}=Z_i$. 

Since the above is true for every $i\in I$ and every $j\in J_i$, we deduce that $\hat\calh$ is stably contained in $\calh$, so they are stably equivalent. This finishes the proof of assertion~2.(a).  

We now turn to proving assertion~2.(b). Let $U\subseteq Y$ be a Borel subset of positive measure, and $\cala'$ and $\calh'$ be measured subgroupoids of $\calg_{|U}$ as in the statement. For every $i\in I$, we let $U_i=U\cap Y_i$. By Lemma~\ref{lemma:support}, there exists a partition $U_i=\dunion_{n\in N} U_{i,n}$ into at most countably many Borel subsets such that for every $n\in N$, there exists a parabolic subgroup $P_{i,n}$ such that $(\cala'_{|U_{i,n}},\rho)$ is tightly $P_{i,n}$-supported. Up to replacing $U_i$ by a conull Borel subset and refining the above partition, Lemma~\ref{lemma:support-invariant-normal} allows us to assume that for every $n\in N$, we have $\rho(\calh'_{|U_{i,n}})\subseteq P_{i,n}\times P_{i,n}^{\perp}$ (and we can also assume that $\rho(\cala'_{|U_{i,n}})\subseteq P_{i,n}$).

We aim to construct an everywhere nonamenable subgroupoid $\calh''$ of $\calh_{|U}$ such that $\calh'\cap\calh''$ is stably trivial. For this, we will only define $\calh''$ on each $U_{i,n}$ -- and $\calh''$ will have no elements with source and range in different subsets $U_{i,n}$. 

We will first define $\calh''_{|U_{i,n}}$ for values of $i,n$ such that $Z^\perp_i$ is not contained in $P_{i,n}\times P^\perp_{i,n}$. As $G$ is transvection-free, the group $Z_i^\perp$ has no abelian factor in its de Rham decomposition. By Proposition~\ref{prop:kim-koberda}, we can (and shall) thus choose elements $a_{1},a_{2}\in Z_i^{\perp}$ that freely generate a rank $2$ free subgroup, so that the support of every nontrivial element of $\langle a_{1},a_{2}\rangle$ is equal to $Z_i^\perp$. Let $\calh''_{|U_{i,n}}=(\rho^{-1}(\langle a_{1},a_{2}\rangle))_{|U_{i,n}}$. Clearly $\calh''_{|U_{i,n}}\subseteq\calh_{|U_{i,n}}$. Since $a_{1}$ and $a_{2}$ generate a nonabelian free subgroup and $\rho$ is action-type, Lemma~\ref{lemma:everywhere-nonamenable} implies that $\calh''_{|U_{i,n}}$ is everywhere nonamenable. In addition, as $Z^\perp_i$ is not contained in $P_{i,n}\times P^\perp_{i,n}$, it follows that $\rho(\calh''_{|U_{i,n}}\cap \calh'_{|U_{i,n}})$ is the trivial subgroup. As $\rho$ has trivial kernel, $\calh''_{|U_{i,n}}\cap \calh'_{|U_{i,n}}$ is trivial.

We will now define $\calh''_{|U_{i,n}}$ for values of $i,n$ such that $Z^\perp_i\subseteq P_{i,n}\times P^\perp_{i,n}$.  Similar to the proof of Proposition~\ref{prop:kim-koberda}, we have $Z^\perp_i= M_1\times M_2$ where $M_1=Z^\perp_i\cap P_{i,n}$ and $M_2=Z^\perp_i\cap P^\perp_{i,n}$ are parabolic subgroups. We claim that $M_1$ is nontrivial. Otherwise we have $Z^\perp_i\subseteq P^\perp_{i,n}$. Thus $P_{i,n}\subseteq (Z^\perp_i)^\perp=Z_i$, with the last equality following from transvection-freeness. Thus $\cala'_{|U_{i,n}}\subseteq \calh_{|U_{i,n}}$. As $\cala'$ is of infinite type, this contradicts our assumption that $\cala'\cap \calh$ is stably trivial, thereby proving our claim. As $G$ is transvection-free, $Z^\perp_i$ does not have any abelian de Rham factor, and neither does $M_1$, in particular $M_1$ is nonabelian. We can thus find three finitely generated nonabelian free subgroups $A_1,A_2,A_3\subseteq M_1$ such that $\langle A_1,A_2,A_3\rangle=A_1\ast A_2\ast A_3$. Notice that  $(\calh'\cap\rho^{-1}(\langle A_1,A_2,A_3\rangle))_{|U_{i,n}}\subseteq (\calh'\cap\rho^{-1}(M_1))_{|U_{i,n}}\subseteq (\calh'\cap\rho^{-1}(P_{i,n}))_{|U_{i,n}}$. By Lemma~\ref{lemma:amenable-normal}, the groupoid $(\calh'\cap \rho^{-1}(P_{i,n}))_{|U_{i,n}}$ is amenable. Therefore $(\calh'\cap\rho^{-1}(\langle A_1,A_2,A_3\rangle))_{|U_{i,n}}$ is also amenable. By Corollary~\ref{cor:free-2}, we can thus find a Borel partition $U_{i,n}=U_{i,n,1}\dunion U_{i,n,2}\dunion U_{i,n,3}$ such that for every $\ell\in\{1,2,3\}$, the groupoid $(\calh'\cap\rho^{-1}(A_\ell))_{|U_{i,n,\ell}}$ is stably trivial. We then let $\calh''_{|U_{i,n}}$ be a measured subgroupoid of $\calh_{|U_{i,n}}$ such that for every $\ell\in\{1,2,3\}$, one has $\calh''_{|U_{i,n,\ell}}=(\rho^{-1}(A_\ell))_{|U_{i,n,\ell}}$. Then $\calh'_{|U_{i,n}}\cap\calh''_{|U_{i,n}}$ is stably trivial. In addition, as every $A_\ell$ is a nonabelian free group, Lemma~\ref{lemma:everywhere-nonamenable} implies that $\calh''_{|U_{i,n}}$ is everywhere nonamenable. We have thus constructed a subgroupoid $\calh''$ satisfying the conclusion of assertion~2.(b). This concludes  our proof of $1\Rightarrow 2$.

\medskip

We now prove that $2\Rightarrow 1$. Assume that $\calh$ satisfies assertions~2.(a) and~2.(b), and let $\cala$ be an amenable subgroupoid of $\calg$ of infinite type which is stably normalized by $\calh$. Lemma~\ref{lemma:support} gives a partition $Y=\dunion_{i\in I} Y_i$ into at most countably many Borel subsets of positive measure such that for every $i\in I$, there exists a parabolic subgroup $P_i$ such that $(\cala_{|Y_i},\rho)$ is tightly $P_i$-supported, and $P_i$ is nontrivial since $\cala$ is of infinite type and $\rho$ has trivial kernel. Up to refining this partition if needed, we can further assume that for every $i\in I$, the groupoid $\cala_{|Y_i}$ is normalized by $\calh_{|Y_i}$. We will prove that for every $i\in I$, the group $P_i$ is cyclic: by Lemma~\ref{lemma:support-invariant-normal}, the groupoid $\calh_{|Y_i}$ will then be contained in the $(\calg_{|Y_i},\rho)$-stabilizer of $P_i$ (up to replacing $Y_i$ by a conull Borel subset), and the maximality assumption from assertion~2.(a) together with $1\Rightarrow 2.(a)$ will complete the proof.

So let $i\in I$, and assume towards a contradiction that $P_i$ is noncyclic. Lemma~\ref{lemma:amenable-normal} ensures that $P_i^\perp$ is nonamenable, and Lemma~\ref{lemma:amenable-normal-maximal} together with assertion~2.(a) ensure that $P_i$ is nonamenable (notice that if we replace each groupoid appearing in this assertion by its restriction on $Y_i$, then the assertion still holds). Up to replacing $Y_i$ by a conull Borel subset, we can assume that $\rho(\cala_{|Y_i})\subseteq P_i$, and Lemma~\ref{lemma:support-invariant-normal} enables us to further assume that $\rho(\calh_{|Y_i})\subseteq P_i\times P_i^\perp$. Notice also that Lemma~\ref{lemma:amenable-normal} implies that $(\calh\cap\rho^{-1}(P_i))_{|Y_i}$ is amenable.

Let $\calh'=\rho^{-1}(P_i^{\perp})_{|Y_i}$. As $P_i^{\perp}$ contains a nonabelian free subgroup and $\rho$ is action-type, Lemma~\ref{lemma:everywhere-nonamenable} ensures that $\calh'$ is everywhere nonamenable. As $P_i$ is nonamenable, we can find three finitely generated nonabelian free subgroups $A_1,A_2,A_3\subseteq P_i$ such that $\langle A_1,A_2,A_2\rangle=A_1\ast A_2\ast A_3$. By Corollary~\ref{cor:free-2} (applied to the amenable groupoid $(\calh\cap\rho^{-1}(P_i))_{|Y_i}$), we can find a Borel partition $Y_i=Y_{i,1}\dunion Y_{i,2}\dunion Y_{i,3}$ such that for every $\ell\in\{1,2,3\}$, the groupoid $(\calh\cap\rho^{-1}(A_\ell))_{|Y_{i,\ell}}$ is stably trivial. For every $\ell\in\{1,2,3\}$, let $a_\ell\in A_\ell$ be a nontrivial element. Let then $\cala'$ be a measured subgroupoid of $\calg_{|Y_i}$ such that for every $\ell\in\{1,2,3\}$, one has $\cala'_{|Y_{i,\ell}}=\rho^{-1}(\langle a_\ell\rangle)_{|Y_{i,\ell}}$. Then $\cala'$ is amenable and of infinite type (as $\rho$ is action-type), it is stably normalized by $\calh'$, and $\cala'\cap\calh_{|Y_i}$ is stably trivial. 

We now aim to contradict assertion~2.(b) with $U=Y_i$. For that, we let $\calh''$ be a measured subgroupoid of $\calh_{|Y_i}$ such that $\calh'\cap\calh''$ is stably trivial, and we aim to prove that there exists a Borel subset $V\subseteq Y_i$ of positive measure such that $\calh''_{|V}$ is amenable.

As $\calh''\subseteq\calh_{|Y_i}$, one has $\rho(\calh'')\subseteq P_i\times P_i^{\perp}$. We post-compose $\rho:\calh''\to P_i\times P_i^{\perp}$ with the projection $P_i\times P_i^{\perp}\to P_i$ to obtain a  strict cocycle $\rho':\calh''\to P_i$. The kernel of $\rho'$ is the set of all elements $g\in\calh''$ such that $\rho(g)\in P_i^{\perp}$, a subgroupoid of $\calh'$. Being a subgroupoid of $\calh''\cap\calh'$, it is stably trivial. We can therefore find a Borel subset $V\subseteq Y_i$ of positive measure such that $\rho':\calh''_{|V}\to P_i$ has trivial kernel. As $\calh''_{|V}\subseteq\calh_{|V}$, it normalizes $\cala_{|V}$. We can therefore apply Lemma~\ref{lemma:amenable-normal} to the subgroupoid of $\calg_{|V}$ generated by $\calh''_{|V}$ and $\cala_{|V}$, and to its natural cocycle towards $P_i$, to derive that $\calh''_{|V}$ is amenable, as desired.
\end{proof}

\subsection{Characterization of adjacency}

In Section~\ref{sec:vertex}, given a measured groupoid $\calg$ with a strict action-type cocycle towards a transvection-free right-angled Artin group $G$ not isomorphic to $\mathbb{Z}$, we characterized subgroupoids of $\calg$ that stabilize  rank one parabolic subgroups of $G$ (up to a countable Borel partition of the base space). In the present section, given two such subgroupoids, we will characterize when the parabolic subgroups they stabilize commute -- i.e.\ define adjacent vertices of the extension graph. This is the contents of Lemma~\ref{lemma:adjacency-groupoids} below; we start with a group-theoretic version of this statement.

\begin{lemma}
	\label{lem:fix point}
Let $\Gamma$ be a finite simple graph, let $G=G_\Gamma$, and let $\Gamma^e$ be the extension graph of $\Gamma$, equipped with its natural $G$-action. Let $v,w\in V\Gamma^e$ be two vertices.

Then $v$ and $w$ are adjacent or equal in $\Gamma^e$ if and only if there are finitely many vertices $u\in V\Gamma^e$ such that every element of $\Stab_G(v)\cap\Stab_G(w)$ has a nontrivial power that fixes $u$. In this case there are in fact at most $|V\Gamma|$ such vertices.
\end{lemma}

\begin{proof}
Let $\mathbb Z_v$ be the cyclic parabolic subgroup of $G$ associated with a vertex $v\in V\Gamma^e$. Let $g_v\in G$ be a generator of $\mathbb Z_v$. Let $\bar v\in\Gamma$ be the type of $\mathbb Z_v$. Note that $\stab_G(v)$ is the centralizer of $\mathbb Z_v$, a parabolic subgroup of type $\st(\bar v)$ (Proposition~\ref{prop:normalizer}).

We first assume that $d_{\Gamma^e}(v,w)\ge 2$, and aim to show that $H=\Stab_G(v)\cap\Stab_G(w)$ fixes infinitely many vertices of $\Gamma^e$. The group $H$ centralizes both $\mathbb Z_v$ and $\mathbb Z_w$, so $H$ centralizes $\langle \mathbb Z_v,\mathbb Z_w\rangle$. By the proof of \cite[Theorem~1.3]{kim2013embedability}, the group $\langle \mathbb Z_v,\mathbb Z_w\rangle$ contains a nonabelian free subgroup $F$, generated by high enough powers of $g_v$ and $g_w$: in fact, Kim and Koberda prove that there exists an injective homomorphism from $G_\Gamma$ into the mapping class group of a finite-type surface, such that $g_v$ and $g_w$ are mapped to powers of Dehn twists about essential simple closed curves that intersect essentially. Therefore $H$ fixes infinitely many vertices of $\Gamma^e$ arising from conjugates of $\mathbb Z_v$ by elements in $F$. 

We now assume that $v$ and $w$ are equal or adjacent. Actually it suffices to consider the case where they are adjacent. Up to conjugation, we can assume without loss of generality that $g_v=\bar v$ and $g_w=\bar w$ (after identifying vertices of $\Gamma$ with the associated standard generators of $G_\Gamma$). Then $H=\Stab_G(v)\cap\Stab_G(w)=G_{\st(\bar v)}\cap G_{\st(\bar w)}=G_{\bar e\circ \bar e^\perp}$, where $\bar e\subset \Gamma$ is the edge joining $\bar v$ to $\bar w$ (the second equality follows from Proposition~\ref{prop:normalizer}, and the last one from Lemma~\ref{lemma:parabolics}(1)). By Proposition~\ref{prop:normalizer}, we have $N_G(H)=H$ (where $N_G(H)$ denotes the normalizer of $H$ in $G$). Let now $u\in V\Gamma^e$ be a vertex such that every element of $H$ has a power that fixes $u$. Notice that if $u$ is fixed by $h^n$ for some $h\in H$ and $n\in\mathbb{N}$, then $h^n$ centralizes $\mathbb Z_u$, and in fact $h$ centralizes $\mathbb Z_u$ (as follows from Lemma~\ref{lemma:setwise vs pointwise}). Therefore $H$ centralizes $\mathbb Z_u$. Hence $\mathbb Z_u\subseteq N_G(H)=H$. It follows that $\mathbb Z_u$ lies in the center of $H$, which is equal to $G_{\Gamma_1}$, where $\Gamma_1$ is the maximal clique factor of $\bar e\circ \bar e^\perp$ \cite{Ser}. Thus vertices of $\Gamma^e$ fixed by $H$ arise from cyclic parabolic subgroups contained in the abelian group $G_{\Gamma_1}$, and there are only finitely many of those -- in fact at most $|V\Gamma_1|$.
\end{proof}

We now move on to the groupoid-theoretic version of our statement.

\begin{lemma}\label{lemma:adjacency-groupoids}
Let $G$ be a non-cyclic transvection-free right-angled Artin group, with underlying graph $\Gamma$. Let $\calg$ be a measured groupoid over a standard finite measure space $Y$, and let $\rho:\calg\to G$ be a strict action-type cocycle. Let $\calh$ and $\calh'$ be two $(\calg,\rho)$-vertex subgroupoids, with parabolic maps $\theta,\theta':Y\to V\Gamma^e$, respectively. Then the following two assertions are equivalent.
\begin{enumerate}
\item For a.e.\ $y\in Y$, the groups $\theta(y)$ and $\theta'(y)$ commute.
\item There exist finitely many $(\calg,\rho)$-vertex subgroupoids $\calg_1,\dots,\calg_k$, such that for every $(\calg,\rho)$-vertex subgroupoid $\calk$, if $\calh\cap\calh'$ is stably contained in $\calk$, then there exists a Borel partition $Y=W_1\dunion \cdots\dunion W_k$ such that for every $\ell\in\{1,\dots,k\}$, the groupoid $\calk_{|W_\ell}$ is stably contained in $(\calg_\ell)_{|W_\ell}$.
\end{enumerate} 
\end{lemma}

\begin{proof}
We first prove that $1\Rightarrow 2$. We can find a conull Borel subset $Y^*\subseteq Y$ and a partition $Y^*=\dunion_{i\in I} Y_i$ into at most countably many Borel subsets such that for every $i\in I$, the maps $\theta_{|Y_i}$ and $\theta'_{|Y_i}$ are constant, with respective values commuting cyclic parabolic subgroups $Z_i$ and $Z'_i$. Let $k=|V\Gamma|$. By Lemma~\ref{lem:fix point}, for every $i\in I$, there exists a finite set $\calz_i=\{Z_{i,1},\dots,Z_{i,k}\}$ of cyclic parabolic subgroups, of cardinality at most $k$ (written possibly with repetitions), such that for every cyclic parabolic subgroup $Z$, if every element of $\Stab_G(Z_{i})\cap\Stab_G(Z'_i)$ has a power that fixes $Z$, then $Z\in\calz_i$. For every $\ell\in\{1,\dots,k\}$, we let $\calg_\ell$ be a measured subgroupoid of $\calg$ such that for every $i\in I$, the groupoid $(\calg_\ell)_{|Y_i}$ is equal to the $(\calg_{|Y_i},\rho)$-stabilizer of $Z_{i,\ell}$.

Let now $\calk\subseteq\calg$ be a $(\calg,\rho)$-vertex subgroupoid, and assume that $\calh\cap\calh'$ is stably contained in $\calk$. Up to replacing $Y^*$ by a conull Borel subset and refining the above Borel partition of $Y^*$, we can assume that for every $i\in I$, the groupoid $\calk_{|Y_i}$ is equal to the $(\calg_{|Y_i},\rho)$-stabilizer of some rank one parabolic subgroup $P_i$. We can further assume that every element of $\Stab_G(Z_i)\cap\Stab_G(Z'_i)$ has a power contained in $\Stab_G(P_i)$: indeed, since $\rho$ is action-type, each element of $\Stab_G(Z_i)\cap\Stab_G(Z'_i)$ has a power which is the $\rho$-image of an element of $(\calh\cap\calh')_{|Y_i}$. It follows from the above that $P_i\in\calz_i$. Assertion~2 follows.

We now prove that $\neg 1\Rightarrow \neg 2$. Assume that there exists a Borel subset $U\subseteq Y$ of positive measure such that the maps $\theta_{|U}$ and $\theta'_{|U}$ are constant, with values two cyclic parabolic subgroups $Z$ and $Z'$ that do not commute. Let $\calg_1,\dots,\calg_k$ be finitely many $(\calg,\rho)$-vertex subgroupoids. Up to restricting to a further Borel subset of $U$ of positive measure, we can assume that for every $\ell\in\{1,\dots,k\}$, the groupoid $(\calg_\ell)_{|U}$ is equal to the $(\calg_{|U},\rho)$-stabilizer of some rank one cyclic parabolic subgroup $P_\ell$. By Lemma~\ref{lem:fix point}, there exists a cyclic parabolic subgroup $P\notin\{P_1,\dots,P_k\}$ such that $\Stab_G(Z)\cap\Stab_G(Z')\subseteq\Stab_G(P)$. Let then $\calk$ be a measured subgroupoid of $\calg$ such that $\calk_{|U}$ is equal to the $(\calg_{|U},\rho)$-stabilizer of $P$, and $\calk_{|Y\setminus U}=\calh_{|Y\setminus U}$. Then $\calh\cap\calh'$ is stably contained in $\calk$; moreover $\calk$ is a $(\calg,\rho)$-vertex subgroupoid. On the other hand, Lemma~\ref{lemma:uniqueness} implies that for every $\ell\in\{1,\dots,k\}$, there is no Borel subset $V\subseteq U$ of positive measure such that $\calk_{|V}$ is contained in $(\calg_\ell)_{|V}$. This shows that assertion~2 fails to hold. 
\end{proof}

\subsection{Conclusion}

In this section, we complete the proof of our main theorem, which is Theorem~\ref{theointro:main} from the introduction (from which, as already explained, Theorem~\ref{theointro:1} follows using Theorem~\ref{theo:rigidity}).

\begin{theo}\label{theo:main}
Let $G_1$ and $G_2$ be two transvection-free right-angled Artin groups with respective defining graphs $\Gamma_1$ and $\Gamma_2$. If $G_1$ and $G_2$ are measure equivalent, then the extension graphs $\Gamma_1^e$ and $\Gamma_2^e$ are isomorphic.
\end{theo}

 Notice that the only nontrivial transvection-free abelian right-angled Artin group is $\mathbb{Z}$; all the other ones are nonamenable, and therefore not measure equivalent to $\mathbb{Z}$ (see \cite[Corollary~3.2]{Fur-survey}). We will therefore assume without loss of generality that $G_1$ and $G_2$ are not cyclic. Using \cite[Proposition~5.11]{HH} which builds on the seminal work of Furman \cite{Fur2}, to prove Theorem~\ref{theo:main} it suffices to show the following statement for measured groupoids.

\begin{theo}\label{theo:main-2}
Let $G_1$ and $G_2$ be two non-cyclic transvection-free right-angled Artin groups, with respective defining graphs $\Gamma_1$ and $\Gamma_2$. Let $\calg$ be a measured groupoid over a standard finite measure space $Y$ (of positive measure), and assume that for every $i\in\{1,2\}$, there exists a strict action-type cocycle $\calg\to G_i$.

Then the extension graphs $\Gamma_1^e$ and $\Gamma_2^e$ are isomorphic.
\end{theo}

\begin{proof}
Let $Z_1$ be a rank one parabolic subgroup of $G_1$ (which also corresponds to a vertex of $\Gamma_1^e$). The $(\calg,\rho_1)$-stabilizer $\calg_{Z_1}$ of $Z_1$ is a $(\calg,\rho_1)$-vertex subgroupoid. In view of Proposition~\ref{prop:raag} -- which characterizes $(\calg,\rho_1)$-vertex subgroupoids in a purely groupoid-theoretic manner, i.e.\ without reference to the cocycle $\rho_1$ -- it follows that $\calg_{Z_1}$ is also a $(\calg,\rho_2)$-vertex subgroupoid. This means that there exist a conull Borel subset $Y^*\subseteq Y$ and a partition $Y^*=\dunion_{i\in I} Y_i$ into at most countably many Borel subsets of positive measure such that for every $i\in I$, the subgroupoid $(\calg_{Z_1})_{|Y_i}$ is equal to the $(\calg_{|Y_i},\rho_2)$-stabilizer of some rank one parabolic subgroup $Z_{2,i}$ of $G_2$ (i.e.\ a vertex of $\Gamma_2^e$), which is unique in view of Lemma~\ref{lemma:uniqueness}. This allows us to define a Borel map $\theta_{Z_1}:Y^*\to V\Gamma_2^e$, with $\theta_{Z_1}(y)=Z_{2,i}$ whenever $y\in Y_i$ -- in other words $\theta_{Z_1}$ is the parabolic map of $(\calg_{Z_1},\rho_2)$. 

By varying $Z_1$ (and up to replacing $Y^*$ by a conull Borel subset), we then get a Borel map $\theta:V\Gamma_1^e\times Y^*\to V\Gamma_2^e$, defined by letting $\theta(Z_1,y)=\theta_{Z_1}(y)$. We now claim that for a.e.\ $y\in Y^*$, the map $\theta(\cdot,y)$ actually determines a graph isomorphism between $\Gamma_1^e$ and $\Gamma_2^e$, which will conclude our proof.

We first prove that for a.e.\ $y\in Y^*$, the map $\theta(\cdot,y)$ is injective. Otherwise, as $V\Gamma_1^e$ is countable, we can find a Borel subset $U\subseteq Y^*$ of positive measure, two distinct rank one parabolic subgroups $Z_1,Z'_1\in V\Gamma^e_1$, and a rank one parabolic subgroup $Z_2\in V\Gamma_2^e$, such that for all $y\in U$, one has $\theta_{Z_1}(y)=\theta_{Z'_1}(y)=Z_2$. This implies that the $(\calg_{|U},\rho_1)$-stabilizers of $Z_1$ and $Z'_1$ are stably equivalent -- as they are both stably equivalent to the $(\calg_{|U},\rho_2)$-stabilizer of $Z_2$. This contradicts Lemma~\ref{lemma:uniqueness}. 

We now prove that for a.e.\ $y\in Y$, the map $\theta(\cdot,y)$ is surjective. Let $Z_2\in V\Gamma^e_2$. Then the $(\calg,\rho_2)$-stabilizer $\calg_{Z_2}$ of $Z_2$ is a $(\calg,\rho_2)$-vertex stabilizer. By Proposition~\ref{prop:raag}, it is also a $(\calg,\rho_1)$-vertex subgroupoid. This means that there exists a conull Borel subset $Y^{**}$ of $Y^*$ and a partition $Y^{**}=\dunion_{j\in J} Y_j$ such that for every $j\in J$, the subgroupoid $(\calg_{Z_2})_{|Y_j}$ is equal to the $(\calg_{|Y_j},\rho_1)$-stabilizer of some rank one parabolic subgroup $Z_{1,j}\in V\Gamma_1^e$. For every $j\in J$ and a.e.\ $y\in Y_j$, we then have $\theta_{Z_{1,j}}(y)=Z_2$. As $V\Gamma_2^e$ is countable, surjectivity holds.

Finally, the fact that for a.e.\ $y\in Y$, the map $\theta(\cdot,y)$ preserves both adjacency and nonadjacency, follows from Lemma~\ref{lemma:adjacency-groupoids} combined with Proposition~\ref{prop:raag} -- which gives a purely groupoid-theoretic characterization of adjacency, with no reference to the cocycles. This completes our proof.
\end{proof}

\subsection{Application to $W^*$-rigidity}\label{sec:von-neumann}

We finish this section by establishing Corollary~\ref{corintro:von-neumann} from the introduction.

\begin{cor}
	\label{cor:von-neumann}
Let $G_1$ and $G_2$ be two right-angled Artin groups with finite outer automorphism groups. Assume that $G_1$ and $G_2$ have free ergodic probability measure-preserving actions on standard probability spaces whose von Neumann algebras have isomorphic amplifications. 

Then $G_1$ and $G_2$ are isomorphic.
\end{cor}

As mentioned in the introduction, if we assumed that the actions were $W^*$-equivalent (without passing to amplifications), then the corollary would also follow from \cite[Corollary~4.2]{HHL}. It would also follow from recent work of Chifan and Kunnawalkam Elayavalli \cite[Theorem~1.3]{CE}. In fact, the statement of \cite[Theorem~1.3]{CE} does not mention amplifications, but the work in \cite[Theorem~3.1]{CE} also deals with virtual Cartan subalgebras. For the convenience of the reader, we are including a proof of Corollary~\ref{corintro:von-neumann}, explaining how it can be deduced from important work of Popa and Vaes \cite{PV1}. 

\begin{proof}
We will prove in the next paragraph of the proof that $G_1$ and $G_2$ satisfy the property~$\mathrm{(HH)}^+$ introduced by Ozawa and Popa in \cite[Definition~1]{OP}. For now, let us explain why this is enough to prove our corollary. Assume that $G_1\actson X$ and $G_2\actson Y$ are two actions as in the statement, whose associated von Neumann algebras have isomorphic amplifications: $(L^\infty(X)\rtimes G_1)^s\simeq (L^\infty(Y)\rtimes G_2)^t$. As $G_1$ and $G_2$ satisfy property~$\mathrm{(HH)}^+$, a theorem of Popa and Vaes \cite[Theorem~1.2 and Remark~1.3]{PV1} implies that $L^\infty(X)\rtimes G_1$ and $L^\infty(Y)\rtimes G_2$ have a unique \emph{virtual Cartan subalgebra} up to unitary conjugacy: more precisely, $L^{\infty}(X)$ is (up to unitary conjugacy) the unique maximal abelian subalgebra of $L^\infty(X)\rtimes G_1$ whose normalizer is a finite index subfactor of $L^\infty(X)\rtimes G_1$ (and likewise for $L^\infty(Y)$ inside $L^\infty(Y)\rtimes G_2$). By \cite[Proposition~4.12]{OP1}, the amplifications of these von Neumann algebras also have a unique virtual Cartan subalgebra (in the above sense) up to unitary conjugacy. The fact that $(L^\infty(X)\rtimes G_1)^s$ and $(L^\infty(Y)\rtimes G_2)^t$ are isomorphic means that there exist positive measure Borel subsets $X'\subseteq X$ and $Y'\subseteq Y$ such that the restrictions to $X'$ and $Y'$ of the equivalence relations $\mathcal{R}_1,\mathcal{R}_2$ associated to the actions $G_1\actson X,G_2\actson Y$ have isomorphic von Neumann algebras: $L((\mathcal{R}_1)_{|X'})\simeq L((\mathcal{R}_2)_{|Y'})$. The above uniqueness statement implies that this isomorphism can be chosen to send the Cartan subalgebra $L^\infty(X')$ to $L^\infty(Y')$. By a theorem of Feldman and Moore \cite{FM}, this in turn implies that the equivalence relations $(\mathcal{R}_1)_{|X'}$ and $(\mathcal{R}_2)_{|Y'}$ are isomorphic, i.e.\ the actions $G_1\actson X$ and $G_2\actson Y$ are stably orbit equivalent. As two countable groups admit stably equivalent free ergodic probability measure-preserving actions if and only if they are measure equivalent (cf. \cite[Theorem 3.3]{Fur2} or \cite[Theorem~2.3]{Gab2}), we deduce that $G_1$ and $G_2$ are measure equivalent, whence isomorphic by Theorem~\ref{theointro:1}.

We now prove that $G_1$ and $G_2$ satisfy property~$\mathrm{(HH)}^+$. As property~$\mathrm{(HH)}^+$ is stable under direct products \cite[Theorem~2.3(2)]{OP}, it is enough to prove that every irreducible nonamenable right-angled Artin group satisfies property $\mathrm{(HH)}^+$. By \cite[Theorem~2.3(5)]{OP}, it is enough to show that for every hyperplane $h$ of the universal cover of the Salvetti complex of $G$, the $G$-stabilizer of $h$ is not co-amenable in $G$. Since the stabilizer of $h$ is conjugate to $G_{\lk(v)}$ for some vertex $v\in V\Gamma$, and since $G_\Gamma$ splits as an amalgamated free product $G_\Gamma=G_{\st(v)}\ast_{G_{\lk(v)}}G_{\Gamma\setminus\{v\}}$ (where the edge group has infinite index in both vertex groups as a consequence of the irreducibility of $G_\Gamma$), the conclusion follows from \cite[Lemma~2.4]{CE}.
\end{proof}

\section{Two sources of failure of superrigidity}

So far, our results have concerned measure equivalence classification within the class of right-angled Artin groups. Given a right-angled Artin group $G_\Gamma$ and a countable group $H$ which is measure equivalent to $G_\Gamma$, it is natural to ask what can be said about the structure of $H$. In particular, motivated by the strong measure equivalence rigidity of mapping class groups \cite{Kid} and some non-right-angled Artin groups \cite{HH}, one naturally wonders whether $H$ is virtually $G_\Gamma$. However, this turns out to always be false.  

\begin{theo}\label{theo:non-rigidity}
For every right-angled Artin group $G$, there exists a continuum of countable groups that are measure equivalent to $G$ but pairwise non commensurable up to finite kernels.
\end{theo}

We are grateful to a referee for pointing out that our constructions actually yield a continuum of examples, improving the statement we had in an earlier version of this work. In this section, we will present  two sources of examples leading to this phenomenon: graph products of amenable groups, and non-uniform lattices in the automorphism group of the universal cover of the Salvetti complex when $G$ is nonabelian.

\subsection{Graph products of amenable groups}

Given a finite simple graph $\Gamma$ and an assignment of a group $G_v$ to every vertex $v$ of $\Gamma$, we recall that the \emph{graph product} over $\Gamma$ with vertex groups $\{G_v\}_{v\in V\Gamma}$ is the group obtained from the free product of the groups $G_v$ by adding as only extra relations that every element of $G_v$ commutes with every element of $G_w$ whenever $v$ and $w$ are adjacent in $\Gamma$. In particular, a right-angled Artin group is a graph product over its defining graph $\Gamma$, with all vertex groups isomorphic to $\mathbb{Z}$.

\begin{prop}\label{prop:graph-product}
 Let $\Gamma$ be a finite simple graph. Let $G$ and $H$ be two graph products over $\Gamma$, with countable vertex groups $\{G_v\}_{v\in V\Gamma}$ and $\{H_v\}_{v\in V\Gamma}$, respectively. Assume that for every $v\in V\Gamma$, the groups $G_v$ and $H_v$ admit orbit equivalent free measure-preserving actions on standard probability spaces.

Then $G$ and $H$ admit orbit equivalent free measure-preserving actions on standard probability spaces. In particular $G$ and $H$ are measure equivalent.
\end{prop}

Notice that the last sentence of the lemma follows from the previous one by \cite[Theorem~2.3]{Gab2}.

\begin{proof}
The proof is an adaptation of an argument of Gaboriau for free products \cite[$\mathbf{P}_{\mathrm{ME}}\mathbf{6}$]{Gab}, specializing ideas that he introduced for graphings in an earlier work \cite[Part~IV]{Gab-cost}. Arguing by induction on the number of vertices of $\Gamma$, we can assume that there exists a unique vertex $v\in V\Gamma$ such that $G_v$ is not isomorphic to $H_v$. Given any vertex $w\neq v$, we fix once and for all an isomorphism between $G_w$ and $H_w$. We will construct orbit equivalent actions of $G$ and $H$.

Recall that every countable group $G$ has a free measure-preserving ergodic action on a standard probability space, by considering the Bernoulli action $G\curvearrowright\{0,1\}^G$, which preserves the probability measure $\left(\frac{1}{2}\delta_0+\frac{1}{2}\delta_1\right)^{\otimes G}$. Let $Y_v$ be a standard probability space equipped with orbit equivalent free measure-preserving actions of $G_v$ and $H_v$. Up to replacing $Y_v$ by an invariant conull Borel subset, we can assume that for every $x\in Y_v$, one has $G_v\cdot x=H_v\cdot x$. Consider also a free measure-preserving action of $G$ on a standard probability space $Z$. Let $X=Z\times Y_v$ (equipped with the product probability measure). Let $\psi_v:G\to G_v$ be the retraction sending all other factors to $\{1\}$. Then we have a $G$-action on $X$, where the action on $Y_v$ is through $\psi_v$. This action is free and measure-preserving. Let $p:X\to Y_v$ be the projection map, which is $G$-equivariant. 

We now construct an action of $H$ on $X$, in the following way. We first define the action of $H_v$, and we will then extend it to the whole group $H$. Let $x\in X$, and let $h\in H_v$. As the $G_v$-action on $Y_v$ is free and orbit equivalent to the $H_v$-action, there is a unique element $g\in G_v$ such that $gp(x)=hp(x)$, and we then let $hx=gx$. We claim that this gives a well-defined action of $H_v$ on $X$. Indeed, let $x\in X$, and let $h_1,h_2\in H_v$. Then there exists a unique element $g_2\in G_v$ such that $h_2p(x)=g_2p(x)$, and a unique element $g_1\in G_v$ such that $h_1(h_2p(x))=g_1(g_2p(x))$, and $g_1g_2$ is then the unique element of $G_v$ such that $(h_1h_2)p(x)=(g_1g_2)p(x)$. This ensures that $(h_1h_2)x=h_1(h_2x)$. Notice also that this action of $H_v$ on $X$ is free (as the action of $H_v$ on $Y_v$ is free).

We now claim that this action of $H_v$ on $X$ commutes with the original action of $H_{\lk(v)}=G_{\lk(v)}$. Indeed, let $x\in X$, let $h\in H_v$ and let $k\in G_{\lk(v)}$. There is a unique element $g\in G_v$ such that $hp(x)=gp(x)$. As the action of $G_{\lk(v)}$ on $Y_v$ is trivial, we also have $hp(kx)=gp(kx)$, and therefore $hkx=gkx=kgx=khx$, as desired (here the first and third equalities come from the definition of the $H_v$-action, and the second comes from the fact that $g$ and $k$ commute in $G$).

The above claim ensures that the $H_v$-action on $X$ defined above extends to an action of $H$, which coincides with the original action of $G_w=H_w$ on each factor with $w\neq v$. Also, by construction, this $H$-action on $X$ is orbit equivalent to the original $G$-action. In particular it is measure-preserving: given $h\in H$ and a Borel subset $A\subseteq X$, the fact that $A$ and $hA$ have the same measure is proved by decomposing $A$ into countably many mutually disjoint Borel subsets $A_i$ so that on $A_i$, the action of $h$ coincides with the action of a given element $g_i\in G$.

We finally claim that this $H$-action on $X$ is free. Indeed, let $h\in H$ and $x\in X$ be such that $hx=x$. Using the graph product structure, we can write $h=h_1\dots h_k$, where for every $i\in\{1,\dots,k\}$, the element $h_i$ is nontrivial and belongs to one of the vertex groups of the graph product, and if $h_i$ and $h_j$ belong to the same vertex group $H_{ij}$ for some $i\neq j$, then there exists $i<k<j$ such that the associated vertex group $H_k$ does not belong to the star of $H_{ij}$. By construction, there exists an element $g=g_1\dots g_k$ satisfying $gx=x$, and such that for every $i\in\{1,\dots,k\}$, the element $g_i$ is nontrivial, and if $h_i\in H_w$, then $g_i\in G_w$ (in fact $h_i=g_i$ whenever $w\neq v$, through our chosen isomorphism between $H_w$ and $G_w$); here the nontriviality of $g_i$ when $h_i\in H_v$ comes from the freeness of the $H_v$-action on $X$. As the $G$-action on $X$ is free, it follows that $g=\mathrm{id}$. The graph product structure thus ensures that $k=0$ (this is essentially a consequence of the normal form theorem in \cite{green1990graph}, see also \cite[Section 3]{hermiller1995algorithms}), and therefore $h=\mathrm{id}$. This completes our proof.
\end{proof}

As a consequence of Proposition~\ref{prop:graph-product} and a theorem of Ornstein and Weiss \cite{OW} saying that any two ergodic measure-preserving free actions of countably infinite amenable groups are orbit equivalent, we reach the following corollary.

\begin{cor}\label{cor:flexibility}
For every finite simple graph $\Gamma$, any two graph products over $\Gamma$ with countably infinite amenable vertex groups are measure equivalent. 

In particular, given any right-angled Artin group $G$, there exist infinitely many pairwise non-commensurable right-angled Artin groups that are measure equivalent to $G$.
\qed 
\end{cor}

In order to prove Theorem~\ref{theo:non-rigidity}, we are thus left with constructing a continuum of pairwise non-commensurable graph products over $\Gamma$ with amenable vertex groups. The basic idea is to let one of the vertex groups vary, using a continuum of amenable groups.

\begin{proof}[Proof of Theorem~\ref{theo:non-rigidity}]
Let $\Gamma$ be the defining graph of $G$. A theorem of Bartholdi and Erschler \cite{BE} ensures that there exists $\alpha<1$ such that for every $t\in (\alpha,1)$, there exists a finitely generated amenable group $A_t$ whose growth function is equivalent to $n\mapsto e^{n^{t}}$ (for the equivalence relation $\sim$ where $f\sim g$ if $f\preceq g$ and $g\preceq f$, where $f\preceq g$ means that there exists $C>0$ such that for all sufficiently large $n\in\mathbb{N}$, one has $f(n)\le g(Cn)$). Moreover we can assume that $A_t$ is torsion-free \cite[Remark~1.1]{BE}. Let $v\in V\Gamma$ be a vertex, and for every $t\in (\alpha,1)$, let $G_t$ be the graph product over $\Gamma$ such that the vertex group $G_v$ is $A_t$ and all other vertex groups are $\mathbb Z$. Then $G_t$ is torsion-free \cite[Corollary~5.9]{antolin2015tits}. By Corollary~\ref{cor:flexibility}, all groups $G_t$ are measure equivalent, so we are left with proving that they are pairwise not commensurable (hence not commensurable up to finite kernel). 

Given $t\in (\alpha,1)$, we claim that every amenable subgroup $H\subseteq G_t$ is contained in an amenable subgroup $K$ which is either finitely generated and virtually abelian, or equal to the direct product of $A_t$ and a finitely generated virtually abelian subgroup of $G_t$. We prove this by induction on the number of vertices of $\Gamma$. Let $\Gamma=\Gamma_0\circ \Gamma_1\circ\cdots\circ\Gamma_k$ be the de Rham decomposition of $\Gamma$ with $\Gamma_0$ being the clique factor. The claim is trivial if $\Gamma=\Gamma_0$. Now we assume $\Gamma\neq\Gamma_0$. If this decomposition is non-trivial, we know the claim is true for each $G_{\Gamma_i}$ by induction. As $G_{\Gamma}=\oplus_{i=0}^kG_{\Gamma_i}$, the claim follows by considering the projection of $H$ to each $G_{\Gamma_i}$. If the de Rham decomposition has only one factor, then \cite[Theorem~4.1]{antolin2015tits} ensures that either $H$ is virtually cyclic (and we are done), or else $H$ is contained in $G_{\Gamma'}$ for some proper subgraph $\Gamma'$ of $\Gamma$, in which case the claim follows by induction.

Now take $t<t'$ and suppose that $G_{t}$ has a finite index subgroup $G^0_t$ which is isomorphic to a finite index subgroup $G^0_{t'}$ of $G_{t'}$. Let $A^0_{t'}=G^0_{t'}\cap A_{t'}$, whose growth function is also equivalent to $n\mapsto e^{n^{t'}}$. By the previous paragraph, $A^0_{t'}$ can be realized as a subgroup of $H=A_t\times P$ where $P$ is virtually isomorphic to $\mathbb Z^\ell$ for some integer $\ell$. As the growth of $H$ is equivalent to $n^\ell e^{n^{t}}\prec e^{n^{t'}}$ and the map $A^0_{t'}\to H$ is injective and Lipschitz with respect to the word metric, this leads to a contradiction. So $G_t$ is not commensurable to $G_{t'}$.
\end{proof}

\begin{rk}\label{rk:oe-non-conj}
Given a right-angled Artin group $G$, let us explain how the methods used in the present section can be leveraged to build two free, ergodic, measure-preserving actions of $G$ on standard probability spaces which are orbit equivalent but not conjugate.

Let $v\in V\Gamma$ be a vertex in the defining graph of $G$. Let $G\actson Z$ be a Bernoulli action, and let $G_v\actson Y_1$ and $G_v\actson Y_2$ be two Bernoulli actions of the infinite cyclic group $G_v$ whose base spaces have distinct entropies (so these actions are orbit equivalent but not conjugate). Let $\theta$ be the product action of $G$ on $Z\times Y_1$ (equipped with the product probability measure), using the retraction $G\to G_v$ for the action on the second factor. Let $\theta'$ be the action obtained from $\theta$ by replacing $G_v\actson Y_1$ by $G_v\actson Y_2$ as in the proof of Proposition~\ref{prop:graph-product}. More precisely, the orbit equivalence between the $G_v$-actions on $Y_1$ and $Y_2$ yields an isomorphism $\varphi:Y_2\to Y_1$ (which extends to $\varphi:Z\times Y_2\to Z\times Y_1$, extending by the identity on $Z$) and a cocycle $c:G_v\times Y_2\to G_v$. The action $\theta'$ is defined on the standard generators of $G$ as follows: if $g_v$ is the standard generator associated to $v$, then $\theta'(g_v)(x)=\varphi^{-1}(\theta(c(g_v,\pi(x))(\varphi(x)))$, where $\pi:Z\times Y_2\to Y_2$ is the projection; if $g_w$ is a standard generator associated to a vertex $w\neq v$, then $\theta'(g_w)(x)=\varphi^{-1}(\theta(g_w)(\varphi(x)))$. The actions $\theta$ and $\theta'$ are then free and orbit equivalent as in the proof of Proposition~\ref{prop:graph-product}. As $\theta$ is ergodic, we know from the orbit equivalence that $\theta'$ is ergodic.

We claim that these two actions are not conjugate. Indeed, assume by contradiction that $\theta$ and $\theta'$ are conjugate, through a group automorphism $\alpha:G\to G$ and an isomorphism $f:Z\times Y_2\to Z\times Y_1$. Note that if $g\in G$ has trivial projection to $G_v$ (under the natural retraction), then the space of ergodic components for the $\theta$-action $\langle g\rangle \actson Z\times Y_1$ can be identified with $Y_1$. If $g\in G$ has nontrivial projection to $G_v$, then $\langle g\rangle$ acts ergodically on $Z\times Y_1$ (for the action $\theta$).


Take a vertex $w\neq v$. By construction, the action (under $\theta'$) of $G_w$ on $Z\times Y_2$ is not ergodic. As this is conjugate to the action (under $\theta$) of $H=\alpha(G_w)$ on $Z\times Y_1$, the $\theta$-action of $H$ on $Z\times Y_1$ is not ergodic. Hence $H$ has trivial projection to $G_{v}$. By the above description of the spaces of ergodic components of the actions $G_w\actson Z\times Y_2$ (under $\theta'$) and $H\actson Z\times Y_1$ (under $\theta$), we see that there must exist a measure-preserving map $f_1:Z\times Y_2\to Z$ and a measure space isomorphism $f_2:Y_2\to Y_1$ such that for almost every $(z,y)\in Z\times Y_2$, one has $f(z,y)=(f_1(z,y),f_2(y))$. Let $r_v:G\to G_v$ be the retraction sending generators which are not $v$ to identity. As $r_v(\alpha(G_w))=\{1\}$ for every $w\neq v$, and the subgroups $\alpha(G_v)$ together with all other $\alpha(G_w)$ generate $G$, we deduce that $r_v$ induces an isomorphism between $\alpha(G_v)$ and $G_v$. Therefore, the map $f_2$ must conjugate the $G_v$-action on $Y_2$ (under $\theta'$) to the action of $G_v$ (identified with $r_v(\alpha(G_v))$) on $Y_1$. This contradicts the fact that the entropies of the actions $G_v\actson Y_1$ and $G_v\actson Y_2$ are distinct.

As a final remark, let us mention that if we further assume that $\Out(G)$ is finite, then the above argument can be extended to construct a continuum of pairwise nonconjugate actions. A subtle point in the above argument comes from the slight asymmetry in the roles of $\theta$ and $\theta'$: we have a complete description of the cyclic subgroups of $G$ that act ergodically for $\theta$, but not for $\theta'$. When $\Out(G)$ is finite, the group $H$ appearing in the above argument is automatically conjugate to a cyclic group generated by one of the standard generators of $G$. We claim that $H$ acts ergodically on $Z\times Y_2$ if and only if it is conjugate to $G_v$, and otherwise its space of ergodic components is isomorphic to $Y_2$. Indeed, for every vertex $u\in V\Gamma$, the $\theta'$-action of $G_u$ on $Z\times Y_2$ is orbit equivalent (by construction) to the $\theta$-action of $G_u$ on $Z\times Y_1$, so one is ergodic if and only if the other is ergodic. And the action of $G_v$ on $Z\times Y_1$ is ergodic (being a product of two mixing actions), while the action of $G_w$ on $Z\times Y_1$ is not if $w\neq v$, which proves our claim. Arguing as in the previous paragraph and letting $\theta_t$ be the action built from $\theta$ by using a base space with entropy $t$ in place of $Y_2$, we can then deduce that $\theta_t$ and $\theta_{t'}$ are nonconjugate as soon as $t\neq t'$.
\end{rk}

\subsection{Non-uniform lattices acting on Salvetti complexes}

Let $X$ be a locally finite polyhedral complex, and let $\Aut(X)$ be the group of cellular automorphisms of $X$, equipped with the compact-open topology. Then $\Aut(X)$ is a second countable locally compact topological group, see e.g.\ \cite[Example~5.B.11]{cornulier2014metric}. Let $\mu$ be a left invariant Haar measure on $\Aut(X)$. A discrete subgroup $H\subseteq \Aut(X)$ is a \emph{lattice} if the quotient space $H\backslash\Aut(X)$ (with respect to the left translation action of $H$ on $\Aut(X)$) has finite $\mu$-measure. Note that any two lattices in $\Aut(X)$ are measure equivalent (see e.g. \cite[Example~1.2]{Fur}). We will use the following characterization of lattices due to Serre \cite{Serre}, see also \cite[Chapter~1.5]{bass2001tree}.

\begin{theo}[Serre]
	\label{theo:non-uniform lattice}
Let $X$ be a locally finite polyhedral complex with finitely many $\Aut(X)$-orbits of cells. Let $H\subseteq \Aut(X)$ be a discrete subgroup. Let $\mathcal{O}$ be the set of all $H$-orbits of the vertex set of $X$. For each orbit $x\in\mathcal{O}$, let $\alpha_x$ be the cardinality of the $H$-stabilizer of an element in the orbit $x$.
 
Then $H$ is a lattice in $\Aut(X)$ if and only if $\sum_{x\in \mathcal{O}}\frac{1}{\alpha_x}<\infty$.
\end{theo}

We will be interested in the case where $X=\widetilde S_\Gamma$, in which case the above theorem applies (i.e.\ $X$ is a locally finite polyhedral complex with finitely many $\Aut(X)$-orbits of cells).

\begin{prop}\label{prop:lattice}
Let $\Gamma$ be a finite simple graph which is not a complete graph. Then there is a non-uniform lattice $H$ in $\Aut(\widetilde S_\Gamma)$ which is not finitely generated. In particular $H$ is measure equivalent to $G_\Gamma$ but not commensurable to $G_\Gamma$.
\end{prop}

\begin{proof}
	Since $\Gamma$ is not a complete graph, there are two vertices $a,b\in V\Gamma$ which are not adjacent. Let $T\subset \widetilde S_\Gamma$ be a standard subcomplex of type $\{a,b\}$. Then $T$ is a 4-valent tree. Recall that edges of $\widetilde S_\Gamma$ are labeled by vertices of $\Gamma$.
	
	We start with a non-uniform lattice $H$ in $\Aut(T)$ whose action on $T$ preserves the labeling of edges. Such $H$ always exists, and we give a concrete example as follows in terms of the graph of group structure on $H\backslash T$, see Figure~\ref{fig:1}.	Each edge of the graph is labeled by $a$ or $b$. Each vertex group and edge group is a direct sum of copies of $\mathbb Z/2\mathbb Z$, and we choose injective homomorphisms from the edge groups to the incident vertex groups. The number at each vertex (resp.\ edge) represents the order of the vertex group (resp.\ edge group). For example, if a vertex is labeled by $8$, that means the vertex group is a direct sum of  3 copies of $\mathbb Z/2\mathbb Z$. The graph is infinite, though only a finite portion of the graph is drawn, the omitted portion of the graph follows a similar pattern.
	\begin{figure}
		\centering
	\includegraphics[scale=0.7]{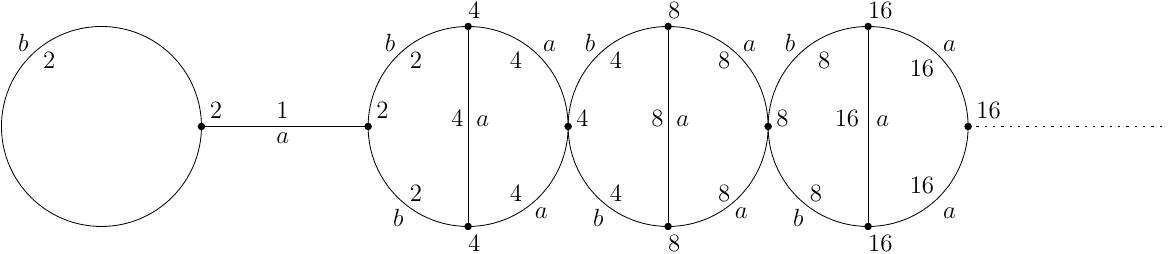}
	\caption{A label-preserving lattice in $\Aut(T)$.}
		\label{fig:1}
    \end{figure}
	One readily checks that the Bass--Serre tree of the graph of groups in Figure~\ref{fig:1} is a 4-valent tree with the standard edge labeling by $a$ and $b$. Moreover, according to Theorem~\ref{theo:non-uniform lattice}, the fundamental group $H$ of this graph of groups is a lattice in $\Aut(T)$. 
	
The next step is to extend the action $H\actson T$ to obtain a lattice in $\Aut(\widetilde S_\Gamma)$. For this purpose, we adjust the argument in \cite[Theorem 9.5]{Hua2}. We first construct an auxiliary complex $X$ as follows. Consider the homomorphism $h:G_\Gamma\to F_{a,b}$ to the free group on $a$ and $b$ defined by sending any generator in $V\Gamma\setminus\{a,b\}$ to the identity element (and sending $a$ to $a$ and $b$ to $b$). Let $X$ be the cover of $S_\Gamma$ corresponding to $\ker h$. We orient each edge of $S_\Gamma$ and equip $X$ with the induced edge labeling and orientation from $S_\Gamma$. Note that the vertex set of $X$ and the vertex set of $T$ can be identified, and $T$ is a subcomplex of $X$. The 1-skeleton $X^{(1)}$ is obtained from $T$ by attaching a collection of circles based at each vertex of $T$ -- one circle for each vertex in  $V\Gamma\setminus\{a,b\}$. The action $H\actson T$ clearly extends to an action $\rho:H\actson X^{(1)}$ such that $\rho$ preserves labels of edges and preserves orientation of edges whose labels are in $V\Gamma\setminus\{a,b\}$. As $\rho$ is label-preserving, it extends to an action $H\actson X$.
	
Let now $H'\subseteq \Aut(\widetilde{S}_\Gamma)$ be the subgroup made of all automorphisms which are lifts of elements of $H$. Letting $\pi:\widetilde{S}_\Gamma\to X$ be the covering map, and $\Aut(\pi)$ the corresponding group of deck transformations, we have a short exact sequence $$1\to\Aut(\pi)\to H'\to H\to 1.$$ We thus have $H'\backslash\widetilde S_\Gamma=H\backslash X$, and in particular there is a 1-1 correspondence between $H'$-orbits of $\widetilde S^{(0)}_\Gamma$ and $H$-orbits of $T^{(0)}$. Moreover, as every element of $H$ fixing a point $x\in X$ has a unique lift fixing a given preimage $\widetilde{x}$ of $x$, the cardinality of vertex stabilizers is preserved under this correspondence. Theorem~\ref{theo:non-uniform lattice} therefore implies that $H'$ is a lattice in $\Aut(\widetilde S_\Gamma)$. As $H$ is a non-uniform lattice acting on a tree, $H$ is not finitely generated, see \cite[Corollary~5.16]{bass2001tree}. Using the above short exact sequence, we deduce that $H'$ is not finitely generated either.
\end{proof}

\begin{rk}
The above proof involves extending the action of a group on a convex subcomplex $Y$ of a $\mathrm{CAT}(0)$ cube complex $X$ to an action of a larger group on $X$. This has been studied in several other contexts, see e.g.\  \cite{haglund2008finite,haglund2008special}.
\end{rk}

\begin{rk}
If one wants to construct further examples of groups that are measure equivalent to a given right-angled Artin group $G$, a natural attempt is to embed $G$ as a lattice in some locally compact second countable topological group $\hat{G}$ and find other lattices in $\hat{G}$. It is therefore natural to try to classify all possible groups $\hat{G}$ in which $G$ embeds as a lattice. Actually, in the case where $\Out(G)$ is finite and $G$ does not split as a product, the compactly generated groups $\hat G$ as above admits a more explicit description as follows. First, by work of Bader, Furman and Sauer \cite[Theorem~1.4 and Theorem~A]{bader2020lattice}, $G$ is cocompact on $\hat G$. Thus $\hat G$ has a quasi-action on $G$. The possible groups $\hat{G}$ are then understood from work of Kleiner and the second named author \cite{KH}: there is a homomorphism with compact kernel and cocompact image from $\hat{G}$ to the automorphism group of a cube complex which is a small ``deformation'' of the universal cover of the Salvetti complex.
\end{rk}

\footnotesize

\bibliographystyle{alpha}
\bibliography{raag-bib}

\newcommand{\etalchar}[1]{$^{#1}$}
\begin{thebibliography}{BCG{\etalchar{+}}09}

\bibitem[AD13]{AD}
C.~Anantharaman-Delaroche.
\newblock The {H}aagerup property for discrete measured groupoids.
\newblock In {\em Operator algebra and dynamics}, volume~58 of {\em Springer
  Proc. Math. Stat.}, pages 1--30, Heidelberg, 2013. Springer.

\bibitem[Ada94]{Ada}
S.~Adams.
\newblock Indecomposability of equivalence relations generated by word
  hyperbolic groups.
\newblock {\em Topol.}, 33(4):785--798, 1994.

\bibitem[AG12]{AG}
A.~Alvarez and D.~Gaboriau.
\newblock Free products, orbit equivalence and measure equivalence rigidity.
\newblock {\em Groups Geom. Dyn.}, 6(1):53--82, 2012.

\bibitem[AM15]{antolin2015tits}
Y.~Antol\'in and A.~Minasyan.
\newblock Tits alternatives for graph products.
\newblock {\em J. Reine Angew. Math.}, 2015(704):55--83, 2015.

\bibitem[Aus16]{Aus}
T.~Austin.
\newblock Integrable measure equivalence for groups of polynomial growth.
\newblock {\em Groups Geom. Dyn.}, 10(1):117--154, 2016.
\newblock With Appendix B by Lewis Bowen.

\bibitem[BC12]{behrstock2012divergence}
J.~Behrstock and R.~Charney.
\newblock Divergence and quasimorphisms of right-angled {A}rtin groups.
\newblock {\em Math. Ann.}, 352(2):339--356, 2012.

\bibitem[BCG{\etalchar{+}}09]{BCGNW}
J.~Brodzki, S.J. Campbell, E.~Guentner, G.A. Niblo, and N.J. Wright.
\newblock Property {A} and $\mathrm{CAT}(0)$ cube complexes.
\newblock {\em J. Funct. Anal.}, 256(5):1408--1431, 2009.

\bibitem[BE04]{BE}
L.~Bartholdi and A.~Erschler.
\newblock Groups of given intermediate word growth.
\newblock {\em Ann. Inst. Fourier (Grenoble)}, 64(5):2003--2036, 2004.

\bibitem[BFS13]{BFS}
U.~Bader, A.~Furman, and R.~Sauer.
\newblock Integrable measure equivalence and rigidity of hyperbolic lattices.
\newblock {\em Invent. Math.}, 194(2):313--379, 2013.

\bibitem[BFS20]{bader2020lattice}
U.~Bader, A.~Furman, and R.~Sauer.
\newblock Lattice envelopes.
\newblock {\em Duke Math. J.}, 169(2):213--278, 2020.

\bibitem[BH99]{bridson2013metric}
M.~Bridson and A.~Haefliger.
\newblock {\em Metric spaces of non-positive curvature}, volume 319 of {\em
  Grundlegren der Mathematischen Wissenschaften}.
\newblock Springer-Verlag, Berlin, 1999.

\bibitem[BJN10]{behrstock2010quasi}
J.A. Behrstock, T.~Januszkiewicz, and W.D. Neumann.
\newblock Quasi-isometric classification of some high dimensional right-angled
  {A}rtin groups.
\newblock {\em Groups Geom. Dyn.}, 4(4):681--692, 2010.

\bibitem[BKS08]{BKS}
M.~Bestvina, B.~Kleiner, and M.~Sageev.
\newblock The asymptotic geometry of right-angled {A}rtin groups. {I}.
\newblock {\em Geom. Topol.}, 12(3):1653--1699, 2008.

\bibitem[BL01]{bass2001tree}
H.~Bass and A.~Lubotzky.
\newblock {\em Tree lattices}, volume 176 of {\em Progress in mathematics}.
\newblock Springer, Birkhäuser Boston, Inc., Boston, MA, 2001.

\bibitem[BN08]{behrstock2008quasi}
J.A. Behrstock and W.D. Neumann.
\newblock Quasi-isometric classification of graph manifold groups.
\newblock {\em Duke Math. J.}, 141(2):217--240, 2008.

\bibitem[CCV07]{charney2007automorphisms}
R.~Charney, J.~Crisp, and K.~Vogtmann.
\newblock Automorphisms of 2--dimensional right-angled {A}rtin groups.
\newblock {\em Geom. Topol.}, 11(4):2227--2264, 2007.

\bibitem[CdlH16]{cornulier2014metric}
Y.~Cornulier and P.~de~la Harpe.
\newblock {\em Metric geometry of locally compact groups}, volume~25 of {\em
  EMS Tracts in Mathematics}.
\newblock European Mathematical Society, Z\"urich, 2016.

\bibitem[CE21]{CE}
I.~Chifan and S.~Kunnawalkam Elayavalli.
\newblock Cartan {S}ubalgebras in von {N}eumann {A}lgebras {A}ssociated to
  {G}raph {P}roduct {G}roups.
\newblock {\em arXiv:2107.04710}, 2021.

\bibitem[CFI16]{CFI}
I.~Chatterji, T.~Fern\'os, and A.~Iozzi.
\newblock The median class and superrigidity of actions on $\mathrm{CAT}(0)$
  cube complexes.
\newblock {\em J. Topol.}, 9(2):349--400, 2016.

\bibitem[Cha07]{charney2007introduction}
R.~Charney.
\newblock An introduction to right-angled {A}rtin groups.
\newblock {\em Geom. Dedic.}, 125(1):141--158, 2007.

\bibitem[CK15]{CK}
I.~Chifan and Y.~Kida.
\newblock {$OE$} and {$W^*$} rigidity results for actions by surface braid
  groups.
\newblock {\em Proc. Lond. Math. Soc. (3)}, 111(6):1431--1470, 2015.

\bibitem[DKR07]{duncan2007parabolic}
A.J. Duncan, I.V. Kazachkov, and V.N. Remeslennikov.
\newblock Parabolic and quasiparabolic subgroups of free partially commutative
  groups.
\newblock {\em J. Algebra}, 318(2):918--932, 2007.

\bibitem[DL03]{DL}
M.W. Davis and I.J. Leary.
\newblock The $\ell^2$-cohomology of {A}rtin groups.
\newblock {\em J. London Math. Soc. (2)}, 68(2):493--510, 2003.

\bibitem[Dro87]{Dro}
C.~Droms.
\newblock Isomorphisms of graph groups.
\newblock {\em Proc. Amer. Math. Soc.}, 100(3):407--408, 1987.

\bibitem[Duc18]{Duc}
B.~Duchesne.
\newblock Groups acting on spaces of non-positive curvature.
\newblock In {\em Handbook of group actions. Vol. III.}, volume~40 of {\em Adv.
  Lect. Math.}, pages 101--141, Int. Press, Sommerville,MA, 2018.

\bibitem[Dye59]{Dye1}
H.A. Dye.
\newblock On groups of measure preserving transformation. {I}.
\newblock {\em Amer. J. Math.}, 81:119--159, 1959.

\bibitem[Dye63]{Dye2}
H.A. Dye.
\newblock On groups of measure preserving transformations. {II}.
\newblock {\em Amer. J. Math.}, 85:551--576, 1963.

\bibitem[Fer18]{Fer}
T.~Fern\'os.
\newblock The {F}urstenberg-{P}oisson boundary and $\mathrm{CAT}(0)$ cube
  complexes.
\newblock {\em Ergodic Theory Dynam. Systems}, 38(6):2180--2223, 2018.

\bibitem[FLM18]{FLM}
T.~Fern\'os, J.~Lécureux, and F.~Mathéus.
\newblock Random {W}alks and {B}oundaries of {CAT}($0$) {C}ubical {C}omplexes.
\newblock {\em Comment. Math. Helv.}, 93(2):291--333, 2018.

\bibitem[FM77]{FM}
J.~Feldman and C.C. Moore.
\newblock Ergodic equivalence relations, cohomology, and von {N}eumann
  algebras.
\newblock {\em Trans. Amer. Math. Soc.}, 234(2):325--359, 1977.

\bibitem[FSZ89]{FSZ}
J.~Feldman, C.E. Sutherland, and R.J. Zimmer.
\newblock Subrelations of ergodic equivalence relations.
\newblock {\em Ergodic Theory Dynam. Systems}, 9(2):239--269, 1989.

\bibitem[Fur99a]{Fur}
A.~Furman.
\newblock Gromov's measure equivalence and rigidity of higher-rank lattices.
\newblock {\em Ann. of Math. (2)}, 150:1059--1081, 1999.

\bibitem[Fur99b]{Fur2}
A.~Furman.
\newblock Orbit equivalence rigidity.
\newblock {\em Ann. of Math. (2)}, 150(3):1083--1108, 1999.

\bibitem[Fur11]{Fur-survey}
A.~Furman.
\newblock A survey of measured group theory.
\newblock In {\em Geometry, rigidity, and group actions}, Chicago Lectures in
  Math., pages 296--374, Chicago, IL, 2011. Univ. Chicago Press.

\bibitem[Gab00]{Gab-cost}
D.~Gaboriau.
\newblock Coût des relations d'équivalence et des groupes.
\newblock {\em Invent. Math.}, 139(1):41--98, 2000.

\bibitem[Gab02a]{Gab-betti}
D.~Gaboriau.
\newblock Invariants $\ell^2$ de relations d'équivalence et de groupes.
\newblock {\em Publ. Math. Inst. Hautes \'Etudes Sci.}, 95:93--150, 2002.

\bibitem[Gab02b]{Gab2}
D.~Gaboriau.
\newblock On orbit equivalence of measure preserving actions.
\newblock In {\em Rigidity in dynamics and geometry (Cambridge, 2000)}, pages
  167--186, Berlin, 2002. Springer.

\bibitem[Gab05]{Gab}
D.~Gaboriau.
\newblock Examples of groups that are measure equivalent to the free group.
\newblock {\em Ergodic Theory Dynam. Systems}, 25(6):1809--1827, 2005.

\bibitem[Gab10]{Gab-survey}
D.~Gaboriau.
\newblock Orbit equivalence and measured group theory.
\newblock In {\em Proceedings of the International Congress of Mathematicians},
  volume III, pages 1501--1527, New Dehli, 2010. Hindustan Book Agency.

\bibitem[Gen20]{genevois2020contracting}
A.~Genevois.
\newblock Contracting isometries of {CAT}(0) cube complexes and acylindrical
  hyperbolicity of diagram groups.
\newblock {\em Algebr. Geom. Topol.}, 20(1):49--134, 2020.

\bibitem[GH21]{GH}
V.~Guirardel and C.~Horbez.
\newblock Measure equivalence rigidity of $\mathrm{Out}({F}_{N})$.
\newblock {\em arXiv:2103.03696}, 2021.

\bibitem[GHL20]{GHL}
V.~Guirardel, C.~Horbez, and J.~Lécureux.
\newblock Cocycle superrigidity from higher-rank lattices to
  $\mathrm{Out}({F}_{N})$.
\newblock {\em arXiv:2005.07477}, 2020.

\bibitem[God03]{God}
E.~Godelle.
\newblock Parabolic subgroups of {A}rtin groups of type {FC}.
\newblock {\em Pacific J. Math.}, 208(3):243--254, 2003.

\bibitem[Gre90]{green1990graph}
E.R. Green.
\newblock {\em Graph products of groups}.
\newblock PhD thesis, University of Leeds, 1990.

\bibitem[Gro93]{Gro}
M.~Gromov.
\newblock Asymptotic invariants of infinite groups.
\newblock In {\em Geometric group theory}, volume~2 of {\em London Math. Soc.
  Lecture Note Ser.}, pages 1--295, Cambridge Univ. Press, Cambridge, 1993.

\bibitem[Hag08]{haglund2008finite}
F.~Haglund.
\newblock Finite index subgroups of graph products.
\newblock {\em Geom. Dedic.}, 135(1):167, 2008.

\bibitem[HH20]{HH}
C.~Horbez and J.~Huang.
\newblock Boundary amenability and measure equivalence rigidity among
  two-dimensional {A}rtin groups of hyperbolic type.
\newblock {\em arXiv:2004.09325}, 2020.

\bibitem[HHL20]{HHL}
C.~Horbez, J.~Huang, and J.~Lécureux.
\newblock Proper proximality in non-positive curvature.
\newblock {\em arXiv:2005.08756}, 2020.

\bibitem[HM95]{hermiller1995algorithms}
S.~Hermiller and J.~Meier.
\newblock Algorithms and geometry for graph products of groups.
\newblock {\em J. Algebra}, 171(1):230--257, 1995.

\bibitem[HR15]{HR}
C.~Houdayer and S.~Raum.
\newblock Baumslag-{S}olitar groups, relative profinite completions and measure
  equivalence rigidity.
\newblock {\em J. Topol.}, 8(1):295--313, 2015.

\bibitem[HS20]{hagen2020hierarchical}
M.F. Hagen and T.~Susse.
\newblock On hierarchical hyperbolicity of cubical groups.
\newblock {\em Israel J. Math.}, 236(1):45--89, 2020.

\bibitem[Hua16]{Hua4}
J.~Huang.
\newblock Quasi-isometry classification of right-angled {A}rtin groups {II}:
  several infinite out cases.
\newblock {\em arXiv:1603.02372}, 2016.

\bibitem[Hua17a]{Hua}
J.~Huang.
\newblock Quasi-isometric classification of right-angled {A}rtin groups {I}:
  the finite out case.
\newblock {\em Geom. Topol.}, 21(6):3467--3537, 2017.

\bibitem[Hua17b]{Hua3}
J.~Huang.
\newblock Top-dimensional quasiflats in {CAT}(0) cube complexes.
\newblock {\em Geom. Topol.}, 21(4):2281--2352, 2017.

\bibitem[Hua18]{Hua2}
J.~Huang.
\newblock Commensurability of groups quasi-isometric to {RAAG}s.
\newblock {\em Invent. Math.}, 213(3):1179--1247, 2018.

\bibitem[HW08]{haglund2008special}
F.~Haglund and D.T. Wise.
\newblock Special cube complexes.
\newblock {\em Geom. Funct. Anal.}, 17(5):1551--1620, 2008.

\bibitem[Kec95]{Kec}
A.S. Kechris.
\newblock {\em Classical descriptive set theory}, volume 156 of {\em Graduate
  Texts in Mathematics}.
\newblock Springer-Verlag, New York, 1995.

\bibitem[KH18]{KH}
B.~Kleiner and J.~Huang.
\newblock Groups quasi-isometric to right-angled {A}rtin groups.
\newblock {\em Duke Math. J.}, 167(3):537--602, 2018.

\bibitem[Kid08]{Kid-memoir}
Y.~Kida.
\newblock The mapping class group from the viewpoint of measure equivalence
  theory.
\newblock {\em Mem. Amer. Math. Soc.}, 196(916), 2008.

\bibitem[Kid09]{Kid-survey}
Y.~Kida.
\newblock Introduction to measurable rigidity of mapping class groups.
\newblock In {\em Handbook of Teichmüller theory, volume II}, volume~13 of
  {\em IRMA Lect. Math. Theor. Phys.}, pages 297--367. Eur. Math. Soc., 2009.

\bibitem[Kid10]{Kid}
Y.~Kida.
\newblock Measure equivalence rigidity of the mapping class group.
\newblock {\em Ann. Math.}, 171(3):1851--1901, 2010.

\bibitem[Kid11]{Kid2}
Y.~Kida.
\newblock Rigidity of amalgamated free products in measure equivalence.
\newblock {\em J. Topol.}, 4(3):687--735, 2011.

\bibitem[Kid14]{kida2014invariants}
Y.~Kida.
\newblock Invariants of orbit equivalence relations and {B}aumslag-{S}olitar
  groups.
\newblock {\em Tohoku Math. J. (2)}, 66(2):205--258, 2014.

\bibitem[KK13]{kim2013embedability}
S.-h. Kim and T.~Koberda.
\newblock Embedability between right-angled {A}rtin groups.
\newblock {\em Geom. Topol.}, 17(1):493--530, 2013.

\bibitem[KK14]{kim2014geometry}
S.-H. Kim and T.~Koberda.
\newblock The geometry of the curve graph of a right-angled {A}rtin group.
\newblock {\em Internat. J. Algebra Comput.}, 24(2):121--169, 2014.

\bibitem[Lau95]{laurence1995generating}
M.R. Laurence.
\newblock A generating set for the automorphism group of a graph group.
\newblock {\em J. London Math. Soc.}, 52(2):318--334, 1995.

\bibitem[Mar20]{Mar}
A.~Margolis.
\newblock Quasi-isometry classification of right-angled {A}rtin groups that
  split over cyclic subgroups.
\newblock {\em Groups Geom. Dyn.}, 14(4):1351--1417, 2020.

\bibitem[Min12]{minasyan2012hereditary}
A.~Minasyan.
\newblock Hereditary conjugacy separability of right-angled {A}rtin groups and
  its applications.
\newblock {\em Groups Geom. Dyn.}, 6(2):335--388, 2012.

\bibitem[MS06]{MS}
N.~Monod and Y.~Shalom.
\newblock Orbit equivalence rigidity and bounded cohomology.
\newblock {\em Ann. of Math. (2)}, 164(3):825--878, 2006.

\bibitem[NS13]{NS}
A.~Nevo and M.~Sageev.
\newblock The {P}oisson boundary of $\mathrm{CAT}(0)$ cube complex groups.
\newblock {\em Groups Geom. Dyn.}, 7(3):653--695, 2013.

\bibitem[OP10a]{OP1}
N.~Ozawa and S.~Popa.
\newblock On a class of $\mathrm{II}_1$ factors with at most one {C}artan
  subalgebra.
\newblock {\em Ann. of Math. (2)}, 172(1):713--749, 2010.

\bibitem[OP10b]{OP}
N.~Ozawa and S.~Popa.
\newblock On a class of $\mathrm{II}_1$ factors with at most one {C}artan
  subalgebra, {II}.
\newblock {\em Amer. J. Math.}, 132(3):841--866, 2010.

\bibitem[OW80]{OW}
D.S. Ornstein and B.~Weiss.
\newblock Ergodic theory of amenable group actions. {I}. {T}he {R}ohlin lemma.
\newblock {\em Bull. Amer. Math. Soc. (N.S.)}, 2(1):161--164, 1980.

\bibitem[PV10]{PV}
S.~Popa and S.~Vaes.
\newblock Group measure space decomposition of ${II}_1$ factors and
  ${W}^*$-superrigidity.
\newblock {\em Invent. Math.}, 182(2):371--417, 2010.

\bibitem[PV14]{PV1}
S.~Popa and S.~Vaes.
\newblock Unique cartan decomposition for ${II}_1$ factors arising from
  arbitrary actions of free groups.
\newblock {\em Acta Mathematica}, 212(1):141--198, 2014.

\bibitem[Rol98]{Rol}
M.A. Roller.
\newblock Poc sets, median algebras and group actions. an extended study of
  {D}unwoody's construction and {S}ageev's theorem.
\newblock 1998.
\newblock preprint.

\bibitem[Sag12]{sageev2012cat}
M.~Sageev.
\newblock $\mathrm{CAT}(0)$ cube complexes and groups.
\newblock In {\em Geometric group theory}, volume~21 of {\em IAS/Park City
  Math. Ser.}, pages 7--54, Providence,RI, 2012. Amer. Math. Soc.

\bibitem[Ser71]{Serre}
J.-P. Serre.
\newblock Cohomologie des groupes discrets.
\newblock In {\em S{\'e}minaire Bourbaki vol. 1970/71 Expos{\'e}s 382--399},
  pages 337--350. Springer, 1971.

\bibitem[Ser89]{Ser}
H.~Servatius.
\newblock Automorphisms of graph groups.
\newblock {\em J. Algebra}, 126(1):34--60, 1989.

\bibitem[Sha05]{Sha-survey}
Y.~Shalom.
\newblock Measurable group theory.
\newblock In {\em European Congress of Mathematics}, pages 391--423, Eur. Math.
  Soc., Zürich, 2005.

\bibitem[Why99]{Why}
K.~Whyte.
\newblock Amenability, bi-{L}ipschitz equivalence, and the von {N}eumann
  conjecture.
\newblock {\em Duke Math. J.}, 99(1):93--112, 1999.

\bibitem[Zim80]{Zim1}
R.J. Zimmer.
\newblock Strong rigidity for ergodic actions of semisimple groups.
\newblock {\em Ann. of Math. (2)}, 112(3):511--529, 1980.

\bibitem[Zim91]{Zim2}
R.J. Zimmer.
\newblock Groups generating transversals to semisimple {L}ie group actions.
\newblock {\em Israel J. Math.}, 73(2):151--159, 1991.

\end{thebibliography}

\begin{flushleft}
Camille Horbez\\ 
Universit\'e Paris-Saclay, CNRS,  Laboratoire de math\'ematiques d'Orsay, 91405, Orsay, France \\
\emph{e-mail:~}\texttt{camille.horbez@universite-paris-saclay.fr}\\[4mm]
\end{flushleft}

\begin{flushleft}
Jingyin Huang\\
Department of Mathematics\\
The Ohio State University, 100 Math Tower\\
231 W 18th Ave, Columbus, OH 43210, U.S.\\
\emph{e-mail:~}\texttt{huang.929@osu.edu}\\
\end{flushleft}

\end{document}